\documentclass[11pt]{article}
\usepackage[tbtags]{amsmath}    
\usepackage{amsfonts,amssymb,amsthm,euscript,makeidx,pifont,boxedminipage,multicol,fullpage}
\usepackage{color}
\usepackage[colorlinks,linkcolor=blue,citecolor=magenta,anchorcolor=magenta]{hyperref}
\usepackage{graphicx}

\def\be{\begin{eqnarray}}
\def\ee{\end{eqnarray}}

\def\b*{\begin{eqnarray*}}
\def\e*{\end{eqnarray*}}

\newtheorem{Theorem}{Theorem}[section]
\newtheorem{Lemma}[Theorem]{Lemma}
\newtheorem{Proposition}[Theorem]{Proposition}

\newtheorem{Definition}[Theorem]{Definition}
\newtheorem{Remark}[Theorem]{Remark}

\newtheorem{Assumption}[Theorem]{Assumption}
\newtheorem{Condition}[Theorem]{Condition}

\def \D{\mathbb{D}}
\def \E{\mathbb{E}}
\def \F{\mathbb{F}}

\def \P{\mathbb{P}}

\def \R{\mathbb{R}}

\def \N{\mathbb{N}}

\def\Bc{{\cal B}}

\def\Dc{{\cal D}}
\def\Ec{{\cal E}}
\def\Fc{{\cal F}}

\def\Hc{{\cal H}}
\def\Ic{{\cal I}}

\def\Lc{{\cal L}}
\def\Mc{{\cal M}}

\def\Pc{{\cal P}}

\def\Sc{{\cal S}}

\def \Om{\Omega}

\def \0{\mathbf{0}}
\def \1{\mathbf{1}}

\def\NA2{\mathrm{NA}_2}
\def\NA{\mathrm{NA}}

\newcommand{\Leb}{\mbox{Leb}}

\title{Supermartingale Brenier's Theorem with full-marginals constraint\thanks{The authors are grateful to Xiaolu Tan for fruitful discussions.}}

\author{
Erhan Bayraktar \thanks{Deptartment of Mathematics, University of Michigan, Ann Arbor, Email: erhan@umich.edu. E.B. is partially supported by the National Science Foundation under grant DMS-2106556 and by the Susan
M. Smith chair.}
\and Shuoqing Deng \thanks{Deptartment of Mathematics, The Hong Kong University of Science and Technology, Clear Water Bay, Hong Kong.
	Email: masdeng@ust.hk. S.D. is partially supported by the Start-up Grant of HKUST.}
\and Dominykas Norgilas \thanks{Deptartment of Mathematics, University of Michigan, Ann Arbor. Email: dnorgila@umich.edu.}
}

\date{}

\begin{document}
\bibliographystyle{plain}

\maketitle

\abstract{

We explicitly construct the supermartingale version of the Fr{\'e}chet-Hoeffding coupling in the setting with infinitely many marginal constraints. This extends the results of Henry-Labord{\`e}re et al. \cite{HLTanTouzi} obtained in the martingale setting. Our construction is based on the Markovian iteration of one-period optimal supermartingale couplings. In the limit, as the number of iterations goes to infinity, we obtain a pure jump process that belongs to a family of local L{\'e}vy models introduced by Carr et al. \cite{CarrGemanMadanYor}. We show that the constructed processes solve the continuous-time supermartingale optimal transport problem for a particular family of (path-dependent) cost functions. The explicit computations are provided in the following three cases: the uniform case, the Bachelier model and the Geometric Brownian Motion case.

%
%
}

\vspace{2mm}


\vspace{2mm}


\tableofcontents




\section{Introduction}

The classical optimal transport (OT) problem, first introduced by Monge, and then relaxed by Kantorovich, has the following formulation: given two probability measures $\mu_0$, $\mu_1$ on $\R$ and a cost (or payoff) function $c:\R \times \R \to \R$, the goal is to minimize (or maximize) the value $\E^{\P}[c(X_0, X_1)]$ among all probability measures $\P$ such that $\P \circ X_0^{-1}=\mu_0$ and $\P \circ X_1^{-1}= \mu_1$. Under the Spence-Mirrlees condition $c_{xy}>0$, the maximizing transport plan is characterized by the Brenier's Theorem (see Brenier \cite{brenier1991polar} and Rachev and R\"uschendorf \cite{rachev1998mass}) and corresponds to the Fr{\'e}chet-Hoeffding (or quantile) coupling $\pi^{FH}$. 

Motivated by the applications in financial mathematics, the Martingale Optimal Transport (MOT) problem was introduced by Beiglb{\"o}ck et al. \cite{BHLP13} (in the disrete-time setting) and Galichon et al. \cite{GHLT} (in the the continuous-time setting), and has been widely studied since then. Given two probability measures $\mu_0, \mu_1$ which are increasing in convex order, the problem consists in minimizing (or maximizing) $\E^{\P}[c(X_0, X_1)]$ among all martingale measures $\P$ such that $\P \circ X_0^{-1}=\mu_0$ and $\P \circ X_1^{-1}= \mu_1$. Under the so-called martingale Spence-Mirrlees condition $c_{xyy}>0$, the maximizer corresponds to the so-called left-curtain martingale coupling $\pi^{lc}$, introduced by Beiglb\"ock and Juillet \cite{BeiglbockJuillet:16}. $\pi^{lc}$ was explicitly constructed by Henry-Labord{\`e}re and Touzi \cite{HLTouzi16} using ODE arguments (and under some technical assumptions). The general construction was later obtained by Hobson and Norgilas \cite{HobsonNorgilas:21} using geometric arguments and properties of the potential functions of marginals.

The financial motivation of MOT comes from the robust sub/super-replication of financial derivatives in a market where the underlying asset and the corresponding vanilla options with certain maturities are available for trading. From the well-known Breenden-Litzenberger formula (see \cite{breeden1978prices}), the marginal distributions of the underlying asset can be then obtained using the prices of vanilla options. Each sensible pricing model should be calibrated to the given market data, and thus should produce the same marginal distributions at fixed times. The search of a model that produces the highest no-arbitrage price of an exotic claim, among all calibrated models, then naturally corresponds to the MOT problem. For more recent developments of MOT problems, see for example Acciaio et al. \cite{acciaio2021weak}, Backhoff-Veraguas et al. \cite{backhoff2020martingale,backhoff2022stability}, Beiglb\"ock et al. \cite{BCH17, BNT17,BCH19, beiglbock2017monotone, beiglbock2022potential, beiglbock2021approximation}, Beiglb\"ock and Juillet \cite{BeiglbockJuillet:16,beiglbock2021shadow}, Br\"uckerhoff et al. \cite{BHJ.20}, Campi et al. \cite{campi2017change}, De Marco and Henry-Labord\`ere \cite{de2015linking}, Dolinsky and Soner \cite{dolinsky2014martingale}, Fahim nad Huang \cite{fahim2016model}, Gaoyue et al. \cite{guo2016optimal}, Hobson and Neuberger \cite{HobsonNeuberger}, Hobson and Klimmek \cite{HobsonKlimmek}, Hobson and Norgilas \cite{hobson2019robust}, Nutz et al. \cite{NutzStebeggTan.17}, Wiesel \cite{wiesel2019continuity}.

The problem of finding the models that give robust no-arbitrage bounds for the prices of exotic derivatives was initially studied in the seminal work of Hobson \cite{Hobson} by means of the Skorokhod Embedding Problem (SEP). SEP consists in finding a stopping time $\tau$ of a Brownian motion $B$, such that $B_{\tau}$ has a prescribed law. This approach generated developments in many (probabilistic and financially motivated) directions, see, for example, Brown et. al \cite{brown2001robust}, Madan and Yor \cite{madan2002making}, Cox et al. \cite{cox2008pathwise,cox2019root}, Cox and Ob{\l}{\'o}j \cite{cox2011robust,cox2015joint}, Davis et al. \cite{davis2014arbitrage}. See also the survey papers on SEP by Ob{\l}{\'o}j \cite{obloj2004skorokhod} and Hobson \cite{Hobson2}.

More recently, the Supermartingale Optimal Transport(SOT) problem was introduced by Nutz and Stebegg \cite{NutzStebegg.18}. Given two probability measures $\mu_0, \mu_1$ which are increasing in convex-decreasing order, the problem consists in minimizing (or maximizing) $\E^{\P}[c(X_0, X_1)]$ among all supermartingale measures $\P$ such that $\P \circ X_0^{-1}=\mu_0$ and $\P \circ X_1^{-1}= \mu_1$. In particular, the authors of \cite{NutzStebegg.18} introduced two canonical supermartingale couplings, namely the increasing supermartingale coupling $\pi^I$ and the decreasing supermartingale coupling $\pi^D$. These transport plans are canonical in the sense that they can be equivalently characterised by any, and then all of the following properties: the optimality (in terms of minimization or maximization problem) for a large class of cost functions, the monotonicity of the support, and the order-theoretic minimality. (The optimal couplings in the OT and MOT settings have similar characterizations.)

   In this paper, we shall focus on the continuous-time case, which corresponds to a suitable limit of the multi-marginal SOT problem. Given  a continuous family of marginal distributions $(\mu_t)_{t\in[0,1]}$, which is non-decreasing in convex-decreasing order, we call a stochastic process that is calibrated to all marginals $(\mu_t)_{t\in[0,1]}$ a \textit{PCOCD}(``Processus Croissant pour l'Ordre Convexe D{\'e}croissant" in French). Notice that in the martingale case such process is baptized as a peacock (\textit{PCOC} for short) and is extensively studied in the book of Hirch et al. \cite{HPRY}. 
            
We will closely follow Henri-Labord{\`e}re et al. \cite{HLTanTouzi}: in the martingale setting, by considering a Markovian iteration of the one-period left-curtain martingale couplings $\pi^{lc}$, the authors obtained a pure jump process in the spirit of the local L{\'e}vy models introduced by Carr et al. \cite{CarrGemanMadanYor} and showed that it solves a particular continuous-time MOT problem. Our main goal is to extend the results of Henri-Labord{\`e}re et al. \cite{HLTanTouzi} to the supermartingale setting.
  
In the supermartingale case, the construction is different in the following aspects. First, in the martingale case, the left-curtain and the right-curtain transport plans are symmetric, see Henry-Labord\`re et al. \cite[Remark 3.12]{HLTanTouzi}. More precisely, the limiting process, associated to the left-curtain coupling, is driven by downward jumps and upward drift, while in the right-curtain case it is the upward jumps and downward drift that drive the limiting process.  On the other hand, in the supermartingale case, the increasing supermartingale transport plan and the decreasing supermartingale transport plan demonstrate different behaviour see Bayraktar et al. \cite{BayDengNorgilas}, \cite{BayDengNorgilas2}. In the increasing case, the limiting process is a martingale in the interior of the support of marginals, and possesses strict supermartingale characteristics when it escapes to the upper boundary of the support (it jumps from the upper all the way down to the lower boundary). In the case of the decreasing supermartingale coupling, there are martingale and supermartingale regions in the interior of the support. Once the boundary curve (that separates these two regions) is hit, the process becomes deterministic and follows a decreasing (w.r.t. time) curve.

Second, in case of the the decreasing supermartingale coupling, there might exist multiple \textit{phase transition curves} for the limiting process, that divide the martingale and strict supermartingale regions. To obtain explicit solutions, in this article we will impose certain conditions that will guarantee the uniqueness of the phase transition curve. On the other hand, if the increasing supermartingale coupling is used for construction, then the regime switching boundary is unique, and is always given by the upper boundary of the support of the marginal distributions.

In financial terms, the dual SOT problem corresponds to the super-replication of an exotic payoff using semi-static strategies with no short-selling constraints. 
In order to prove the optimality of the limiting process, we will modify the arguments of Henry-Labord\`ere et al. \cite{HLTouzi16} and explicitly construct a candidate optimal dual strategy, for a large (but very particular) class of cost functions. The choice of dual variables is closely related to the martingale and strict supermartingale regions of a a candidate optimal transport plan. In the case of the decreasing supermartingale coupling, the dual methods are related to those of the dual MOT problem when the right-curtain martingale coupling is optimal, and also the classical OT duality when the quantile coupling is an optimizer. The strategies must be modified accordingly, if the construction of the limiting process is based on the increasing supermartingale coupling.      
     
The rest of the paper is organized as follows. In Section \ref{sec:SMOT}, we recall the formulation of the discrete-time supermartingale transport problem and introduce the continuous-time SOT problem. In Section \ref{sec:main_results}, we provide the convergence theorem which gives the characterization of the limiting process, and the optimality theorem which shows that the limiting process solves a continuous-time SOT problem for a class of costs functions. In Section \ref{sec:examples}, we explicitly obtain the limiting process when the marginals correspond to uniform measures on bounded support, Normal distributions with decreasing means (Brownian motion with drift) and log-Normal distributions (geometric Brownian motion). In Appendix \ref{sec:Brenier} and \ref{sec:Decreasing_SMOT_one_period} we collect some important results regarding one-period OT, MOT and SOT problems.
	

\section{Supermartingale optimal transport problem}	\label{sec:SMOT}

\subsection{Discrete-time supermartigale optimal transport}\label{subsec:discreteSOT}
Let $\Mc$ (resp. $\Pc$) be the set of finite (resp. probability) Borel measures on $\R$ with finite first moments. The support of $\eta\in\Mc$, denoted by supp$(\eta)$, is the smallest closed set $I\subseteq\R$ with $\eta(I)=\eta(\R)$. We set $\ell_\eta:=\inf\{k\in\textrm{supp}(\eta)\}\in[-\infty,\infty)$ and $r_\eta:=\sup\{k\in\textrm{supp}(\eta)\}\in(-\infty,\infty]$. The mean of $\eta\in\Mc$ is denoted by $\overline\eta$, so that the barycenter of $\eta$ is given by $\overline\eta/\eta(\R)$.

Throughout the paper, for any random variable $\xi$, the expectation of $\xi$ is defined as $\E[\xi]:=\E[\xi^+] - \E[\xi^-]$ with the convention $\infty-\infty=-\infty$.

The canonical process on $\R^2$ is denoted by $(X_0,X_1)$, so that, for $i=0,1$, $X_i(x_0, x_1)=x_i$ for all $(x_0,x_1) \in \R^2$. For $\mu_0,\mu_1\in\Pc$, let $\Pc(\mu_0,\mu_1)$ be the set of (Borel) probability measures on $\R^2$ with first and second marginals $\mu_0$ and $\mu_1$, respectively. (Each $\P\in\Pc(\mu_0,\mu_1)$ is often called a transport plan or coupling of $\mu_0$ and $\mu_1$.) Then $X_0 \sim^{\P} \mu_0, ~ X_1 \sim^{\P} \mu_1$ for all $\P\in\Pc(\mu_0,\mu_1)$. We also introduce the set of supermartingale couplings of $\mu_0$ and $\mu_1$:
$$
\Sc_2(\mu_0, \mu_1) := \{ \P \in \Pc(\mu_0,\mu_1):\E^{\P} [ X_1| X_0 ] \leq X_0~\P\textrm{-a.s.}\}.
$$

By the classical result of Strassen \cite{strassen1965existence}, $\Sc_2(\mu_0,\mu_1)$ is non-empty if and only if $\mu_0\leq_{cd}\mu_1$, i.e., $\mu_0$ is smaller than $\mu_1$ in convex-decreasing order. Recall that $\mu_0\leq_{cd}\mu_1$ if $\mu_0(\phi)\leq\mu_1(\phi)$ for all convex and non-increasing $\phi:\R\to\R$, where for a measurable $f:\R\to\R$ and $\eta\in\Mc$ we write $\eta(f)=\int_\R f d\eta$. Note that $\overline\mu_0\leq\overline\mu_1$ whenever $\mu_0\leq_{cd}\mu_1$. On the other hand, if $\mu_0\leq_{cd}\mu_1$ and $\overline\mu_0=\overline{\mu}_1$, then $\Sc_2(\mu_0, \mu_1)$ reduces to the set of martingale couplings of $\mu_0$ and $\mu_1$. Indeed, every supermartingale with constant mean is a martingale.

For a (Borel) measurable reward (or cost) function $c: \R^2 \to \R$, a two-marginal supermartingale transport problem is defined by 
\begin{equation} \label{eq:SMOT_onePrimal}
\mathbf{P}_2(\mu_0, \mu_1) : = \sup_{\P \in \Sc_2(\mu_0, \mu_1)} \E^{\P} \left[ c(X_0, X_1) \right].
\end{equation}

Suppose that there exists $a_0,a_1:\R\to\R$ that are $\mu_0$ and $\mu_1$-integrable, respectively, and such that $\lvert c(x,y)\lvert  \leq a_0(x) + a_1(y)$, $x,y\in\R$. Then in the case $c$ satisfies the following \textit{supermartingale Spence-Mirrlees} condition
$$
c(x',\cdot)-c(x,\cdot)\textrm{ is decreasing and convex for all }x<x',
$$
Nutz and Stebegg \cite{NutzStebegg.18} proved that the supremum in \eqref{eq:SMOT_onePrimal} is attained by the so-called \textit{increasing} supermartingale coupling $\pi^I$. Note that, in the case $c$ is smooth, supermartingale Spence-Mirrlees condition can be equivalently stated in terms of cross-derivatives: $c_{xy}\leq0$ and $c_{xyy}\geq0$. On the other hand, if $-c$ is supermartingale Spence-Mirrlees, then the optimizer is given by the \textit{decreasing} supermartingale coupling $\pi^D$. The main focus of this paper is on $\pi^D$ (see Appendix \ref{sec:Decreasing_SMOT_one_period_primal}). 

We now describe the dual problem that is associated to \eqref{eq:SMOT_onePrimal}. We write $(\phi,\psi,h)\in\Dc_2(\mu_0,\mu_1)$ if $\phi,\psi,h:\R\to\R$ are (Borel) measurable functions such that $(\int_\R(\phi\vee0)d\mu_0)\vee(\int_\R(\psi\vee0)d\mu_1)<\infty$, $h\geq 0$ and $$c(x,y)\leq\phi(x)+\psi(y)+h(x)(y-x),\quad x,y\in\R.$$  Then the weak duality holds
\begin{equation}  \label{eq:SMOT_oneDual}
\mathbf{P}_2(\mu_0, \mu_1)\leq\mathbf{D}_2(\mu_0, \mu_1) : = \inf_{(\varphi, \psi, h) \in \Dc_2(\mu_0,\mu_1)} \{ \mu_0(\varphi) + \mu_1(\psi) \}.
\end{equation}
Nutz and Stebegg \cite{NutzStebegg.18} showed that the strong duality (i.e., equality in \eqref{eq:SMOT_oneDual}) holds under some additional assumptions (for example, if $\mu_0\leq_{cd}\mu_1$ are \textit{irreducible}). In general (and without the irreducibility condition), the strong duality and the existence of dual optimizers can still be obtained, provided one agrees to switch from a point-wise to a quasi-sure formulation of the dual problem. An explicit construction of optimal dual variables is (under some further assumptions) provided in Appendix \ref{sec:SMOTdual}.

The problem \eqref{eq:SMOT_onePrimal} can be easily extended to a $T$-period (or $(T+1)$-marginal) setting if we restrict to cost functions $c:\R^{T+1}\to\R$ of the form $c(x_0,\ldots,x_{T})=\sum_{i=1}^{T} c_i(x_{i-1}, x_{i})$, $(x_0,\ldots,x_{T})\in\R^{T+1}$, for a collection of one-period (Borel) cost functions $c_i:\R\to\R$, $i=1,...,T$. More precisely, we are given a sequence of probability measures $\mu_0, \ldots, \mu_T\in\Pc$ satisfying $\mu_0\leq_{cd}\ldots\leq_{cd} \mu_T $, and the goal is to maximize
\begin{equation}\label{eq:multiSOT}
\E^\P[c(X_0,\ldots,X_T)]=\sum^{T}_{i=1}\E^\P[c_i(X_{i-1},X_{i})]
\end{equation}
over all (Borel) probability measures $\P$ on $\R^{T+1}$ such that $X_i\sim^\P\mu_i$, $i=0,...,T$, and satisfying $\E^\P[X_i\lvert X_0,\ldots,X_{i-1}]\leq X_{i-1}$ $\P$-a.s., for all $i=1,...,T$. (For the martingale version with general cost functions $c:\R^{T+1}\to\R$, see Nutz et al. \cite{NutzStebeggTan.17}, while a corresponding continuous-time extension is given by Juillet et al. \cite{BHJ.20}.) Then, if each $c_i$ (or $-c_i$) is supermartingale Spence-Mirrlees, the optimal supermartingale coupling (with fixed $T+1$ marginals) can be obtained by a Markovian iteration of the increasing (or decreasing) one-period supermartingale transport plans.

\subsection{Continuous-time supermartigale optimal transport}\label{sec:cont_smot}

In this section we introduce a continuous-time supermartingale optimal transport problem under full marginals constraint, as the limit of the multi-period supermartingale optimal transport introduced at the end of Section \ref{subsec:discreteSOT}. 

Let $\Om:= \D([0,1],\R)$ be the canonical space of all c{\`a}dl{\`a}g paths on $[0,1]$, $X$ the canonical process and $\F=(\Fc_t)_{0 \leq t \leq 1}$ the canonical filtration generated by $X$, i.e. $\Fc_t:=\sigma(\{ X_s, 0 \leq s \leq t \})$. We denote by $\Sc_{\infty}$ the collection of all supermartingale measures on $\Om$, i.e., under each $\P\in\Sc_\infty$ the canonical process $X$ is a supermartingale. We equip $\Sc_{\infty}$ with the weak convergence topology throughout the paper. By the calssical results (see, for example, Karandikar \cite{Karandikar}), there is a non-decreasing process $([X]_t)_{t \in [0,1]}$, defined on $\Om$, and which coincides with the $\P$-quadratic variation of $X$, $\P$-\mbox{a.s.} for every supermartingale measure $\P \in \Sc_{\infty}$. Denote by $[X]^c_.$ the continuous part of $[X]_.$.  

Given a family of (integrable) probability measures $\mu=(\mu_t)_{0 \leq t \leq 1}$, we denote by $\Sc_{\infty}(\mu) \subset \Sc_{\infty}$ the collection of all supermartingale measures on $\Om$ such that $X_t \sim^{\P} \mu_t$ for all $t \in [0,1]$. In particular, from Ewald and Yor \cite{ewald2018peacocks}, we know that $\Sc_{\infty}(\mu)$ is non-empty if and only if the family $(\mu_t)_{0 \leq t \leq 1}$ is non-decreasing in convex-decreasing order and $t \mapsto \mu_t$ is right-continuous.

For all $t \in [0,1]$, we denote by $-\infty \leq \ell_t:=\ell_{\mu_t} \leq r_{\mu_t}=:r_t \leq \infty$ the left and right extreme boundaries of supp$(\mu_t)$. Suppose that $\mu_s\leq_{cd}\mu_t$ for all $0\leq s\leq t\leq 1$. Then since $x\mapsto p_k(x):=(k-x)^+$ is convex and non-increasing for all $k\in\R$, we have that $\mu_s(p_k)\leq\mu_t(p_k)$ for $k\in\R$. In particular, $\mu_s((-\infty,k))>0$ implies that $\mu_t((-\infty,k))>0$, from which it follows that $t\mapsto \ell_t$ is non-increasing. Note that, in general, $t\mapsto r_t$ may fail to be non-decreasing. In addition, if $\mu_t$ admits a density function for all $t \in [0,1]$ and $t \mapsto \mu_t$ is continuous w.r.t. the weak convergence topology, then $t\mapsto l_t$ and $t\mapsto r_t$ are continuous (see Henry-Labord\`ere and Touzi \cite[p. 2805]{HLTouzi16}). 

As in Henry-Labord\`{e}re et al. \cite{HLTanTouzi} (see also Hobson and Klimmek \cite{HobsonKlimmek}), our continuous-time SOT problem arises as a limit of the multi-period SOT problems, by considering the limit of the reward function $\sum_{i=1}^{T} c_i(x_{t_{i-1}}, x_{t_{i}})$ as in \eqref{eq:multiSOT}, where $(t_i)_{0\leq i \leq T}$ is a partition of $[0,1]$ with $\max_{1\leq i \leq T}\lvert t_i-t_{i-1}\lvert\to0$ as $T\to\infty$. To obtain the convergence we need the pathwise It\^{o} calculus. 

\begin{Definition}[F{\"o}llmer \cite{Follmer}] Let $\pi_n=(0=t_0^n < \cdots < t_n^n=1)$, $n \geq 1$ be partitions of $[0,1]$ with $|\pi_n|:=\max_{1 \leq k \leq n} | t_k^n - t_{k-1}^n | \to 0$ as $n \to \infty$.
	A c{\`a}dl{\`a}g path $\mathbf{x}: [0,1] \to \R$ has a finite quadratic variation along $(\pi_n)_{n \geq 1}$ if the sequence of measures on $[0,1]$,
	$$
	\sum_{1 \leq k \leq n} ( \mathbf{x}_{t_k^n} - \mathbf{x}_{t_{k-1}^n} )^2 \delta_{\{ t_{k-1} \}} (dt),
	$$
	converges weakly to a measure $[\mathbf{x}]^F$ on $[0,1]$. For each $t\in[0,1]$, set $[\mathbf{x}]_t^F := [\mathbf{x}]^F([0,t])$, and let $[\mathbf{x}]^{F,c}_.([0,t])$ be the continuous part of this non-decreasing path.
\end{Definition}

The following will be one of our standing assumptions.

\begin{Assumption} \label{Assump:CostFunction}
The cost function $c: \R^2 \to \R$ is in $C^3((l_1, \max_{0 \leq t \leq 1} r_t) \times (l_1, \max_{0 \leq t \leq 1} r_t))$ and satisfies
$$c(x,x)=c_y(x,x)=0,~ c_{xyy}<0,~ c_{xy}>0, {\quad(x,y) \in (l_1, \max_{0 \leq t \leq 1} r_t) \times (l_1, \max_{0 \leq t \leq 1} r_t).}$$
\end{Assumption}

Following Hobson and Klimmek \cite{HobsonKlimmek} we then have (in fact, a weaker assumption on $c$ would suffice; see Touzi et al. \cite{HLTouzi16})

\begin{Lemma} \label{Lemma:Quadratic_Variation}
Let Assumption \ref{Assump:CostFunction} hold ture. Then for every path $\mathbf{x} \in \Om$ with finite quadratic variation $[\mathbf{x}]^F$ along a sequence of partitions $(\pi_n)_{n \geq 1}$, we have
$$
\sum_{k=0}^{n-1} c( \mathbf{x}_{t_k^n}, \mathbf{x}_{t_{k-1}^n} ) \to \frac{1}{2} \int_0^1 c_{yy}( \mathbf{x}_t, \mathbf{x}_t ) d [\mathbf{x}]^{F,c}_t + \sum_{0 \leq t \leq 1} c(\mathbf{x}_{t-}, \mathbf{x}_t ).
$$
\end{Lemma}

The non-decreasing process $[X]_\cdot$, as in Karandikar \cite{Karandikar}, is defined for every c\`adl\`ag path $\mathbf{x}$ and coincides $\P$-a.s. with both, the quadratic variation and $[\mathbf{x}]^F$, for every $\P\in\Sc_\infty$. This and Lemma \ref{Lemma:Quadratic_Variation} motivates us to define the following continuous-time reward function:
$$
C(\mathbf{x}): = \frac{1}{2} \int_0^1 c_{yy}( \mathbf{x}_t, \mathbf{x}_t ) d [\mathbf{x}]^{c}_t + \sum_{0 \leq t \leq 1} c(\mathbf{x}_{t-}, \mathbf{x}_t ), ~~~\mbox{for all } \mathbf{x} \in \Om,
$$
where the integral and the sum are defined using positive and negative parts with a convention $\infty-\infty=-\infty$.

The continuous-time supermartingale optimal transport problem is now given by
\begin{equation} \label{eq:Primal_Full}
\mathbf{P}_{\infty}(\mu) : = \sup_{\P \in \Sc_{\infty}(\mu)} \E^{\P} \left[ C(X) \right].
\end{equation}
Note that $\sum_{0\leq t\leq 1}c(X_{t-},X_t)<\infty$ $\P$-a.s. for all $\P\in\Sc_\infty$; see Henry-Labord\`er\`e et al. \cite[Remark 2.5]{HLTanTouzi}.

We now follow Henry-Labord\`er\`e et al. \cite{HLTanTouzi} and introduce the dual formulation of the problem \eqref{eq:Primal_Full}. First we introduce the set of admissible semi-static trading strategies. Let $\mathbb{H}_0$ be the class of $\F$-predictable and locally bounded processes $H: [0,1] \times \Om \to \R$, i.e., there exists an increasing family of $\F$-stopping times $(\tau_n)_{n\geq1}$ taking values in $[0,1]\cup\{\infty\}$ such that $H_{\cdot\wedge\tau_n}$ is bounded for all $n\geq1$ and $\tau_n\to\infty$ as $n\to\infty$. Then for all $H \in \mathbb{H}_0$ and under every supermartingale measure $\P \in \Sc_{\infty}$, one can define the integral, denoted by $H \cdot X$, see Jacod and Shiryaev \cite[Chapter I.4]{JacodShiryaev}. Define
$$
\Hc:= \{ H \in \mathbb{H}_0: H \geq 0 \mbox{ and } H\cdot X \mbox{ is a } \P\mbox{-supermartingale for every } \P \in \Sc_{\infty} \}.
$$
For the static strategy, we denote by $M([0,1])$ the space of all finite signed measures on $[0,1]$ which is a Polish space under the weak convergence topology, and by $\Lambda$ the class of all measurable maps $\lambda: \R \to M([0,1])$ which admit a representation $\lambda(x,dt)=\lambda_0(t,x) \gamma (dt)$ for some finite non-negative measure $\gamma$ on $[0,1]$ and some measurable function $\lambda_0:[0,1]\times \R \to \R$ which is bounded on $[0,1] \times K$ for all compact $K$ of $\R$. We then denote
$$
\Lambda(\mu):=\{ \lambda \in \Lambda :\mu(\lvert\lambda\lvert)<\infty \}, ~~\mbox{where } \mu(|\lambda|):=\int_0^1 \int_\R |\lambda_0(t,x)| \mu_t(dx) \gamma(dt).
$$
We also introduce a family of random measures $\delta^X=(\delta_t^X)_{0 \leq t \leq 1}$, associated to the canonical process $X$, by setting $\delta_t^X(dx):=\delta_{X_t}(dx)$. In particular, one has
$$
\delta^X(\lambda) = \int_0^1 \lambda(X_t, dt) = \int_0^1 \lambda_0(t,X_t)\gamma(dt).
$$

The set of admissible superhedging strategies with no short-selling constraint is defined by
$$
\Dc_{\infty}:= \{ (H, \lambda) \in \Hc \times \Lambda(\mu): \delta^X(\lambda) + (H \cdot X)_1 \geq C(X_\cdot), \P\mbox{-a.s.}, \forall \P \in \Sc_{\infty} \},
$$
and then the dual problem is given by 
\begin{equation} \label{eq:Dual_Full}
\mathbf{D}_{\infty}(\mu) : = \inf_{(H, \lambda) \in \Dc_{\infty}(\mu)} \mu(\lambda).
\end{equation}
Furthermore, the weak duality holds: $\mathbf{P}_{\infty}(\mu)\leq\mathbf{D}_{\infty}(\mu)$.
\section{Main results}\label{sec:main_results}
Our contribution splits into three parts. First we construct a continuous-time supermartingale with given marginals $(\mu_t)_{t\in[0,1]}$. This is achieved as an accumulation point of a sequence $(\P^n)_{n}$ of solutions of $n$-period supermartingale transport problems. We then characterize the limiting law further and show that it corresponds to the distribution of a local L\'evy process. Finally we prove that this limit solves the continuous-time supermartingale optimal transport problem \eqref{eq:Primal_Full}. The last point is achieved by proving the strong duality $\mathbf{P}_{\infty}(\mu)=\mathbf{D}_{\infty}(\mu)$, and explicitly constructing an optimal superhedging strategy.

\subsection{Convergence of the sequence $(\P^n)_{n \geq 1}$}
For every $t \in [0,1]$, we denote by $\R\ni x\mapsto F(t,x)$ and $[0,1]\ni u\mapsto F^{-1}(t,u)$ the cumulative distribution function and the corresponding right-continuous inverse (w.r.t. the $x$-variable) of the probability measure $\mu_t$. For $t\in[0,1)$, $\epsilon\in(0,1-t]$ and $x,y\in\R$ let
\begin{align*}
\delta^\epsilon F(t,x)&:=F(t+\epsilon,x)-F(t,x)
\end{align*}
and define
$$
E:=\{ (t,x): t \in [0,1], x \in (l_t, r_t) \}.
$$
\begin{Assumption} \label{Assump:1}
	\begin{enumerate} 
		\item[\rm{(i)}]  The marginal distributions $\mu=(\mu_t)_{t \in [0,1]}$ belong to $\Pc$, are non-decreasing in convex-decreasing and $t \mapsto \mu_t$ is continuous w.r.t. the weak convergence topology.

		\item[\rm{(ii)}]  $F \in C_b^4(E)$. In particular, for all $t \in [0,1]$, the measure $\mu_t$ has a density function $f(t,\cdot)$ such that $f(t,\cdot)>0$ on $(\ell_t,r_t)$.
		
		\item[\rm{(iii)}] There exists $\epsilon_1\in(0,1]$ such that for all $t \in [0,1)$ and $0<\epsilon\leq \epsilon_1\wedge(1-t)$, $\mu_t$ and $\mu_{t+\epsilon}$ satisfy the Dispersion Assumption:
		\begin{itemize}
			\item There exists ${m}_{\epsilon}(t)<{m}^{\epsilon}(t)$ such that  $\ell_{t+\epsilon}\leq\ell_t\leq{m}_{\epsilon}(t)<{m}^{\epsilon}(t)\leq r_t\leq r_{t+\epsilon}$, with $\ell_{t+\epsilon}<\ell_t\wedge{m}_{\epsilon}(t)$ and ${m}^{\epsilon}(t)\vee r_t< r_{t+\epsilon}$.
		\item Furthermore, $f(t+\epsilon,\cdot)<f(t,\cdot)$ on $(m_\epsilon(t),m^\epsilon(t))$, $f(t+\epsilon,x)=f(t,x)$ for $x\in\{m_\epsilon(t),m^\epsilon(t)\}$,  and $f(t+\epsilon,\cdot)>f(t,\cdot)$ on $(\ell_{t+\epsilon},m_\epsilon(t))\cup(m^\epsilon(t),r_{t+\epsilon})$.
		\end{itemize}
	\end{enumerate}
\end{Assumption}
Note that Assumption \ref{Assump:1}(iii) is precisely Assumption \ref{ass:Dispersion_c_appendix}, and thus all the results of Appendix \ref{sec:Decreasing_SMOT_one_period} are valid.

\begin{Remark} \label{rem:TouziAssumption3.1}
Under Assumption \ref{Assump:1} several further implications follow. First, since, for all $t \in [0,1]$, $f(t,\cdot)>0$ on $(\ell_t,r_t)$, we have that $\inf_{x\in[-K,K]\cap(\ell_t,r_t)}f(t,x)>0$ for all $K>0$. Furthermore, due to the Dispersion Assumption (see Assumption \ref{Assump:1}(iii)), for every $t \in [0,1)$ and sufficiently small $\epsilon\in(0,1-t]$, the function $x\mapsto\delta^\epsilon F(t,x)$ attains the (unique) maximum and minimum at $m_\epsilon(t)$ and $m^\epsilon(t)$, respectively. It is also easy to see that, for all $x\in(\ell_t,r_t)$, $\int^\infty_x(y-x)\delta^\epsilon F(t,dy)>0$. In particular, Assumption \ref{Assump:1} implies Henry-Labord\`ere et al. \cite[Assumption 3.1]{HLTanTouzi}.
\end{Remark}

In the rest of this paper we will always work with $(\mu_t)_{t\in[0,1]}$ satisfying Assumption \ref{Assump:1}.

Fix $t\in[0,1)$ and $\epsilon\in(0,1-t]$, and consider a pair of marginals $(\mu_t, \mu_{t+\epsilon})$. Then the corresponding decreasing supermartingale coupling can be determined by a pair of functions $(T_d^\epsilon(t,\cdot)=\hat T_d^{\mu_t,\mu_{t+\epsilon}}(\cdot), T_u^\epsilon(t,\cdot)=\hat T_u^{\mu_t,\mu_{t+\epsilon}}(\cdot))$, where $(\hat T_d^{\mu_t,\mu_{t+\epsilon}}(\cdot), T_u^\epsilon(t,\cdot))$ are as in Appendix \ref{sec:Decreasing_SMOT_one_period_primal}. Furthermore, set $x_1^\epsilon(t):=x_1^{\mu_t,\mu_{t+\epsilon}}$ and  $y_1^\epsilon(t):=y_1^{\mu_t,\mu_{t+\epsilon}}$, where $x_1^{\mu_t,\mu_{t+\epsilon}}$ is defined by \eqref{eq:x_1Appendix}, while $y_1^{\mu_t,\mu_{t+\epsilon}}$ is given by \eqref{eq:y_1Appendix}.  
 
Recall that the canonical space of c{\`a}dl{\`a}g paths $\Om:= \D([0,1] \R)$, is a Polish space equipped with the Skorokhod topology, while $X$ denotes the canonical process. Let $(\pi_n)_{n \geq 1}$ be a sequence of partitions of $[0,1]$, so that each $\pi_n=(t_k^n)_{0 \leq k \leq n}$ is such that $0 = t_0^n < \cdots < t_n^n = 1$. Suppose in addition that $|\pi_n|:=\max_{1 \leq k \leq n} (t_k^n - t_{k-1}^n) \to 0$ as $n\to\infty$. Then for every partition $\pi_n$, by considering the marginal distributions $(\mu_{t_k^n})_{0 \leq k \leq n}$, one obtains an $(n+1)$-marginal (or $n$-period) SOT problem, where the goal is to maximize 
$$
\E \left[ \sum_{0 \leq k \leq n-1} c(\tilde{X}_k^n, \tilde{X}_{k+1}^n) \right]
$$
among all discrete-time supermartingales $\tilde{X}^n=(\tilde{X}^n_k)_{0 \leq k \leq n}$ satisfying marginal constraints. Under  Assumptions \ref{Assump:CostFunction} and \ref{Assump:1}, the iterated $n$-period decreasing supermartingale coupling is the solution to the above supermartingale transport problem. Let $\Om^{*,n}:= \R^{n+1}$ be the canonical space of discrete-time process, and $X^n= (X_k^n)_{0 \leq k \leq n}$ be the canonical process on $\Om^{*,n}$. Then under the optimal supermartingale measure $\P^{*,n}$, $X^n$ is a discrete-time supermartingale and at the same time a Markov chain, characterized by $T_u^{\Delta t^n_{k+1}}(t^n_k,\cdot)$, $T_d^{\Delta t^n_{k+1}}(t^n_k,\cdot)$ with $\Delta t_{k+1}^n:= t_{k+1}^n-t_k^n$, induced by the two marginals $(\mu_{t_k^n}, \mu_{t_{k+1}^n})$; see Appendix \ref{sec:Decreasing_SMOT_one_period_primal}. We then extend the Markov chain $X^n$ to a continuous-time c{\`a}dl{\`a}g process $X^{*,n}=(X^{*,n}_t)_{0 \leq t \leq 1}$ defined by 
$$
X_t^{*,n}:= \sum_{k=1}^{n} X_{k-1}^n \mathbf{1}_{[t_{k-1}^{n}, t_k^n)}(t), ~~t \in [0,1],
$$
and define the probability measure $\P^n:= \P^{*,n} \circ (X^{*,n})^{-1}$ on $\Om$.

We further denote by $\mathbf{m} (\delta^{\epsilon} F(t, \cdot))$ (resp. $\mathbf{m} (\partial_t F(t, \cdot))$) the set of all local minimizers of functions $\delta^{\epsilon} F(t, \cdot))$ (resp. $\partial F(t, \cdot)))$. Let $\mathbf{M} (\partial_t F(t, \cdot))$ be the set of local maximizers of $\partial F(t, \cdot))$. Note that, by Assumption \ref{Assump:1}, $\mathbf{m} (\delta^{\epsilon} F(t, \cdot))=\{m^\epsilon(t)\}$ and $\mathbf{M} (\delta^{\epsilon} F(t, \cdot))=\{m_\epsilon(t)\}$ for all $0<\epsilon\leq \epsilon_1\wedge(1-t)$.
\begin{Assumption} \label{Assump:2}
	\begin{itemize}
		\item[\rm(i)]  There is a constant $\epsilon_2>0$ such that, for all $t \in [0,1]$ and $0 < \epsilon \leq \epsilon_2 \wedge (1-t)$, we have $\mathbf{m} (\partial_t F(t, \cdot))=\{ m_t \}$ and $\mathbf{M} (\partial_t F(t, \cdot))=\{\tilde{m}_t\}$.	
		\item[\rm(ii)] Let $m^0(t)=m_t$. Then the map $(t,\epsilon) \mapsto m^{\epsilon}(t)$ is continuous (and hence uniformly continuous with continuity modulus $\rho_0$) on $\{ (t,\epsilon): 0 \leq \epsilon \leq \epsilon_2, 0\leq t \leq 1-\epsilon \}$.\end{itemize}\end{Assumption}

\begin{Proposition} \label{Prop::Tightness}
	Suppose that Assumptions \ref{Assump:1} and \ref{Assump:2} are  valid. Then the sequence $(\P^n)_{n \geq 1}$ is tight w.r.t. the Skorokhod topology on $\Om$. Moreover, every limit point $\P^0$ satisfies $\P^0 \in \Sc_{\infty}(\mu)$.
\end{Proposition}
The proof of Proposition \ref{Prop::Tightness} requires two additional lemmas. 

\begin{Lemma} \label{lemma:uniform_integral}
	Let $(\mu_t)_{t \in [0,1]}$ be a family of probability measures in $\Pc$ which are increasing in convex-decreasing order. Then $(\mu_t)_{t \in [0,1]}$ is uniformly integrable, i.e.
	$$
	\lim_{K \to +\infty}\sup_{0 \leq t \leq 1} \int |x| \mathbf{1}_{\{ |x| \geq K \}} \mu_t(dx) = 0.
	$$
\end{Lemma}

\begin{proof}
	Using the inequality $|x| \mathbf{1}_{\{|x| \geq K\}} \leq 2(|x|- \frac{K}{2}) \mathbf{1}_{\{|x| \geq \frac{K}{2}\}}$, we have that
	$$
	\begin{aligned}
	\sup_{0 \leq t \leq 1} \int |x| \mathbf{1}_{\{ |x| \geq K \}} \mu_t(dx) &=  \sup_{0 \leq t \leq 1} \int 2(|x|- \frac{K}{2}) \mathbf{1}_{\{|x| \geq \frac{K}{2}\}} \mu_t(dx) \\
	&\leq 2  \sup_{0 \leq t \leq 1} \int (x- \frac{K}{2}) \mathbf{1}_{\{x \geq \frac{K}{2}\}} \mu_t(dx) + 2  \sup_{0 \leq t \leq 1} \int (-x- \frac{K}{2}) \mathbf{1}_{\{x \leq -\frac{K}{2}\}} \mu_t(dx) \\
	& = 2  \sup_{0 \leq t \leq 1} C_{\mu_t} (\frac{K}{2}) + 2  \sup_{0 \leq t \leq 1} P_{\mu_t} (-\frac{K}{2}) \\
	&= 2  C_{\tilde{\mu}} (\frac{K}{2}) + 2 P_{\mu_1} (-\frac{K}{2}) \to 0, \mbox{   when } K \to \infty,
	\end{aligned}
	$$
	where $\tilde{\mu}$ is some probability measure on $\R$ with the same mass and mean with $\mu_0$. We now explain the last equality. First, $x \mapsto (-x-\frac{K}{2})^+$ is convex and non-increasing, and thus using the convex-decreasing ordering of $(\mu_t)_{t\in[0,1]}$ we obtain that $\sup_{0 \leq t \leq 1} P_{\mu_t} (-\frac{K}{2})=P_{\mu_1} (-\frac{K}{2})$. On the other hand, $ \sup_{0 \leq t \leq 1} C_{\mu_t} (\frac{K}{2}) = C_{\tilde{\mu}} (\frac{K}{2})$ follows from the fact that the point-wise supremum of a family of convex functions is convex, and $\lim_{\lvert x\lvert \to \infty} \{\sup_{0 \leq t \leq 1} C_{\mu_t} (x)-\{\overline{\mu_0}-\mu_0(\R)x\}\} =0$. Therefore, there exists some probability measure $\tilde{\mu}$, such that $\tilde\mu(\R)=\mu_0(\R)$, $\overline{\mu_0}=\overline{\tilde{\mu}}$ and $C_{\tilde\mu}=\sup_{0 \leq t \leq 1} C_{\mu_t}$.
\end{proof}

Recall that $(\mu_t)_{t\in[0,1]}$ satisfies Assumption \ref{Assump:1}. Now for fixed $t \in [0,1)$, $\epsilon \in (0, 1-t]$ and $x\in(r_t,m^\epsilon(t))$ denote by $q^{\epsilon}(t,x)$ the conditional probability of the (martingale) upward jump under the decreasing supermartingale coupling, so that  $q^{\epsilon}(t,x)= \frac{x-T_d^{\epsilon}(t, x)}{T_u^{\epsilon}(t, x) - T_d^{\epsilon}(t, x)}$. Note that, for $x\leq x_1^\epsilon(t)$, $T^\epsilon_u(t,x)=\infty$ and therefore $q^\epsilon(t,x)=0$.

Furthermore, define $J^\epsilon_d(t,\cdot),J^\epsilon_u(t,\cdot)$ by
$$
J^\epsilon_d(t,x)=x-T_d^\epsilon(t,x),\quad J^\epsilon_u(t,x)=T_u^\epsilon(t,x)-x,\quad x\in(\ell_t,r_t),
$$
so that $J^\epsilon_d(t,\cdot),J^\epsilon_u(t,\cdot)$ correspond to the downward and upward distance each particle travels (between times $t$ and $t+\epsilon$) under the decreasing supermartingale coupling, conditioned it started at $x$.
(Note that $J^\epsilon_u(t,x)=\infty$ for $x\leq x^\epsilon_1(t)$, but this happens with zero probability.)

\begin{Lemma} \label{Lemma:uniform_boudedness}
	
Suppose Assumptions \ref{Assump:1} and \ref{Assump:2} hold. Then
	\begin{itemize}
		\item[\rm(i)] For every $K>0$, there is a constant $C_1$, independent of $(t,x,\epsilon)$ such that
		$$
		J_d^{\epsilon} (t,x) + q^{\epsilon}(t,x) \leq C_1 \epsilon, ~~~ \forall x \in [-K,K] \cap (l_t, r_t) \cap (x_1^{\epsilon}(t), m^{\epsilon}_t).
		$$

		\item[\rm(ii)] For every $K>0$, there is a constant $C_2$, independent of $(t,x,\epsilon)$ such that
		$$
		J_d^{\epsilon} (t,x) \leq C_2 \epsilon, ~~~\forall x \in [-K,K] \cap (l_t, x_1^{\epsilon}(t)].
		$$
	\end{itemize}
\end{Lemma}
\begin{proof}
	(i) Since $\mu_t\lvert_{[x^\epsilon_1(t),r_t)}\leq_c\mu_{t+\epsilon}\lvert_{[y_1^\epsilon(t),r_{t+\epsilon})}$ and the decreasing supermartingale coupling $\hat\P^{\mu_t,\mu_{t+\epsilon}}$ coincides with the right-curtain coupling of $\mu_t\lvert_{[x^\epsilon_1(t),r_t)}\leq_c\mu_{t+\epsilon}\lvert_{[y_1^\epsilon(t),r_{t+\epsilon})}$ (see Appendix \ref{sec:Decreasing_SMOT_one_period_primal}), the result follows immediately from Henry-Labord\`ere and Touzi \cite[Lemma 6.4]{HLTouzi16} applied to the right-curtain coupling.
	
	(ii) For $x\leq x^\epsilon_1(t)$, $T_d^\epsilon(t,x)=F^{-1}(t+\epsilon,F(t,x))$ and therefore
	$$
	\begin{aligned}
	J_d^{\epsilon} (t,x) &=  F^{-1} (t+\epsilon, F(t+\epsilon,x)) - F^{-1} (t+\epsilon, F(t,x))  \\
	&= \frac{1}{f(t+\epsilon, F^{-1}(t+\epsilon, \xi) )} \delta^{\epsilon} F(t,x),
	\end{aligned}
	$$
	for some $\xi$ between $F(t+\epsilon,x)$ and $F(t,x)$, by the mean value theorem. Now one draws the desired conclusion from the fact that $|\delta^{\epsilon} F| \leq C \epsilon$, for some constant $C>0$.
\end{proof}

We are now ready to prove Proposition \ref{Prop::Tightness}.
\begin{proof}[Proof of Proposition \ref{Prop::Tightness}]

Recall that $\P^n$ is a supermartingale measure on the canonical space $\Om$, induced by the continuous-time supermartingale $X^{*,n}$ under the probability $\P^{*,n}$. The supermartingale $X^{*,n}$ jumps only at discrete time points relative to the partition $\pi_n=(t_k^n)_{0 \leq k \leq n}$. Moreover, at time $t_{k+1}^n$, on $\{x^\epsilon_1(t^n_k)\leq X_{t_n^k}\}$, the upward jump size is $J_u^{\epsilon}(t_n^k, X_{t_n^k})$ and downward jump size is $J_u^{\epsilon}(t_n^k, X_{t_n^k})$, where $\epsilon:=t_{k+1}^n - t_k^n$. In addition, at time $t_{k+1}^n$ and on $\{x^\epsilon_1(t^n_k)\geq X_{t_n^k}\}$, the downward jump size is $J_d^{\epsilon}(t_n^k, X_{t_n^k})$ with probability 1.

For given positive constants $C,\theta$, we introduce 
$$
\begin{aligned}
\Ec_n(C,\theta):=& \inf \{ \Pi_{i=j}^{k-1} ( 1 - C(t_{i+1}^n - t_i^n) ): \mbox{ for some } s \in [0,1 ) \mbox{ and }  \\
&~~~~~~0 \leq j \leq k \leq n \mbox{ such that } s \leq t_j^n \leq t_{k+1}^n \leq s+ \theta \}.
\end{aligned}
$$
\begin{itemize}
	\item[\rm(i)]  
	First, using Jacod and Shiryaev \cite[Thoerem VI.4.5]{JacodShiryaev}, we show that $(\P^n)_{n\geq 1}$ is tight. Let $\tau:= 1 \wedge \inf\{ s: |X_s| \geq K \}$, and $\tau$ is a stopping time w.r.t. the canonical filtration generated by $(X_s)_{0 \leq s \leq 1}$. In addition,
	$$
	\begin{aligned}
	K \P^n \left[ \sup_{0 \leq t \leq 1} |X_t| \geq K \right] & \leq \E^{\P^n} \left[ |X_{\tau}| \right] \\
	&= \E^{\P^n} \left[ X_{\tau} + 2 X_{\tau}^- \right] \leq \E^{\P^n} \left[ X_0 + 2 X_{\tau}^- \right] \\
	&\leq  \E^{\P^n} \left[ X_0 + 2 X_{1}^- \right] \leq  \E^{\P^n} \left[ |X_0| \right] + 2  \E^{\P^n} \left[ |X_{1}| \right] < +\infty,
	\end{aligned}
	$$
	where on the above, the first inequality comes from Markov inequality, $\P^{*,n} \circ (X^{*,n})^{-1} =  \P^{n}\circ X^{-1}$, and that fact that $X^{*,n}_t$ is piecewise constant. The third inequality comes from optional stopping theorem and the fact that $X^-$ is a submartingale. Let $\eta > 0$ be an arbitrary small real number, then there is some $K > 0$ such that
	$$
	\P^n[ \sup_{0 \leq t \leq 1} |X_t| \geq K ] \leq \eta, ~~~\mbox{for all } n \geq 1.
	$$ 
	 We can assume w.l.o.g. that $K$ satisfies $-K<m_t < K$ for all $t \in [0,1]$; recall Assumption \ref{Assump:2}.
	 
	 The remaining proof is almost identical to the proof of \cite[Proposition 3.2]{HLTanTouzi}. However, in few places the supermartingale transitions appear (which are not present in \cite{HLTanTouzi}) and thus we sketch the arguments.
	 
	 Denote in addition that $r^{K}(t) : = r_t \wedge K$ and $\ell^{K}(t) : = \ell_t \lor (-K)$.	Let $\delta > 0$, then it follows from Lemma \ref{Lemma:uniform_boudedness} that $J_d^{\epsilon}(t_n^k, X_{t_n^k})$ is uniformly bounded by $C \epsilon$ for some constant $C$ on $D^K_{\delta}:=\{ (t,x): l^K(t) + \delta/2 \leq x \leq m_t \}$. Let $\theta > 0$ satisfy $\theta \leq \frac{\delta}{2C}$ and $|l^K(t+\theta) - l^K(t)|+|r^K(t+\theta) - r^K(t)| \leq \delta/2$.
Let $S, T$ be two stopping times w.r.t. the filtration generated by $X^{*,n}$ such that $0 \leq S \leq T \leq S+\theta \leq 1$. Since, by Lemma \ref{Lemma:uniform_boudedness}, the big jumps of $X^{*,n}$ correspond to $J_u^{\epsilon}(t_n^k, X_{t_n^k})$, and not to $J_d^{\epsilon}(t_n^k, X_{t_n^k})$,  using the same arguments as in the proof of \cite[Proposition 3.2]{HLTanTouzi}, we have that 

	$$
	\P^{*,n} [ \sup_{0 \leq t\leq 1} | X_t^{*,n}| \leq K, |X_T^{*,n} - X_S^{*,n}| \geq \delta ] \leq 1 - \Ec_n(C,\theta).
	$$
	It follows that
	$$
	\begin{aligned}
	&\underset{\theta \to 0}{\lim \sup} ~ \underset{n \to \infty}{\lim \sup}  \P^{*,n} [ | X_T^{*,n} - X_S^{*,n} | \geq \delta] \\
	&\leq \underset{\theta \to 0}{\lim \sup} ~ \underset{n \to \infty}{\lim \sup}  \P^{*,n} [ \sup_{0 \leq t \leq 1} |X_t^{*,n}| \leq K, | X_T^{*,n} - X_S^{*,n} | \geq \delta] + \P^{*,n} \left[ \sup_{0 \leq t \leq 1} |X_t^{*,n}| \geq K \right]  \\
	&\leq  \underset{\theta \to 0}{\lim \sup} ~ \underset{n \to \infty}{\lim \sup} ( 1 - \Ec_n(C, \theta) ) + \eta = \eta.
	\end{aligned}
	$$
	Since $\eta >0$ can be arbitrarily small, we conclude that
	$$
	\underset{\theta \to 0}{\lim \sup} ~ \underset{n \to \infty}{\lim \sup}  \P^{*,n} [ | X_T^{*,n} - X_S^{*,n} | \geq \delta] = 0.
	$$
	Finally, Jacod and Shiryaev \cite[Thoerem VI.4.5]{JacodShiryaev} shows that $(\P^n)_{n\geq 1}$ is tight.
	\item[\rm(ii)] Let $\P^0$ be a limit of $(\P^n)_{n\geq 1}$. Then repeating the arguments of the second part of the proof of \cite[Proposition 3.2]{HLTanTouzi} (which relies on the continuity of $F(t,x)$), we obtain that $\P^0\circ X^{-1}_t=\mu_t$.

	\item[\rm(iii)] In the last step, we show that $X$ is still a supermartingale under $\P^0$. For every $K>0$, we first define the auxiliary process $X_t^K:=(-K) \lor X_t \wedge K$. Given $s< t$ and $\varphi(s, X_\cdot)$ a bounded continuous, $\Fc_s$-measurable function, it follows from weak convergence that
	$$
	\lim_{n \to +\infty}\E^{\P^n}[\varphi(s, X_\cdot)(X_t^K-X_s^K))] = \E^{\P^0}[\varphi(s, X_\cdot)(X_t^K-X_s^K))].
	$$
	As $X$ is a $\P^n$-supermartingale, we have that
	$$
	\begin{aligned}
	\E^{\P^n}[\varphi(s, X_\cdot)(X_t^K-X_s^K))] &\leq -\E^{\P^n} [ \varphi(s, X_\cdot) ( X_t \mathbf{1}_{\{ |X_t| \geq K \}} - X_s \mathbf{1}_{\{ |X_s| \geq K \}} ) ] \\
	&\leq |\E^{\P^n} [ \varphi(s, X_\cdot) ( X_t \mathbf{1}_{\{ |X_t| \geq K \}} - X_s \mathbf{1}_{\{ |X_s| \geq K \}} ) ] | \\
	& \leq 2 |\varphi|_{\infty} \sup_{0 \leq t \leq 1} \int |x| \mathbf{1}_{\{ |x| \geq K \}} \mu_t(dx) \\
	& \to 0, \mbox{       as } K \to \infty,
	\end{aligned}
	$$
	where the last inequality follows from Lemma \ref{lemma:uniform_integral} and the above convergence is uniformly in $n$. Finally, by the dominated convergence theorem, we have that
	$$
	\begin{aligned}
	\E^{\P^0}[\varphi(s, X_\cdot)(X_t-X_s))] &= \lim_{K \to \infty} \E^{\P^0}[\varphi(s, X_\cdot)(X_t^K-X_s^K))] \\
	&=\lim_{K \to \infty}  \lim_{n \to +\infty}\E^{\P^n}[\varphi(s, X_\cdot)(X_t^K-X_s^K))] \\
	&\leq 0.
	\end{aligned}
	$$
	As $\varphi$ is arbitrary, we conclude that $X$ is a $\P^0$-supermartingale.
\end{itemize}
\end{proof}

\subsection{Characterization of the limiting process}\label{sec:SDE}

To further analyse the limiting process of the previous section, we consider the decreasing  supermartingale coupling (see Appendix \ref{sec:Decreasing_SMOT_one_period}) of two marginals $\mu_t$ and $\mu_{t+\epsilon}$. Recall that $x_1^{\epsilon}(t)$ denotes the unique phase transition point which divides the martingale region (to the right of $x_1^{\epsilon}(t)$) and the supermartingale region (to the left of $x_1^{\epsilon}(t)$).  In the supermartingale region, the mass at $x\leq x^\epsilon_1(t)$ is transported (through the quantile coupling) to the destination $T_d^{\epsilon}(t,x)=F^{-1}(t+\epsilon,F(t,x))$. In the martingale region, the two supporting maps $T_u^{\epsilon}$ and $T_d^{\epsilon}$ are characterized in Proposition \ref{Prop:one_period_transport_map}. 

	Set $\epsilon_0:=\epsilon_1\wedge\epsilon_2$, where $\epsilon_1,\epsilon_2$ are as in Assumptions \ref{Assump:1} and \ref{Assump:2}, respectively.
\begin{Assumption} \label{Assump:3}
	\begin{itemize}

		\item[\rm(i)] The map $(t,\epsilon) \mapsto x_1^{\epsilon}(t)$ is uniformly continuous (with continuity modulus $\rho_1$) and also uniformly bounded on $\{ (t,\epsilon): 0 < \epsilon \leq \epsilon_0, 0\leq t \leq 1-\epsilon \}$.
		\item[\rm(ii)] For every $t \in [0,1]$, $r_t=r\in(-\infty,\infty]$.
		\item[\rm(iii)] For every $t \in [0,1]$, we have $\partial_{tx} f(t, m_t) >0$.
	\end{itemize}
\end{Assumption}

\begin{Lemma} \label{lemma:phase_transition_curve}
	Suppose that  Assumptions \ref{Assump:1}, \ref{Assump:2}, \ref{Assump:3}(i) and (ii) hold. Then the limit $\lim_{\epsilon \to 0} x_1^{\epsilon}(t)$ exists. Moreover, if we denote the limit by $x_1(t)$, then $t\mapsto x_1(t)$ is continuous and uniquely determined by the equation
	$$
	\int_{x_1(t)}^{+\infty} (x_1(t) - \xi) \partial_t f(t, \xi) d \xi = 0,\quad t\in[0,1].
	$$
\end{Lemma}
\begin{proof}
It follows from Lemma \ref{lem:transition_dispersion} that $x_1^{\epsilon}(t)$ is unique for all $t \in [0,1)$, $\epsilon \in (0, \epsilon_0]$. From Assumption \ref{Assump:3} {\rm(i)}, using the Arzel{\`a}-Ascoli theorem, we have that $x_1^{\epsilon}(t)$ converges when ${\epsilon} \to 0$, and the limit remains a continuous function.

Using the properties of the decreasing supermartingale coupling (see Appendix \ref{sec:Decreasing_SMOT_one_period}), we have that $x_1^{\epsilon}(t)$ and $y_1^{\epsilon}(t)$ satisfy the mean and mass preservation conditions:
$$
\int_{x_1^{\epsilon}(t)}^{r_t} f(t,x) dx = \int_{y_1^{\epsilon}(t)}^{r_{t+\epsilon}} f(t+{\epsilon},x) dx, ~~~\int_{x_1^{\epsilon}(t)}^{r_t} x f(t,x) dx = \int_{y_1^{\epsilon}(t)}^{r_{t+\epsilon}} x f(t+\epsilon,x) dx.
$$
Rearranging, and using that $r_t=r_{t+\epsilon}$, gives
$$
\int_{x_1^{\epsilon}(t)}^{r_t} (f(t,x)-f(t+\epsilon,x) ) dx = \int_{y_1^{\epsilon}(t)}^{x_1^{\epsilon}(t)} f(t+\epsilon,x) dx, ~~~\int_{x_1^{\epsilon}(t)}^{+\infty} x(f(t,x)-f(t+\epsilon,x) ) dx = \int_{y_1^{\epsilon}(t)}^{x_1^{\epsilon}(t)} x f(t+\epsilon,x) dx.
$$
Using Taylor's expansion $f(t+\epsilon,x) - f(t,x)=\partial_t f(t,x) \epsilon + \frac{1}{2} \partial_{tt} f(\xi_1, x) \epsilon^2$ (with $\xi_1 \in (t, t+\epsilon)$) in the first equation above and applying the mean value theorem for definite integrals, we get that
$$
- \epsilon \int_{x_1^{\epsilon}(t)}^{+\infty} \partial_t f(t,x) dx - \frac{1}{2}  \epsilon^2 \int_{x_1^{\epsilon}(t)}^{+\infty} \partial_{tt} f(\xi_1,x) dx   = \left(x_1^{\epsilon}(t) - y_1^{\epsilon}(t)\right)  f(t, \xi_2), 
$$
where $\xi_2$ lies in $[y_1^{\epsilon}(t),x_1^{\epsilon}(t)]$. Letting $\epsilon$ tend to $0$ on both sides, the l.h.s. tends to $0$, and consequently r.h.s. also tends to $0$. It follows that $\lim_{\epsilon \to 0} x_1^{\epsilon}(t) =\lim_{\epsilon \to 0} y_1^{\epsilon}(t)$. Dividing both sides by $\epsilon$ we get
\begin{equation} \label{eq:phase_transition1}
- \int_{x_1^{\epsilon}(t)}^{+\infty} \partial_t f(t,x) dx   = \frac{ x_1^{\epsilon}(t) - y_1^{\epsilon}(t)}{\epsilon} f(t, \xi_2) + O(\epsilon), 
\end{equation}
Similarly,
\begin{equation}\label{eq:phase_transition2}
- \int_{x_1^{\epsilon}(t)}^{+\infty} x \partial_t f(t,x) dx   = \frac{ x_1^{\epsilon}(t) - y_1^{\epsilon}(t)}{\epsilon} \xi_2 f(t, \xi_2) + O(\epsilon).
\end{equation}
Note that Assumption \ref{Assump:3}(ii) ensures that $x\mapsto\partial_tf(t,x)$ is not a constant.

Combining \eqref{eq:phase_transition1} and \eqref{eq:phase_transition2}, and letting ${\epsilon}$ tend to $0$, we get that
$$
\int_{x_1(t)}^{+\infty} ( x - x_1(t) ) \partial_t f(t,x) dx = 0.
\vspace{-10mm}
$$
\end{proof} 
 \begin{Remark}
 	The Assumption \ref{Assump:3}(ii) can be relaxed. For example, if $t\mapsto r_t$ is differentiable, then the equation that characterizes $x_1(t)$, as in Lemma \ref{lemma:phase_transition_curve}, can be shown to be of the following form:
 	$$
 	\int_{x_1(t)}^{r_t} (x_1(t) - \xi) \partial_t f(t, \xi) d \xi = -f(t,r_t)(x_1(t)-r_t)\frac{d r_t}{dt},\quad t\in[0,1].
 	$$
 	Consequently, the subsequent arguments would have to be adjusted. Since our main motivation is the examples that satisfy Assumption \ref{Assump:3}(ii), we do not include the details.
 \end{Remark}
We also introduce $T_u$ through the following integral equation:
\begin{equation} \label{eq:T_S_charac}
\int_x^{T_u(t,x)} (x - \xi) \partial_t f(t, \xi) d \xi = 0,
\end{equation}

\begin{Proposition}
Suppose that  Assumptions \ref{Assump:1}, \ref{Assump:2}, \ref{Assump:3}(i) and (ii) hold. For $x \in (x_1(t), m_t)$, \eqref{eq:T_S_charac} has the unique solution $T_u(t,x)$ on $(m_t,r_t)$.
\end{Proposition}
\begin{proof}
Define $G(t,x,y):=\int_x^{y} (x - \xi) \partial_t f(t, \xi) d \xi $, as $y \mapsto G(t,x,y)$ is continuous, it is enough to show that, for $x \in [x_1(t), m_t)$, $y \mapsto G(t, x, y)$ is decreasing on $y \in [m_t, \infty)$, $G(t,x,m_t)>0$ and $G(t,x, +\infty)<0$. 

We first show that for $x \in [x_1(t), m_t]$, $y \mapsto G(t, x, y)$ is decreasing on $y \in [m_t, \infty)$. Indeed, for $m_t \leq y_1 < y_2$, we have 
$$
G(t, x, y_2) - G(t, x, y_1) = \int_{y_1}^{y_2} (x - \xi) \partial_t f(t, \xi) d \xi < 0,
$$
where the last inequality follows from the fact that $\xi > y_1 \geq m_t \geq x$ and $\partial_t f(t, \xi)>0$ on $[y_1, y_2]$ (as $x \mapsto \partial_t F(t,x)$ is increasing on $(m_t, \infty)$).

Notice that the equation $\partial_t F(t,x)=0$ has a unique solution under Assumption \ref{Assump:2}, and we denote it by $\bar{x}(t)$.  Now let us show that $x \mapsto G(t,x,+\infty)$ is decreasing on $(x_1(t), \bar{x}(t) )$, and increasing on $(\bar{x}(t), m_t )$. Indeed,
$$
\partial_x G(t,x,+\infty) = \int_x^{+\infty} \partial_t f(t, \xi) d \xi = \partial_t F(t,+\infty) - \partial_t F(t,x) = - \partial_t F(t,x), 
$$
which is negative on $(x_1(t), \bar{x}(t) )$, and positive on $(\bar{x}(t), m_t )$. To show $G(t,x,+\infty) < 0$ for all $ x_1(t) < x < m_t$, it is enough to show that $G(t, x_1(t), +\infty) \leq 0$ and $G(t, m_t, +\infty) \leq 0$. Notice that from Lemma \ref{lemma:phase_transition_curve}, $x_1(t)$ satisfies that 
$$
\int_{x_1(t)}^{+\infty} (x_1(t) - \xi) \partial_t f(t, \xi) d \xi = 0,
$$
i.e., $G(t, x_1(t), +\infty)=0$. Moreover, as $G(t, m_t, m_t)=0$ and $y \mapsto G(t, m_t, y)$ is strictly decreasing on $y \in [m_t, \infty]$ which is shown above, it follows that $G(t, m_t, +\infty) < 0$.

In addition, $x \mapsto G(t,x, m_t)$ is decreasing on the interval $(x_1(t), m_t)$:
$$
\partial_x G(t,x,m_t) = \int_x^{m_t} \partial_t f(t, \xi) d \xi = \partial_t F(t,m_t) - \partial_t F(t,x) < 0.
$$
It follows that $G(t,x, m_t) > G(t, m_t, m_t) = 0$.  \end{proof}

We further define $j_u$ and $j_d$ for $t \in [0,1], x \in (x_1(t), m_t)$ by $j_u(t,x):= T_u(t,x) - x$ and $j_d(t,x):= \frac{\partial_t F(t,x) - \partial_t F(t, T_u(t,x)) }{f(t,x)}$. The following lemma analyzes the asymptotic behaviour of $T_d^{\epsilon}$, $T_u^{\epsilon}$ and $q^\epsilon(t,x)$ in the martingale region, and when $\epsilon$ is small. The proof is analogous to the proof of Henry-Labord\`ere et al. \cite[Lemma 6.6]{HLTanTouzi} when applied to the right-curtain coupling, and thus omitted.
\begin{Lemma} \label{lemm:asymptotic_expansion} Suppose that  Assumptions \ref{Assump:1}, \ref{Assump:2}, \ref{Assump:3}(i) and (ii) hold.
	
	For $0 < \delta < K < \infty$, let $E_{\delta,K}:=\{ (t,x):   t\in[0,1], (m_t-K)\lor x_1(t) < x \leq m_t-\delta \}$. Define $j_d^{\epsilon}(t,x):= \frac{\partial_t F(t,x) - \partial_t F(t, T_u^{\epsilon}(t,x)) }{f(t+\epsilon,x)}$. Then $T_u^{\epsilon}(t,x)$ and $T_d^{\epsilon}(t,x)$ admit the following expansions:
	$$
	T_u^{\epsilon}(t,x)-x:= J_u^{\epsilon} (t,x) = j_u (t,x)  + (\epsilon \lor \rho_0(\epsilon) ) e_u^{\epsilon}(t,x),
	$$
	and
	$$
	x-T_d^{\epsilon}(t,x):= J_d^{\epsilon} (t,x) =  \epsilon j_d^{\epsilon}(t,x) + \epsilon^2 e_d^{\epsilon}(t,x),
	$$
	with $e_d^{\epsilon}(t,x)$ and $e_u^{\epsilon}(t,x)$ bounded on $E_{\delta,K}$. Consequently, there exists a constant $C_{\delta,K}$ such that the probability of the upward jump $q^{\epsilon}$ admits the asymptotic expansion:
	\begin{equation} \label{eq:jump_proba_expan}
	q^{\epsilon}(t,x) = \frac{J_d^{\epsilon} (t,x)}{J_u^{\epsilon} (t,x) + J_d^{\epsilon} (t,x)} = \epsilon \frac{j_d(t,x)}{j_u(t,x)} + C_{\delta,K} \epsilon \left( \epsilon \lor \rho_0(\epsilon) \right), ~~\mbox{for } (t,x) \in E_{\delta,K}.
	\end{equation}
	\end{Lemma}

The next lemma considers the asymptotic expansion of $J^\epsilon_d$ in the supermartingale region.

\begin{Lemma} \label{lemm:asymptotic_expansion_2} Suppose that  Assumptions \ref{Assump:1}, \ref{Assump:2}, \ref{Assump:3}(i) and (ii) hold.
	
 For $0 < K < \infty$, let $E_{K}:=\{ (t,x):   t\in[0,1], (x_1(t)-K)\lor l_t \leq x \leq x_1(t) \}$. Then, $x - T_d^{\epsilon} (t,x) = J_d^{\epsilon} (t,x) =  \epsilon j_d^{\epsilon}(t,x) + \epsilon^2 e_d^{\epsilon}(t,x)$, with $j_d^{\epsilon}(t,x) : = \frac{1}{f(t+\epsilon,x)} \partial_t F (t,x)$, and $e_d^{\epsilon}(t,x)$ bounded on $E_{K}$.
\end{Lemma}
\begin{proof}
For $J_d^{\epsilon} (t,x)$, applying Taylor's expansion to the first order, we have that
$$
\begin{aligned}
J_d^{\epsilon} (t,x)& = x - T_d^{\epsilon} (t,x)  =  F^{-1} (t+\epsilon, F(t+\epsilon,x)) - F^{-1} (t+\epsilon, F(t,x)) \\
&=\frac{\partial F^{-1}}{\partial x}(t+\epsilon, F(t+\epsilon,x))  \left[ F(t+\epsilon,x) - F(t,x) \right] + \frac{1}{2}  \frac{\partial^2 F^{-1}}{\partial x^2}(t+\epsilon,c)   \left[ F(t+\epsilon,x) - F(t,x) \right] ^2 \\
&= \frac{1}{f(t+\epsilon, x)} \left[ F(t+\epsilon,x) - F(t,x) \right] + \frac12 (-1) \frac{f'(t+\epsilon,\xi)}{f^3(t+\epsilon,\xi)} \left[ F(t+\epsilon,x) - F(t,x) \right] ^2,
\end{aligned}
$$
where on the above, $c$ is between $F(t,x)$ and $F(t+\epsilon,x)$, and $\xi=F^{-1}(t+\epsilon, c)$. Applying further second-order Taylor's expansion to the first $F(t+\epsilon,x) - F(t,x)$ and first-order Taylor's expansion to the second $F(t+\epsilon,x) - F(t,x)$, we have 
$$
\begin{aligned}
J_d^{\epsilon} (t,x) &=  \frac{1}{f(t+\epsilon, x)} \left[ \partial_t F(t,x) \epsilon+ \frac12 \partial_{tt} F(\xi_1, x) \epsilon^2 \right] - \frac12  \frac{f'(t+\epsilon,\xi)}{f^3(t+\epsilon,\xi)}  \left[ \partial_t F(\xi_2,x) \epsilon \right]^2\\
&=\frac{ \partial_t F(t,x)}{f(t+\epsilon, x)} \epsilon + \left[ \frac{\partial_{tt} F(\xi_1, x)}{2 f(t+\epsilon, x) }  - \frac{f'(t+\epsilon,\xi)}{2f^3(t+\epsilon,\xi)}  \left( \partial_t F(\xi_2,x) \right)^2 \right] \epsilon^2,
\end{aligned}
$$
with $\xi_1, \xi_2$ between $t$ and $t+\epsilon$. Hence the expansion of $J_d^{\epsilon}$ is valid, with 
$$
|e^{\epsilon}_T(t,x)| \leq \sup_{t \leq s \leq t+\epsilon, F^{-1}(t+\epsilon, F(t,x)) \leq \xi \leq x } \left\lvert \frac{\partial_{tt} F(s, \xi)}{2 f(s, \xi) }  - \frac{f'(s,\xi)}{2f^3(s,\xi)}  \left( \partial_t F(s,\xi) \right)^2 \right\lvert .
$$
\end{proof}

For $t\in[0,1]$, extend the definition of  $j_d$ to $(\ell_t,x_1(t)]$, by setting $j_d(t,x) := \frac{1}{f(t,x)} \partial_t F (t,x)$. For the main result of this section we need one additional result.
\begin{Lemma} \label{Lemma:character_limit} Suppose that  Assumptions \ref{Assump:1}, \ref{Assump:2} and \ref{Assump:3} hold.
\begin{itemize}
\item[\rm(i)] $T_u$ is strictly decreasing in $x$ on $(x_1(t), m_t)$.

\item[\rm(ii)] $j_d(t,x) \mathbf{1}_{\{ x_1(t) < x < m_t\}}, j_u(t,x) \mathbf{1}_{\{x_1(t) < x < m_t\}}$, $\frac{j_d}{j_u}(t,x) \mathbf{1}_{\{x_1(t)< x <m_t\}}$ and $j_d(t,x) \mathbf{1}_{\{x \leq x_1(t)\}}$ are all locally Lipschitz in $(t,x) \in \{ 0 \leq t \leq 1, ( -K \vee l_t ) \leq x \leq (K \wedge r_t) \}$.
\end{itemize}
\end{Lemma}
\begin{proof}
\begin{itemize}
\item[\rm(i)]  Differentiating both sides of \eqref{eq:T_S_charac} with respect to $x \in (x_1(t), m_t)$, we get that 
$$
\partial_x T_u(t,x) \left( x - T_u(t,x) \right) \partial_t f(t, T_u(t,x)) + \int_x^{T_u(t,x)} \partial_t f(t, \xi) d \xi = 0.
$$
Hence
\begin{equation} \label{eq:TS_derivative}
\partial_x T_u(t,x) = \frac{ \partial_t F(t,T_u(t,x)) - \partial_t F(t,x) }{\left( x - T_u(t,x) \right) \partial_t f(t, T_u(t,x))}<0, 
\end{equation}
where the last inequality follows from $ \partial_t F(t,T_u(t,x)) > \partial_t F(t,x)$.

\item[\rm(ii)] Using Assumption \ref{Assump:2} and repeating the arguments of the proof of \cite[Lemma 3.6]{HLTanTouzi} (applied to the right-curtain coupling case) we get that all $j_d(t,x) \mathbf{1}_{\{ x_1(t) < x < m_t\}}, j_u(t,x) \mathbf{1}_{\{x_1(t) < x < m_t\}}$, $\frac{j_d}{j_u}(t,x) \mathbf{1}_{\{x_1(t) < x <m_t\}}$ are locally Lipschitz. It is left to verify the regularity of the function $j_d(t,x) \mathbf{1}_{\{x \leq x_1(t)\}}$. But Assumption \ref{Assump:1} and the definition $j_d(t,x) = \frac{1}{f(t,x)} \partial_t F (t,x)$ immediately imply the local Lipschitz property.
\end{itemize}
\end{proof}

In the following, we provide a characterization of the limit of the sequence $(\P^n)_{n\geq 1}$. The limiting process is a pure-jump process that is similar in spirit to the local L{\'e}vy models introduced by Carr, Geman, Madan and Yor \cite{CarrGemanMadanYor}. (See also \cite{HLTanTouzi} for the related martingale case.)

\begin{Theorem} \label{Thm:main}
Suppose that Assumptions \ref{Assump:1}, \ref{Assump:2}, \ref{Assump:3} hold. Then $\P^n \to \P^0$, where $\P^0$ is the unique weak solution to the SDE:
\begin{equation} \label{eq:SDE_main}
X_t = X_0 + \int_0^t j_u(t,X_{s-}) ( d N_s - \nu_s ds )   \mathbf{1}_{\{x_1(s) < X_{s-} < m_s\}} - \int_0^t \mathbf{1}_{\{ X_{s-}\leq x_1(s)\}}  j_d(t,X_{s-}) ds,
\end{equation}
and $(N_s)_{0 \leq s \leq 1}$ is a unit size jump process with predictable compensated process $(\nu_s)_{0 \leq s \leq 1}$ given by:
$$
\nu_s:= \frac{j_d}{j_u} (s,X_{s-}) \mathbf{1}_{\{x_1(s) < X_{s-}< m_s\}};
$$
here $x_1(t)$ is given by Lemma \ref{lemma:phase_transition_curve}, $j_u(t,x):= T_u(t,x) - x$ and $j_d(t,x):= \frac{\partial_t F(t,x) - \partial_t F(t, T_u(t,x)) }{f(t,x)}$ for $x\in(x_1(t),m_t)$, $j_d(t,x) := \frac{1}{f(t,x)} \partial_t F (t,x)$ for $x\in(\ell_t,x_1(t)]$, and $T_u$ is characterized  by \eqref{eq:T_S_charac} on $x\in(x_1(t),m_t)$.
\end{Theorem}
\begin{proof} 
	By Proposition \ref{Prop::Tightness}, the sequence of supermartingale measures $(\P^n)_{n \geq 1}$ induced by the decreasing supermartingale transport plan is tight. We shall prove that any limit of $(\P^n)_{n \geq 1}$ provides a weak solution to \eqref{eq:SDE_main}. Our strategy is identical to that of Henry-Labord\`ere and Touzi \cite[Theorem 3.7]{HLTouzi16}, and thus we only present the steps that are affected by the supermartingale feature of our construction.
	
	 For all $\mathbf{x} \in \Om:= \D([0,1], \R)$ and $\varphi \in C^1_b(\R)$, we define
	$$
	\begin{aligned}
	&M_t(\varphi, \mathbf{x}): = \varphi(\mathbf{x}_t) - \int_0^t j_d(s, \mathbf{x}_{s-}) D \varphi (\mathbf{x}_{s-}) \mathbf{1}_{\{ x_1(s) < \mathbf{x}_{s-} < m_s\}} ds - \int_0^t \large[ [ \varphi(\mathbf{x}_{s-} + j_u(s, \mathbf{x}_{s-}) )  \\
	& - \varphi(\mathbf{x}_{s-}) ] \frac{j_d}{j_u}(s, \mathbf{x}_{s-})\large]  \mathbf{1}_{ \{x_1(s) < \mathbf{x}_{s-} < m_s\}}  ds - \int_0^t j_d(s, \mathbf{x}_{s-}) D \varphi (\mathbf{x}_{s-}) \mathbf{1}_{ \{\mathbf{x}_{s-} \leq x_1(s)\}} ds.
	\end{aligned}
	$$
	The process $M(\varphi, X)$ is progressively measurable w.r.t. the canonical filtration $\F$. For every constant $p > 0$, we introduce an $\F$-stopping time and the corresponding stopped canonical process
	$$
	\tau_p:= \inf \{ t \geq 0: |X_t| \geq p \mbox{ or } |X_{t-}| \geq p \}, ~~~X_t^p:= X_{t \wedge \tau_p}.
	$$
	Denote further $J(\mathbf{x}):=\{t >0: \Delta \mathbf{x}(t) \neq 0 \}$, $V(\mathbf{x}):=\{ a >0: \tau_a (\mathbf{x}) < \tau_{a^+}(\mathbf{x}) \}$ and 
	$$
	V'(\mathbf{x}):= \{ a>0: \tau_a (\mathbf{x}) \in J(\mathbf{x}) \mbox{ and } |\mathbf{x}(\tau_a (\mathbf{x}))| =a \}.
	$$
	By extracting a subsequence, we can suppose w.l.o.g. that $\P^n \to \P^0$ weakly. To prove that $\P^0$ is a weak solution to SDE \eqref{eq:SDE_main}, it is enough to show that $(M_t(\varphi, X))_{t \in [0,1]}$ is a local martingale under $\P^0$ for every $\varphi \in C_b^1(\R)$. Since the functions $j_d$ and $j_u$ are locally Lipchitz, but not necessarily uniformly bounded, as in the proof of Henry-Labord\`ere and Touzi \cite[Theorem 3.7]{HLTouzi16}, we need to adapt the localization technique in Jacod and Shiryaev \cite{JacodShiryaev}.
	
	First, since $\P^n$ is induced by the Markov chain $(X^n, \P^{*,n})$ for all $n \geq 1$, we have 
	$$
	\E^{\P^n}_{t_k^n} [ \varphi(X_{t^n_{k+1}})-\varphi(X_{t^n_k}) ] = \alpha_u + \alpha_d +  \alpha_{\tilde d},
	$$
	where
	\begin{align*}
\alpha_u &:= \E^{\P^n}_{t_k^n} [ \{ \varphi(X_{t^n_k} + J_u^{\epsilon^n_k} (t^n_k, X_{t^n_k}) ) -  \varphi(X_{t^n_{k}}) \} \frac{J_{d}^{\epsilon^n_k}}{J_d^{\epsilon^n_k}+J_u^{\epsilon^n_k}} \mathbf{1}_{\{x_1^{\epsilon_k^n}(t^n_k) < X_{t_k^n} < m^{\epsilon_k^n}(t^n_k)\}} ],\\
\alpha_d &:= \E^{\P^n}_{t_k^n} [ \{ \varphi(X_{t^n_k} - J_d^{\epsilon^n_k} (t^n_k, X_{t^n_k}) ) -  \varphi(X_{t^n_{k}}) \} \frac{J_{u}^{\epsilon^n_k}}{J_d^{\epsilon^n_k}+J_u^{\epsilon^n_k}} \mathbf{1}_{\{x_1^{\epsilon_k^n}(t^n_k) < X_{t_k^n} < m^{\epsilon_k^n}(t^n_k)\}} ],\\
\alpha_{\tilde d}&:=  \E^{\P^n}_{t_k^n} [ \{ \varphi(X_{t^n_k} -J_d^{\epsilon^n_k} (t^n_k, X_{t^n_k}) ) -  \varphi(X_{t^n_{k}}) \} \mathbf{1}_{\{X_{t_k^n} \leq x_1^{\epsilon_k^n}(t^n_k)\}} ],
	\end{align*}
	with $\epsilon_k^n:= t^n_{k+1} - t^n_k$ (also recall that, under the present assumptions (Dispersion Assumption in particular), $x^{\epsilon^n_k}_1(t^n_k)$ is the unique regime-switching boundary).
	
	By the continuity of $(\epsilon, t) \mapsto ( m^{\epsilon}(t), x_1^{\epsilon}(t))$ and noticing that the density $f(t,x) \in C^3_b(E)$, we have that for $\epsilon^n_k$ small and for all $s \in [t_k^n, t_{k+1}^n]$, 
	\begin{equation*} \label{eq:indicator_estimate}
	\begin{aligned}
	& |\E^{\P^n}_{t_k^n} [ \mathbf{1}_{\{x_1^{\epsilon_k^n}(t^n_k) < X_{t_k^n} < m^{\epsilon_k^n}(t^n_k)\}}  - \mathbf{1}_{\{ x_1(s) < X_{s} < m_s\}} ] | \\
	\leq & |\int_{x_1^{\epsilon_k^n}(t^n_k)}^{ x_1(s)} f(t,x) dx | + |\int_{m^{\epsilon_k^n}(t^n_k)}^{ m_s} f(t,x) dx | = O( \rho_0(\epsilon_k^n)\lor \rho_1(\epsilon_k^n)).
	\end{aligned}
	\end{equation*}
	
Then using Lemma \ref{lemm:asymptotic_expansion} and the arguments of \cite[Theorem 3.7]{HLTanTouzi} we have that
			\begin{align*}
		\alpha_u &= \E^{\P^n}_{t_k^n} \left[  \int_{t_k^n}^{t_{k+1}^n}  [ \varphi(X_s + j_u(s, X_s) ) - \varphi(X_s) ] \frac{j_d}{j_u}(s, X_s) \mathbf{1}_{ \{x_1(s) < X_s < m_s)\}}  ds \right],\\
		\alpha_d &= -\E^{\P^n}_{t_k^n} [  \int_{t_k^n}^{t_{k+1}^n} j_d(s, X_{s}) D \varphi (X_{s}) \mathbf{1}_{ \{x_1(s) < X_{s}< m_s\}} ds ] + O(\epsilon_k^n(\epsilon_k^n \lor \rho_0(\epsilon_k^n)\lor \rho_1(\epsilon_k^n))).
		\end{align*}
		Similarly, using Lemma \ref{lemm:asymptotic_expansion_2} we have that
		$$
		\alpha_{\tilde d} = -\E^{\P^n}_{t_k^n} \left[  \int_{t_k^n}^{t_{k+1}^n} j_d(s, X_{s}) D \varphi (X_{s}) \mathbf{1}_{  \{X_{s} \leq x_1(s)\}} ds \right] + O(\epsilon_k^n(\epsilon_k^n \lor \rho_0(\epsilon_k^n)\lor \rho_1(\epsilon_k^n))).
		$$

	Now let $0\leq s < t \leq 1$, $p \in \N$, $\phi_s(X_\cdot)$ be a $\Fc_s$-measurable bounded random variable on $\Om$ such that $\phi: \Om \to \R$ is continuous under the Skorokhod topology. Then 
	$$
	\E^{\P^n} [ \phi_s(X_\cdot) (M_{t \wedge \tau_p}(\varphi, X) - M_{s \wedge \tau_p}(\varphi, X) ) ] \leq C_p ( |\pi_n| \lor \rho_0(|\pi_n|) \lor \rho_1(|\pi_n|) ).
	$$
	for some constant $C_p > 0$. The remaining arguments are those of \cite[Theorem 3.7]{HLTanTouzi}, and thus omitted. 
\end{proof}

\subsection{Optimality of the local L{\'e}vy process}\label{sec:optimalityDual}
In this section, we establish the optimality of the previously obtained limiting process to the continuous-time primal problem $\mathbf{P}_{\infty}$.

We will need the optimal dual strategy (in discrete time one-period case) which is given in Lemma \ref{Lemma:OptimalDual}; in the following, we denote $h^{\epsilon}:=\hat h^{\mu_t, \mu_{t+\epsilon}}, \varphi^{\epsilon}:=\hat \varphi^{\mu_t, \mu_{t+\epsilon}}, \psi^{\epsilon}:=\hat \psi^{\mu_t, \mu_{t+\epsilon}}$.

Now let us introduce the optimal dual strategies for the continuous-time transport problem. First, we define the dynamic strategy $h^*$ by
$$
h^*(t,x) =0, ~~\mbox{when } x \leq x_1(t),
$$
$$
\partial_x h^*(t,x) = \frac{c_x(x,T_u(t,x)) - c_x(x, x)}{j_u(t,x)},\quad  x_1(t) < x < m_t,
$$
$$\lim_{x\downarrow x_1(t)}h^*(t,x)=0,
$$
$$
h^*(t,x) = h^{*}(t,T_u^{-1}(t,x)) - c_y( T_u^{-1}(t,x), x ), ~~\mbox{when } x \geq m_t.
$$

We remark that from the above definition, for all $t\in [0,1]$, $x \mapsto h^*(t,x)$ is continuous both at $x_1(t)$ and $m_t$.  In addition, $\psi^*$ is defined, up to a constant, by
$$
\partial_x \psi^*(t,x):= - h^*(t,x),\quad (t,x)\in E=\{ (t,x): t \in [0,1], x \in (l_t, r_t) \}.
$$
Now using the same arguments as \cite[Corollary 3.10]{HLTanTouzi} together with the fact that $h^*(t,x)=0$ for $t\in[0,1]$ and $x\leq x_1(t)$, we have that $\psi^* \in C^{1,1}(E)$.

For the (static part of the) dual limiting strategy, let $\gamma^*(dt) := \delta_{\{ 0 \}} (dt) +  \delta_{\{ 1 \}} (dt) + Leb(dt)$ be a finite measure on $[0,1]$, where $Leb$ denotes the Lebesgue measure on $[0,1]$.

 One further defines $\lambda_0^*$ and $\bar{\lambda}_0^*$ by $\lambda_0^*(0,x):= \psi^*(0,x)$, $\lambda_0^*(1,x):= \psi^*(1,x)$, $\bar{\lambda}_0^*(0,x):= | \psi^*(0,x) |$, and $\bar{\lambda}_0^*(1,x):= \sup_{t \in [0,1]} | \psi^*(t,x) |$; and for all $(t,x) \in (0,1) \times \R$,
\begin{align*}
\lambda_0^*(t,x):= \partial_t \psi^*(t,x) &+ \mathbf{1}_{\{x_1(t) < x \leq m_t\}} \left( \partial_x \psi^* j_d + \nu [ \psi^* - \psi^*(\cdot, T_u) + c(\cdot, T_u) ] \right)\\&+ \mathbf{1}_{\{x \leq x_1(t)\}}  \partial_x \psi^* j_d,
\end{align*}
\begin{align*}
\bar{\lambda}_0^*:= | \partial_t \psi^*(t,x) &+ \mathbf{1}_{\{x_1(t) < x < m_t\}} \left( \partial_x \psi^* j_d + \nu [ \psi^* - \psi^*(\cdot, T_u)] \right)| + \mathbf{1}_{\{x_1(t) < x < m_t\}} \nu [|c(\cdot, T_u) |] \\&+ \mathbf{1}_{x \leq x_1(t)}  |\partial_x \psi^* j_d|.
\end{align*}
Finally we define $\lambda^*(x,dt):=\lambda_0^*(t,x) \gamma^*(dt)$ and $\bar{\lambda}^*(x,dt):=\bar{\lambda}_0^*(t,x) \gamma^*(dt)$.

The next result shows that the limiting process of the previous sections solves the continuous-time supermartingale optimal transport problem. Let $H^*$ be the $\F$-predictable process on $\Om$ defined by
$$
H_t^*:=h^*(t,X_{t-}), \quad t \in [0,1].
$$

\begin{Theorem} \label{Thm:optimality}
Suppose that Assumptions \ref{Assump:CostFunction}, \ref{Assump:1}, \ref{Assump:2}, \ref{Assump:3} hold, and that 
\begin{equation} \label{eq:optimality_integrability}
\mu(\bar{\lambda}^*)=\int_0^1 \int_{\R} \bar{\lambda}^*_0 (t,x) \mu_t(dx) \gamma^*(dt) < \infty.
\end{equation} 
Then the supermartingale transport problem \eqref{eq:Primal_Full} is solved by the local L{\'e}vy process \eqref{eq:SDE_main}. Moreover, we have that $(H^*, \lambda^*) \in \Dc_{\infty} (\mu)$ and the following duality is valid:
$$
\E^{\P^0} [ C(X_\cdot) ] = \mathbf{P}_{\infty}(\mu) =  \mathbf{D}_{\infty}(\mu) = \mu(\lambda^*),
$$
and the optimal value is given by 
$$
\mu(\lambda^*) = \int_0^1 \int_{x_1(t)}^{m_t} \frac{j_d(t,x)}{j_u(t,x)} c(x, x+j_u(t,x)) f(t,x) dx dt.
$$
\end{Theorem}

In order to prove Theorem \ref{Thm:optimality}, we consider the following partition of $[0,1]$: $\pi_n=(t_k^n)_{0 \leq k \leq n}$, where $t_k^n: =k \epsilon$ with $\epsilon=\frac{1}{n}$. In the following, we shall simplify the notation $t_k^n$ to $t_k$.  Remind that under every $\P^n$, we have the following super-replication inequality:
\begin{equation} \label{eq:Discrete_Superhedge}
\sum_{k=0}^{n-1} \left( \varphi^{\epsilon}(t_k, X_{t_k}) + \psi^{\epsilon}(t_k, X_{t_{k+1}}) \right) + \sum_{k=0}^{n-1} h^{\epsilon} (t_k, X_{t_k}) (X_{t_{k+1}}-X_{t_k}) \geq \sum_{k=0}^{n-1} c(X_{t_k}, X_{t_{k+1}}).
\end{equation}
We further define $\Psi^*: \Om \to \R$ by
$$
\begin{aligned}
& \Psi^*(\mathbf x): = \psi^*(1,\mathbf{x}_1) - \psi^*(0,\mathbf{x}_0) - \int_0^1 \large( \partial_t \psi^*(t, \mathbf{x}_t) + j_d(t, \mathbf{x}_t) \mathbf{1}_{\{\mathbf{x}_t <m_t\}}   
\partial_x \psi^*(t, \mathbf{x}_t)  \large) dt  \\&+ \int_0^1 \frac{j_d(t,\mathbf{x}_t)}{j_u(t,\mathbf{x}_t)} \mathbf{1}_{\{x_1(t) < \mathbf{x}_t <m_t\}} \left( \psi^*(t, \mathbf{x}_t) -  \psi^*(t, \mathbf{x}_t + j_u(t,\mathbf{x}_t)) + c(\mathbf{x}_t, \mathbf{x}_t + j_u(t,\mathbf{x}_t)) \right) dt.
\end{aligned}
$$

To prove Theorem \ref{Thm:optimality}, we need several auxiliary results. The proof of the following lemma relies on Lemmas \ref{lemm:asymptotic_expansion} and \ref{lemm:asymptotic_expansion_2}, but otherwise is the same as the proof of \cite[Lemma 6.7]{HLTanTouzi}, and thus we omit the details.
\begin{Lemma} \label{lemm:local_uniform_converge}
Under Assumptions \ref{Assump:CostFunction}, \ref{Assump:1}, \ref{Assump:2}, we have that 
\begin{align*}
T_u^{\epsilon} \mathbf{1}_{\{x_1^{\epsilon}(t) < x < m^{\epsilon}_t\}} \to T_u \mathbf{1}_{\{x_1(t) < x < m_t\}},\ h^{\epsilon} \to h^*,\ \partial_t \psi^{\epsilon} \to \partial_t \psi^*,\ \psi^{\epsilon} \to \psi^*
\end{align*}
locally uniformly on $E=\{ (t,x): t \in [0,1], x \in (l_t, r_t) \}$.
\end{Lemma}

\begin{Lemma} \label{Lemma:Convergence_Static}
Suppose Assumptions \ref{Assump:CostFunction}, \ref{Assump:1}, \ref{Assump:2} hold. Then for every c{\`a}dl{\`a}g path $\mathbf{x} \in \D([0,1], \R)$ taking value in $(l_1,r_1)$, we have
$$
\lim_{n \to \infty} \sum_{k=0}^{n-1} \left( \varphi^{\epsilon} (t_k, \mathbf{x}_{t_k})+ \psi^{\epsilon} (t_k, \mathbf{x}_{t_{k+1}}) \right) \to \Psi^*(\mathbf{x}), ~~\mbox{as } \epsilon \to 0.
$$
\end{Lemma}
\begin{proof} The arguments are almost identical to those of \cite[Lemma 6.9]{HLTanTouzi}, and thus we only briefly sketch them. First note that 
$$
\begin{aligned}
\sum_{k=0}^{n-1} \left( \varphi^{\epsilon} (t_k, \mathbf{x}_{t_k})+ \psi^{\epsilon} (t_k, \mathbf{x}_{t_{k+1}}) \right) &= \sum_{k=0}^{n-1} \left( \varphi^{\epsilon} (t_k, \mathbf{x}_{t_k})+ \psi^{\epsilon} (t_k, \mathbf{x}_{t_{k}}) \right) + \psi^{\epsilon} (t_{n-1}, \mathbf{x}_{1}) \\
&+\sum_{k=1}^{n-1} \left( \psi^{\epsilon} (t_{k-1}, \mathbf{x}_{t_{k}}) - \psi^{\epsilon} (t_k, \mathbf{x}_{t_k})\right) - \psi^{\epsilon} (0, \mathbf{x}_{0}).
\end{aligned}
$$
Then by Lemma \ref{lemm:local_uniform_converge} and the arguments of \cite[Lemma 6.9]{HLTanTouzi}, as $\epsilon\to0$, we have that $\psi^{\epsilon} (t_{n-1}, \mathbf{x}_{1}) \to \psi^{*} (1, \mathbf{x}_{1})$,
$$
\sum_{k=1}^{n-1} \left( \psi^{\epsilon} (t_{k-1}, \mathbf{x}_{t_{k}}) - \psi^{\epsilon} (t_k, \mathbf{x}_{t_k})\right) \to - \int_0^1 \partial_t \psi^*(t, \mathbf{x}_t) dt,
$$
and $\sum_{k=0}^{n-1} \left( \varphi^{\epsilon} (t_k, \mathbf{x}_{t_k})+ \psi^{\epsilon} (t_k, \mathbf{x}_{t_{k}}) \right)\mathbf{1}_{\{x_1^\epsilon(t_k)<\mathbf{x}_{t_k}<m^\epsilon(t_k)\}}$ converges to
$$
\begin{aligned}
& \int_0^1 -\partial_x \psi^*(t, \mathbf{x}_{t}) j_d (t, \mathbf{x}_{t}) dt \\
& + \int_0^1 \frac{j_d (t, \mathbf{x}_{t})}{j_u (t, \mathbf{x}_{t})} \mathbf{1}_{\{x_1(t) < \mathbf{x}_{t} <m_t\}} \left( \psi^*(t, \mathbf{x}_{t}) - \psi^*(t, \mathbf{x}_{t}+j_u(t, \mathbf{x}_{t})) + c(\mathbf{x}_{t},\mathbf{x}_{t}+j_u(t, \mathbf{x}_{t})) \right) dt.
\end{aligned}
$$
Finally, using similar arguments, and noting that $c(x,x)=0$ for all $x \in \R$, we have that 
$$
\sum_{k=0}^{n-1} \left( \varphi^{\epsilon} (t_k, \mathbf{x}_{t_k})+ \psi^{\epsilon} (t_k, \mathbf{x}_{t_{k}}) \right)\mathbf{1}_{\{\mathbf{x}_{t_k}\leq x_1^\epsilon(t_k)\}}\to- \int_0^1 j_d(t, \mathbf{x}_t) \mathbf{1}_{ \{\mathbf{x}_t \leq x_1(t)\}}\partial_x \psi^*(t, \mathbf{x}_t)  dt,
$$
which concludes the proof.
\end{proof}

\ \\

\begin{Lemma} \label{Lemma:optimality}
Suppose that Assumptions \ref{Assump:CostFunction}, \ref{Assump:1}, \ref{Assump:2}, \ref{Assump:3} hold and $\mu(\bar{\lambda}^*) < \infty$. Then for the probability measure $\P^0$ given in Theorem \ref{Thm:main}, we have that 
$$
\E^{\P^0} [ C(X_\cdot) ] = \E^{\P^0} [ \Psi^*(X_\cdot) ] = \mu(\lambda^*) =  \int_0^1 \int_{x_1(t)}^{m_t} \frac{j_d(t,x)}{j_u(t,x)} c(x, x+j_u(t,x)) f(t,x) dx dt.
$$
\end{Lemma}
\begin{proof}
The arguments are almost identical to those of \cite[Lemma 6.10]{HLTanTouzi} (adapted to the right-curtain coupling situation). We only need to adapt various definitions to include the seupermartingale regions governed by $t\mapsto x_1(t)$. 

By It\^o's formula, the following process is a local martingale:
$$
\begin{aligned}
& \psi^*(t, X_t) - \psi^*(0,X_0) - \int_0^1 \large( \partial_t \psi^*(t, X_t) + \left(j_d(t, X_t) \mathbf{1}_{\{x_1(t) < X_t < m_t\}} +j_d(t, X_t) \mathbf{1}_{ \{X_t \leq x_1(t)\}} \right) \\
& \partial_x \psi^*(t, X_t)  \large) dt+ \int_0^1 \frac{j_d(t,x_t)}{j_u(t,x_t)} \mathbf{1}_{\{x_1(t) < X_t < m_t\}} \left( \psi^*(t, X_t) -  \psi^*(t, X_t + j_u(t, X_t))  \right) dt.
\end{aligned}
$$
Moreover, using the fact that $\mu(\bar{\lambda}^*) < \infty$ together with the dominated convergence theorem, we have that 
$$
 \E^{\P^0} [ \Psi^*(X_\cdot) ] =  \int_0^1 \int_{x_1(t)}^{m_t} \frac{j_d(t,x)}{j_u(t,x)} c(x, x+j_u(t,x)) f(t,x) dx dt.
$$
To compute $\E^{\P^0} [ C(X_\cdot) ]$, noticing that $[X]_t^c=0$, $\P^0$-a.s., and the process
$$
Y_t: = \sum_{s \leq t} | c(X_{s-}, X_s) | - \int_0^t | c(X_{s-}, X_{s-} + j_u(t,X_{s-})| \frac{j_d(t,X_{s-})}{j_u(t,X_{s-})} \mathbf{1}_{\{x_1(t) < X_t < m_t\}} dt
$$ 
is a local martingale. The remainder of the proof follows the arguments of \cite[Lemma 6.10]{HLTouzi16}.
\end{proof}

In the following lemma, we establish the continuous limit of the dynamic part. The proof relies on Lemma \ref{lemm:local_uniform_converge}, but otherwise is identical to the proof of \cite[Lemma 6.11]{HLTanTouzi}, and thus omitted.
\begin{Lemma} \label{Lemma:Convergence_Rynamic}
Suppose Assumptions \ref{Assump:1}, \ref{Assump:2}, \ref{Assump:3} hold. Then the following convergence in probability (as $\epsilon\to0$) holds under every supermartingale measure $\P \in \Sc_{\infty}$:
$$
\sum_{k=1}^{n-1} h^{\epsilon} (t_k, X) (X_{t_{k+1}} - X_{t_k}) \to \int_0^1 h^*(t, X_{t-}) d X_t.
$$
\end{Lemma}
Finally, we provide the proof of Theorem \ref{Thm:optimality}.
\begin{proof}[Proof of  Theorem \ref{Thm:optimality}]

Using the discrete superhedging duality \eqref{eq:Discrete_Superhedge}, together with the convergence results Lemma \ref{Lemma:Quadratic_Variation}, \ref{Lemma:Convergence_Static}, \ref{Lemma:Convergence_Rynamic}, we have that under every $\P \in \Sc_{\infty}$, $(\Psi^*,h^*)$ superhedges the continuous-time cost $C(X_\cdot)$:
$$
\Psi^*(X_\cdot) + \int_0^1 h^*(t, X_{t-}) d X_t \geq \int_0^1 \frac{1}{2} c_{yy} (X_t, X_t) d [X]_t^c + \sum_{0 \leq t \leq 1} c(X_{t-}, X_t), \quad \P\mbox{-a.s.}
$$
By the weak duality, 
$$
\E^{\P^0} [ C(X_\cdot) ] \leq \mathbf{P}_{\infty} (\mu) \leq \mathbf{D}_{\infty} (\mu) \leq \mu(\lambda^*).
$$
Then Lemma \ref{Lemma:optimality} concludes the proof. 
\end{proof}

\section{Examples} \label{sec:examples}

In this section, we will consider three examples of supermartingales matching a given PCOCD. In contrary to the martingale case of \cite{JuilletPCOC,HLTanTouzi}(that covers left-curtain and right-curtain couplings), in the supermartingale context, the increasing and decreasing couplings provide significantly different dynamics in continuous-time.

\subsection{Uniform distribution with bounded support} \label{sec:uniform}
In this section the family of marginal distributions $(\mu_t)_{t\in[0,1]}$ satisfies
\begin{equation}\label{eq:uniformMarginals}
\mu_t=\frac{1}{e^t+e^{2t}} \lambda_{[-e^{2t}, e^t ]}, \quad t \in [0,1].
\end{equation}
(This example was also considered in the martingale case by Juillet \cite{JuilletPCOC}.) Our goal is to explicitly determine the quantities that define the SDE given in Theorem \ref{Thm:main}.

Note that $\ell_t=-e^{2t}$ and $r_t=e^t$ for all $t\in[0,1]$. Furthermore, $(\mu_t)_{t\in[0,1]}$ satisfies Assumption \ref{Assump:1} with $m_\epsilon(t)=\ell_t$ and $m^\epsilon(t)=r_t$, for each $t \in [0,1)$ and $\epsilon \in (0,1-t]$. Recall that $m_\epsilon,m^\epsilon$ maximizes and minimizes, respectively, the function $x\mapsto F(t+\epsilon,x)-F(t,x)$. Observe that $m_\epsilon,m^\epsilon$ are independent of $\epsilon>0$. By direct computation we have that, for all $t\in(0,1)$, $\partial_t F(t,\cdot)$ is minimized and maximized at $m_t=e^t$ and $\tilde{m}_t=-e^{2t}$. Then $m^0(t):=m_t=\lim_{\epsilon\downarrow 0}m^\epsilon(t)=e^t$ for all $t\in[0,1]$. It is now easy to see that $(\mu_t)_{t\in[0,1]}$ satisfies Assumption \ref{Assump:2}.
However, since $t\mapsto r_t$ is not constant and $m_t=r_t$, Assumption \ref{Assump:3} does not hold. For this reason, most of the results of Section \ref{sec:SDE} (that lead to Theorem \ref{Thm:main}) cannot be applied, and thus we will prove the relevant statements by direct calculations.

Now fix $t \in [0,1)$ and $\epsilon \in (0,1-t]$, and consider the decreasing supermartingale coupling $\hat\P^{\mu_t,\mu_{t+\epsilon}}$ of $\mu_t$ and $\mu_{t+\epsilon}$ (see Appendix \ref{sec:Decreasing_SMOT_one_period_primal}). Since, Assumption \ref{Assump:1} is satisfied, the phase transition point, that separates the martingale and supermartingale regions, is unique; see Lemma \ref{lem:transition_dispersion}. The transition threshold is denoted by
$$
x_1^\epsilon(t):=x_1^{\mu_t,\mu_{t+\epsilon}},
$$
see \eqref{eq:x_1Appendix}.
\begin{Lemma}\label{Lemma:trans_uniform_decreasing}
	If $(\mu_t)_{t\in[0,1]}$ is given by \eqref{eq:uniformMarginals}, then the phase transition $x_1^{\epsilon}(t)$ is given by
	\begin{equation} \label{eq:transition_decreasing_uniform}
	x_1^{\epsilon} (t)= \frac{e^t ( e^{2(t+\epsilon)}-e^{t+\epsilon}+e^{t} ) + e^{2t} ( e^{t} - 2 e^{t+\epsilon})}{e^{t+\epsilon} + e^{2(t+\epsilon)} - e^t - e^{2t}}.
	\end{equation}
	Furthermore, the limit $x_1(t):=\lim_{\epsilon\downarrow 0}x_1^\epsilon(t)$ exists and is given by

	\begin{equation} \label{eq:transition_decreasing_uniform_cont}
	x_1(t)=  \frac{-e^{t} }{1 + 2 e^{t} } \in (-e^{2t}, e^t),\quad t\in[0,1].
	\end{equation}
\end{Lemma}
\begin{proof}
To determine $x_1^{\epsilon}(t)\in(\ell_t,r_t)$ we use the fact that under the decreasing supermartingale coupling $\hat\P^{\mu_t,\mu_{t+\epsilon}}$, $\mu\lvert_{[x_1^\epsilon(t),r_t)}$ is mapped to $\nu\lvert_{[y_1^\epsilon(t),r_{t+\epsilon})}$ in a martingale way, where $y_1^\epsilon(t)=y_1^{\mu_t,\mu_{t+\epsilon}}$ (see \eqref{eq:y_1Appendix}). In particular, the pair $(x_1^\epsilon(t),y_1^\epsilon(t))$ is uniquely determined by the mass and mean preservation condition
\begin{equation}
\int^{r_t}_{x_1^\epsilon(t)}z^i\mu_t(dz)=\int^{r_{t+\epsilon}}_{y_1^\epsilon(t)}z^i\mu_{t+\epsilon}(dz),\quad i=0,1.
\end{equation}
Direct computation leads to
$$
\frac{e^t-x_1^{\epsilon}(t)}{e^t+e^{2t}} = \frac{e^{t+\epsilon}-y_1^{\epsilon}(t)}{e^{t+\epsilon}+e^{2(t+\epsilon)}},\quad
\frac{e^{2t}-(x_1^{\epsilon}(t))^2}{e^t+e^{2t}} = \frac{e^{2(t+\epsilon)}-(y_1^{\epsilon}(t))^2}{e^{t+\epsilon}+e^{2(t+\epsilon)}}.
$$
Solving the above equations gives
$$
x_1^{\epsilon} (t)= \frac{e^t ( e^{2(t+\epsilon)}-e^{t+\epsilon}+e^{t} ) + e^{2t} ( e^{t} - 2 e^{t+\epsilon})}{e^{t+\epsilon} + e^{2(t+\epsilon)} - e^t - e^{2t}},\quad y_1^{\epsilon} (t) = x_1^{\epsilon} (t) - e^{t+\epsilon} + e^t.
$$
By letting $\epsilon\downarrow 0$, direct calculation (using the L'H\^opital's rule) gives that $x_1(t) =  \frac{-e^{t} }{1 + 2 e^{t} }$ (note that $\lim_{\epsilon\downarrow 0 }y^\epsilon_1(t)=x_1(t)$), which concludes the proof.
\end{proof}

Even though not all the Assumptions \ref{Assump:1}, \ref{Assump:2}, \ref{Assump:3} are satisfied in the case of \eqref{eq:uniformMarginals}, the following result shows that the statement of Theorem \ref{Thm:main} remains true. 

\begin{Proposition}
 If $(\mu_t)_{t\in[0,1]}$ is given by \eqref{eq:uniformMarginals}, then the statement of Theorem \ref{Thm:main} holds, and the corresponding SDE is explicitly given by
	$$
	X_t = X_0 + \int_0^t\mathbf{1}_{\{X_{s-} > x_1(s)\} } j_u(s,X_{s-}) ( d N_s - \nu_s ds )    -\int_0^t \mathbf{1}_{ \{X_{s-}\leq x_1(s)\}} \tilde{\nu}_s ds,
	$$
	where $(N_s)_{0 \leq s \leq 1}$ is a unit size jump process with predictable compensated process $(\nu_s)_{0 \leq s \leq 1}$ given by  $\nu_s:= \frac{j_d}{j_u} ( s, X_{s-} )$, with $j_d(s,x)= \frac{1}{2} \frac{e^{s}-x}{1+e^{s}}(1+2e^{s})$, $j_u(s,x) = e^s -x$, and $\tilde{\nu}_s:= \frac{e^{2s}-x(1+2e^s)}{1+e^s}$.
	
	In addition, if Assumption \ref{Assump:CostFunction} and the integrability condition \eqref{eq:optimality_integrability} are satisfied, then the above process is the optimal solution for the 
	primal problem $\eqref{eq:Primal_Full}$ and the conclusion of Theorem \ref{Thm:optimality} holds.
\end{Proposition}
\begin{proof}
Since, for $t\in[0,1]$, we have that $m_t=r_t$, it follows that $\mathbf{1}_{\{x_1(t)<X_{t-}<m_t\}}=\mathbf{1}_{\{X_{t-}>x_1(t)\}}$ a.s. We now determine $(j_d,j_u)$.

For $t \in [0,1)$ and $\epsilon \in (0, 1-t]$, we consider the one-period decreasing supermartingale coupling $\hat\P^{\mu_t,\mu_{t+\epsilon}}$, which is supported on the graphs of two functions $T_u^{\epsilon}(t,\cdot)=T_u^{\mu_t,\mu_{t+\epsilon}}(\cdot)$, $T_d^{\epsilon}(t,\cdot)=T_d^{\mu_t,\mu_{t+\epsilon}}(\cdot)$; see Appendix \ref{sec:Decreasing_SMOT_one_period_primal}. On $(x_1^\epsilon(t),m^\epsilon(t)=r_t)$, $(T_d^\epsilon(t,\cdot),T_u^\epsilon(t,\cdot)$ can be determined from the mass and mean preservation condition
\begin{equation*}
\int^{T^\epsilon_u(t,x)}_xz^i\mu_t(dz)=\int^{T^\epsilon_u(t,x)}_{T^\epsilon_d(t,x)}z^i\mu_{t+\epsilon}(dz),\quad i=0,1.
\end{equation*}
Note that, since $T_u^\epsilon(t,x)\geq m^\epsilon(t)=r_t$ for $x\leq r_t$, we have that $\int^{T^\epsilon_u(t,x)}_xz^i\mu_t(dz)=\int^{r_t}_xz^i\mu_t(dz)$ for $i=0,1$. Then direct calculation gives that
$$
\frac{e^t-x}{e^t+e^{2t}} =  \frac{T_u^{\epsilon}(t,x)-T_d^{\epsilon}(t,x)}{e^{t+\epsilon}+e^{2(t+\epsilon)}},\quad T_u^{\epsilon}(t,x)  = x +e^t - T_d^{\epsilon}(t,x),\quad x\in(x^\epsilon_1(t),e^t).
$$
It follows that
$$
x-T^\epsilon_d(t,x)=\frac{1}{2}\frac{e^t-x}{e^t+e^{2t}}\left[e^{t+\epsilon}-e^t+e^{2(t+\epsilon)}-e^{2t}\right],\quad x\in(x^\epsilon_1(t),e^t).
$$
This further leads to 
$$
j_d(t,x) := \lim_{\epsilon \downarrow 0}j_d^{\epsilon}(t,x) := \lim_{\epsilon \downarrow 0} \frac{x-T_d^{\epsilon}(t,x)}{\epsilon} = \frac{1}{2} \frac{e^{t}-x}{1+e^{t}}(1+2e^{t}),\quad x\in(x_1(t),e^t),
$$
$$
j_u(t,x) := \lim_{\epsilon \downarrow 0} ( T_u^{\epsilon}(t,x) -x ) = e^t -x,\quad x\in(x_1(t),e^t).
$$
Then also
$$
\frac{j_d}{j_u}(t,x)=\frac{1}{2}\frac{1+2e^t}{1+e^t},\quad x\in(x_1(t),e^t).
$$

If we define $T_u$ by $T_u(t,x):=\lim_{\epsilon\downarrow 0}T^\epsilon_u(t,x)$, then $T_u(t,x)=e^t$ for $x\in(x_1(t),e^t)$, which is clearly non-decreasing and continuous.

The definition of $j_d$ on $[0,1)\times(\ell_t,x_1(t)]$ remains the same:
$$
j_d(t,x) := \lim_{\epsilon \downarrow 0}j_d^{\epsilon}(t,x) = \lim_{\epsilon \downarrow 0} \frac{x-T_d^{\epsilon}(t,x)}{\epsilon},\quad x\in (-e^{2t},x_1(t)].
$$
The explicit form is recovered by using the fact that $T_d^{\epsilon}(t,x)=F^{-1}(t+\epsilon,F(t,x))$ for $x\leq x_1^\epsilon(t)$. Direct computation gives that
$$
j_d(t,x)=\frac{e^{2t}-x(1+2e^t)}{1+e^t},\quad x\in (-e^{2t},x_1(t)].
$$

It follows that the statements of Lemmas \ref{lemm:asymptotic_expansion}, \ref{lemm:asymptotic_expansion_2}, \ref{Lemma:character_limit}(ii) (and consequently Theorem \ref{Thm:main}) hold, which proves the first part of the present proposition.

The second claim is established by following the arguments of the proof of Theorem \ref{Thm:optimality}. In particular we need continuity of $T_u$, Lemmas \ref{Lemma:character_limit}(ii) and \ref{lemm:local_uniform_converge}. However, these results (by direct computation) also hold in the present setting, 
\end{proof}

\paragraph{Increasing supermartingale coupling}
Let $(\mu_t)_{t\in[0,1]}$ be as in \eqref{eq:uniformMarginals}. Fix $t \in [0,1)$ and $\epsilon \in (0,1-t]$, and denote by $\bar\P^{\mu_t,\mu_{t+\epsilon}}$ the increasing supermartingale coupling of $\mu_t$ and $\mu_{t+\epsilon}$ (see Nutz and Stebegg \cite{NutzStebegg.18} and Bayraktar et al. \cite{BayDengNorgilas}).

The main difference between $\bar\P^{\mu_t,\mu_{t+\epsilon}}$ and the decreasing supermartingale coupling $\hat\P^{\mu_t,\mu_{t+\epsilon}}$ is that $\bar\P^{\mu_t,\mu_{t+\epsilon}}$ is constructed by working from left to right and mapping $\mu_t\lvert_{(-\infty,x]}$ to the shadow measure $S^{\mu_{t+\epsilon}}(\mu_t\lvert_{(-\infty,x]})$, for each $x\in\R$. The construction of $\hat\P^{\mu_t,\mu_{t+\epsilon}}$ is similar, but works from right to left and thus considers the measures $(\mu_t\lvert_{[x,\infty)})_{x\in\R}$ and their corresponding `shadows' in $\mu_{t+\epsilon}$.

The main feature of $\bar\P^{\mu_t,\mu_{t+\epsilon}}$ is that there exists the unique point $x_1^\epsilon(t)\in(\ell_t,r_t)$ such that $\bar\P^{\mu_t,\mu_{t+\epsilon}}$ is a martingale on $(-\infty,x^\epsilon_1(t)]\times \R$ and strict supermartingale elsewhere. In particular,
\begin{itemize}
	\item on $(-\infty,x^\epsilon_1(t)]\times \R$, $\bar\P^{\mu_t,\mu_{t+\epsilon}}$ coincides with the left-curtain martingale coupling of $\mu_t\lvert_{(-\infty,x^\epsilon_1(t)]}$ and $S^{\mu_{t+\epsilon}}(\mu_t\lvert_{(-\infty,x^\epsilon_1(t)]})$;
	\item the support of $(\mu_t-\mu_t\lvert_{(-\infty,x^\epsilon_1(t)]})$ is strictly to the right of the support of  $(\mu_{t+\epsilon} - S^{\mu_{t+\epsilon}}(\mu_t\lvert_{(-\infty,x^\epsilon_1(t)]}))$;
	\item on $(x^\epsilon_1(t),\infty)\times \R$, $\bar\P^{\mu_t,\mu_{t+\epsilon}}$ coincides with the \textit{antitone} coupling of $(\mu_t-\mu_t\lvert_{(-\infty,x^\epsilon_1(t)]})$ and $(\mu_{t+\epsilon}-S^{\mu_{t+\epsilon}}(\mu_t\lvert_{(-\infty,x^\epsilon_1(t)]}))$;
\end{itemize}
see Bayraktar et al. \cite{BayDengNorgilas} for details. (In fact, these properties hold for any measures $\mu\leq_{cd}\nu$, and not necessarily $\mu_t\leq_{cd}\mu_{t+\epsilon}$ as in \eqref{eq:uniformMarginals}.) The left-curtain martingale coupling was introduced by Beiglb\"ock and Juillet \cite{BeiglbockJuillet:16} (it is the symmetric counterpart of the right-curtain martingale coupling; see Appendix \ref{sec:Brenier_MOT}). The aforementioned antitone coupling of two measures $\eta,\chi$ (with $\eta(\R)=\chi(\R)$ and atom-less $\eta$) is denoted by $\pi^{a,\eta,\chi}$ and given by
$$
\pi^{a,\eta,\chi}(dx,dy)=\eta(dx)\delta_{F^{-1}_\chi(\chi(\R)-F_\eta(x))}(dy).
$$
Note that $\pi^{a,\eta,\chi}$ is supported on a graph of a decreasing funtion $x\mapsto F^{-1}_\chi(\chi(\R)-F_\eta(x))$. 
\begin{Lemma}\label{Lemma:trans_uniform_increasing}
	If $(\mu_t)_{t\in[0,1]}$ is defined by \eqref{eq:uniformMarginals}, then the phase transition $x_1^{\epsilon}(t)$ is given by
	\begin{equation} \label{eq:transition_increasing_uniform}
x_1^{\epsilon} (t)= \frac{e^{3t} (1-e^{2\epsilon})+e^t(2e^\epsilon+e^t(1+e^\epsilon))}{1+e^\epsilon+e^t(1+e^{2\epsilon})},\quad t\in [0,1),\ \epsilon\in (0,1-t].
	\end{equation}
	In particular, under $\bar\P^{\mu_t,\mu_{t+\epsilon}}$, $\mu_t\lvert_{(\ell_t,x_1^\epsilon(t)]}$ is mapped to $\mu_{t+\epsilon}\lvert_{[y_1^\epsilon(t),r_t)}$ in a martingale way, where
	$$
	y_1^{\epsilon} (t) = x_1^{\epsilon} (t) - e^{t+\epsilon} -e^{2t},
	$$
	while $\mu_t\lvert_{(x^\epsilon_1(t),\infty)})$ is mapped to $\mu_{t+\epsilon}\lvert_{(-\infty,y^\epsilon_1(t))}$ in a strict supermartingale way.
	
	Furthermore, the limit $x_1(t):=\lim_{\epsilon\downarrow 0}x_1^\epsilon(t)$ exists and is given by
	
	\begin{equation} \label{eq:transition_increasing_uniform_cont}
	x_1(t)=  e^t=r_t,\quad t\in[0,1].
	\end{equation}
\end{Lemma}
\begin{proof}
	To determine $x_1^{\epsilon}(t)\in(\ell_t,r_t)$ we use the fact that $x_1^{\epsilon}(t)$ is the largest $x\in(\ell_t,r_t)$ for which $\mu_t\lvert_{(\ell_t,x]}\leq_{c} S^{\mu_{t+\epsilon}}(\mu_t\lvert_{(\ell_t,x]})$ and that $r_{S^{\mu_{t+\epsilon}}(\mu_t\lvert_{(\ell_t,x_1^\epsilon(t)]})}=\sup\{k\in supp(S^{\mu_{t+\epsilon}}(\mu_t\lvert_{(\ell_t,x_1^\epsilon(t)]}))\}=r_t$. Using that $(\mu_t)_{t\in[0,1]}$ are given by \eqref{eq:uniformMarginals} and te properties of the shadow measures, it follows that $\mu_t\lvert_{(\ell_t,x_1^\epsilon(t)]}\leq _c S^{\mu_{t+\epsilon}}(\mu_t\lvert_{(\ell_t,x_1^\epsilon(t)]})=\mu_{t+\epsilon}\lvert_{[y_1^\epsilon(t),r_{t+\epsilon})}$ for some (unique) $y_1^\epsilon(t)\in(\ell_{t+\epsilon},\ell_t)$. In particular, the pair $(x_1^\epsilon(t),y_1^\epsilon(t))$ is uniquely characterized by the mass and mean preservation condition
\begin{equation}\label{eq:mean-massExamplesIncreasing}
\int^{x_1^\epsilon(t)}_{\ell_t}z^i\mu_t(dz)=\int^{r_{t+\epsilon}}_{y_1^\epsilon(t)}z^i\mu_{t+\epsilon}(dz),\quad i=0,1.
\end{equation}
	Direct computation leads to
	$$
	\frac{x_1^{\epsilon}(t)+e^{2t}}{e^t+e^{2t}} = \frac{e^{t+\epsilon}-y_1^{\epsilon}(t)}{e^{t+\epsilon}+e^{2(t+\epsilon)}},\quad
	\frac{(x_1^{\epsilon}(t))^2-e^{4t}}{e^t+e^{2t}} = \frac{e^{2(t+\epsilon)}-(y_1^{\epsilon}(t))^2}{e^{t+\epsilon}+e^{2(t+\epsilon)}}.
	$$
	Solving the above equations gives
	$$
	x_1^{\epsilon} (t)= \frac{e^{3t} (1-e^{2\epsilon})+e^t(2e^\epsilon+e^t(1+e^\epsilon))}{1+e^\epsilon+e^t(1+e^{2\epsilon})},\quad y_1^{\epsilon} (t) = x_1^{\epsilon} (t) - e^{t+\epsilon} -e^{2t}.
	$$

	Letting $\epsilon\downarrow 0$ gives that $x_1(t) =  e^t$ and $y^\epsilon_1(t)=-e^{2t}$, which concludes the proof.
\end{proof}
It is well-known (see Henry-Labord\`re and Touzi \cite{HLTouzi16}) that the left-curtain martingale coupling (of two measures in convex order) is supported on the graphs of two functions that can be computed explicitly. The following gives an explicit representation of $\bar\P^{\mu_t,\mu_{t+\epsilon}}$ in terms of the supporting functions:

\begin{Lemma}\label{lem:Decreasing}
Let $(\mu_t)_{t\in[0,1]}$ be given by \eqref{eq:uniformMarginals}. Fix $t\in[0,1)$ and $\epsilon(0,1-t]$. Let $x_1^\epsilon(t)$ be as in Lemma \ref{Lemma:trans_uniform_increasing}.

Then $\bar\P^{\mu_t,\mu_{t+\epsilon}}$ is given by
\begin{equation*}\label{eq:increasing_representation}
\begin{aligned}
\bar\P^{\mu_t,\mu_{t+\epsilon}}(dx,dy)&=\mu(dx)\mathbf{1}_{[x_1^{\epsilon}(t), e^t)}(x)\delta_{\bar T_d^\epsilon(t,x)}(dy)\\
&+\mu(dx)\mathbf{1}_{(-e^{2t},x_1^{\epsilon}(t))}\left\{\frac{\bar T^\epsilon_u(t,x)-x}{\bar T^\epsilon_u(t,x)-\bar T_d^\epsilon(t,x)}\delta_{\bar T_d^\epsilon(t,x)}(dy)+\frac{x-\bar T_d^\epsilon(t,x)}{\bar  T^\epsilon_u(t,x)-\bar T_d^\epsilon(t,x)}\delta_{\bar T_u(x)}(dy)\right\},
\end{aligned}\end{equation*}
where
\begin{equation*}
\begin{aligned}
T^\epsilon_u(t,x)=\frac{1}{2}\left\{\frac{e^\epsilon(1+e^{t+\epsilon})(x+{e^{2t}})}{1+e^t}+x-e^{2t}\right\},\quad &x\in(-e^{2t},x^\epsilon_1(t)),\\
T^\epsilon_d(t,x)=x-e^{2t}-T^\epsilon_u(t,x),\quad &x\in(-e^{2t},x^\epsilon_1(t)),\\ T^\epsilon_d(t,x)=\left(e^{t+\epsilon}+e^{2(t+\epsilon)}\right)\left(1-\frac{x+e^{2t}}{e^t+e^{2t}}\right)-e^{2t}\quad &x\in [x^\epsilon_1(t),e^t).
\end{aligned}
\end{equation*}
\end{Lemma}
\begin{proof}
	Since, on $(\ell_t,x^\epsilon_1(t))\times(y^\epsilon_1(t),r_{t+\epsilon})$, $\bar\P^{\mu_t,\mu_{t+\epsilon}}$ coincides with the left-curtain martingale coupling of $\mu_t\lvert_{(\ell_t,x^\epsilon_1(t))}$ and $\mu_{t+\epsilon}\lvert_{(y^\epsilon_1(t),r_{t+\epsilon})}$, the supporting functions $(T_d^\epsilon(t,\cdot),T_u^\epsilon(t,\cdot))$
can be determined from the mass and mean preservation condition
	\begin{equation*}
	\int^{x}_{-e^{2t}}z^i\mu_t(dz)=\int^{T^\epsilon_u(t,x)}_{T^\epsilon_d(t,x)}z^i\mu_{t+\epsilon}(dz),\quad i=0,1.
	\end{equation*}
The direct computation gives that
$$
\frac{x+e^{2t}}{e^t+e^{2t}}=\frac{T_u^\epsilon(t,x)-T_d^\epsilon(t,x)}{e^{t+\epsilon}+e^{2(t+\epsilon)}},\quad T_d^\epsilon(t,x)=x-e^{2t}-T_u^\epsilon(t,x),\quad x\in(-e^{2t},x_1^\epsilon(t)).
$$
Solving for $T_u^\epsilon(t,x)$ (which is easy) gives the required representations of the supporting functions on $(-e^{2t},x^\epsilon_1(t))$.

To obtain $T_d^\epsilon(t,\cdot)$ on $[x^\epsilon_1(t),r_t)$, we use that $\bar\P^{\mu_t,\mu_{t+\epsilon}}$ coincides with the antitone coupling $\pi^{a,\mu_t\lvert_{[x_1^\epsilon(t),r_t)},\mu_{t+\epsilon}\lvert_{(\ell_{t+\epsilon},y_1^\epsilon(t)]}}$ of $\mu_t\lvert_{[x_1^\epsilon(t),r_t)}$ and $\mu_{t+\epsilon}\lvert_{(\ell_{t+\epsilon},y_1^\epsilon(t),r_{t+\epsilon}]}$. It follows that
$$
T^\epsilon_d(t,x)=F^{-1}(t+\epsilon,1-F(t,x)),\quad x\in[x^\epsilon_1(t),r_t),
$$
with $F(t,x)=(x+e^{2t})/(e^t+e^{2t})$, which proves the claim.
\end{proof}

Now define $J^\epsilon_d(t,x)=x-T^\epsilon_d(t,x)$ for $x\in(\ell_t,r_t)$ and $J^\epsilon_u(t,x)=T^\epsilon_u(t,x)-x$ for $x\in(\ell_t,x_1^\epsilon(t))$. Then direct computation gives that
\begin{equation}\label{eq:increasing_j_d}
j_d(t,x):=\lim_{\epsilon\downarrow 0}J^\epsilon_d(t,x)=e^{2t}+\lim_{\epsilon\downarrow 0}T^\epsilon_u(t,x)=e^{2t}+x,\quad x\in(\ell_t=-e^{2t},x_1(t)=r_t=e^t)
\end{equation}
and
\begin{equation}\label{eq:increasing_j_u}
j_u(t,x):=\lim_{\epsilon\downarrow 0}\frac{J^\epsilon_u(t,x)}{\epsilon}=\frac{1+2e^{2t}}{2}\frac{x+e^{2t}}{1+e^t},\quad x\in(-e^{2t},e^t).
\end{equation}

\begin{Proposition}\label{Prop:sde_Sniform_decreasing} Let $(\mu_t)_{t\in[0,1]}$ be as in \eqref{eq:uniformMarginals}. For each $n\geq 1$, let $\P^n:= \P^{*,n} \circ (X^{*,n})^{-1}$, where $\P^{*,n}$ is the $n$-period supermartingale coupling (w.r.t. partition $\pi_n$), obtained through the Markovian iteration of one-period increasing supermartingale couplings $\bar\P^{\mu_{t^n_k},\mu_{t^n_{k+1}}}$. (Here $X^{*,n}$ is the piece-wise constant canonical process; see Section \ref{sec:main_results}.)

Then the sequence $(\P^n)_{n \geq 1}$ is tight w.r.t. the Skorokhod topology on $\Om$.

Moreover, any accumulation point of $(\P^n)_{n\geq 1}$, denoted by  $\P^0$, is the law of the following SDE:
	$$
	X_t = X_0 - \int_0^t j_d(t,X_{s-}) ( d N_s - \nu_s ds ) - \sum_{s \leq t} \mathbf{1}_{\{X_{s-}=e^s\}}(e^{s}+e^{2s}), 
	$$
where $(N_s)_{0 \leq s \leq 1}$ is a unit size jump process, with predictable compensator $ \nu_s:= \frac{j_u}{j_d} ( s, X_{s-} )$, with $j_d,j_u$ as in \eqref{eq:increasing_j_d}, \eqref{eq:increasing_j_u}.

Finally, if Assumption \ref{Assump:CostFunction} holds with $c_{xyy}>0$ and $c_{xy}<0$, then $\P^0$ solves $\mathbf{P}_\infty(\mu)$ and the strong duality $\mathbf{P}_\infty(\mu)=\mathbf{D}_\infty(\mu)$ holds.
\end{Proposition}
\begin{proof}
	The tightness of $(\P^n)_{n\geq 1}$ is guaranteed by Proposition \ref{Prop::Tightness}.
	
	The proof of the SDE characterization of $\P^0$ and its optimality use similar arguments as in Theorems \ref{Thm:main} and \ref{Thm:optimality}. Some notable difference arise due to the fact that the construction is based on the one-period increasing supermartingale coupling (and not on he decreasing one). For example, the drift part of $X$ is determined by $j_u$ (and thus by the upward transitions of the increasing supermrtingale coupling), while the jumps of $X$ are governed by $j_d$ (and thus arise from the downward transitions of the increasing coupling).  In particular, the terms $\alpha_u,\alpha_d$ (which are defined as in the proof of Theorem \ref{Thm:main}) will determine the integral part of the SDE. This does not introduce new arguments and thus we omit the details.
		
	 In order to explain the summation term that appears in the SDE, we turn to the term $\alpha_{\tilde d}$ (see the proof of Theorem \ref{Thm:main}):
	\begin{align*}\alpha_{\tilde d}=  \E^{\P^n}_{t_k^n} [ \{ \varphi(X_{t^n_k} -J_d^{\epsilon^n_k} (t^n_k, X_{t^n_k}) ) -  \varphi(X_{t^n_{k}}) \} \mathbf{1}_{\{ x_1^{\epsilon_k^n}(t^n_k)\leq X_{t_k^n}<e^{t_k^n}\}} ],
	\end{align*}
	where $\epsilon^{n}_k=t^n_{k+1}-t^n_k$. Note that in the present setting
$x_1^{\epsilon}(t)\to e^t$, and thus in the limit, the supermartingale region $[x^\epsilon_1(t),r_t)$ becomes a single point. Using this together with the fact that $J_d^{\epsilon}(t,x)\to e^{2t}+x $, we obtain that
	$$
	\alpha_{\tilde d} = \E^{\P^n}_{t_k^n} \left[  \sum_{ t_k^n\leq s\leq t_{k+1}^n} \left[ \varphi(-e^{2s})-\varphi(e^s) \right]\mathbf{1}_{  \{X_{s} = e^s\}} \right] + O(\epsilon_k^n(\epsilon_k^n \lor  \rho_1(\epsilon_k^n))),
	$$
	which, in the limit, determines the summation term $\sum_{s \leq t} \mathbf{1}_{\{X_{s-}=e^s\}}(e^{s}+e^{2s})$ of the SDE.
	
For the optimality, the arguments are similar to those of Theorem \ref{Thm:optimality}. The candidate optimal dual strategies should be redefined to reflect the fact that we are working with the increasing supermartingale coupling (or the left-curtain martingale coupling in the martingale region). In the limit, we have that $X$ is a martingale on $(\ell_t,r_t)$, $t\in[0,1]$, and has a strict supermartingale transition (a jump down to the lower boundary $\ell_t$) only if it escapes to the upper boundary $r_t$ at time $t$. Hence the dual strategies (see Section \ref{sec:optimalityDual}) in fact can be defined as in Henry-Labord\`ere et al. \cite[Eq. (3.9), (3.10), (3.11)]{HLTanTouzi}:
$$
\partial_x h^*(t,x) = \frac{c_x(x, x)-c_x(x,T_d(t,x))  }{j_d(t,x)},\quad  x\in(\ell_t,r_t),
$$
$$\lim_{x\uparrow x_1(t)=e^t}h^*(t,x)=0=h^*(t,e^t).
$$
We remark that from the above definition, for all $t\in [0,1]$, $x \mapsto h^*(t,x)$ is continuous.  In addition, $\psi^*$ is defined, up to a constant, by
$$
\partial_x \psi^*(t,x):= - h^*(t,x),\quad (t,x)\in E=\{ (t,x): t \in [0,1], x \in (l_t, r_t) \}.
$$
The limiting arguments that prove optimality of $(h^*,\psi^*)$ are as in the decreasing case, and thus we omit the details.
\end{proof}
`														
\subsection{Bachelier dynamics with negative drift} \label{sec:bachelier}

In this section we consider the marginals $(\mu_t)_{t\in(0,1]}$ such that each $\mu_t$ has the density
\begin{equation}\label{eq:normal_density}
f(t,x)= \frac{1}{\sqrt{2 \pi t}} e^{-\frac{(x+t)^2}{2t}}, \quad t \in (0,1].
\end{equation}
Note that Assumption \ref{Assump:1} is almost satisfied; indeed $\mu_0$ is left undefined. We could overcome this by introducing $\delta>0$ and then working with $t\in[\delta,1]$. Alternatively, we could specify $\tilde\mu_t$ to have density $f(t+\delta,\cdot)$ for each $t\in[0,1]$. For the convenience of notation we choose to work on $(0,1]$.

Fix $t\in(0,1)$ and $\epsilon\in(0,1-t]$. Since $m_\epsilon(t)$ and $m^\epsilon(t)$, as in Assumption \ref{Assump:1}, correspond to two unique crossing points of the densities $f(t,\cdot),f(t+\epsilon,\cdot)$, by direct computation one easily obtains that
$$
m^2 = \frac{t (t+\epsilon)}{\epsilon} \ln \frac{t+\epsilon}{t} + (t+\epsilon)t,\quad m\in\{m_\epsilon(t),m^\epsilon(t)\}.
$$
On the other hand, straightforward calculations show that $m_t$ (resp. $\tilde{m}_t$), the minimizer (resp. maximizer) of $\partial_t F(t,\cdot)$, is given by $m_t=-\tilde{m}_t=\sqrt{t(t+1)}$. Note that  $m_t =  \lim_{\epsilon \downarrow 0} m^{\epsilon}(t), \tilde{m}_t= \lim_{\epsilon \downarrow 0} m_{\epsilon}(t)$, for each $t\in(0,1)$. It follows that Assumption \ref{Assump:2} holds. 

Our next goal is to show that $(\mu_t)_{t\in(0,1]}$ also satisfies Assumption \ref{Assump:3}. Since, $r_t=\infty$ for all $t\in(0,1]$, Assumption \ref{Assump:3}(ii) is immediate. For Assumption \ref{Assump:3} (iii), simple calculations lead to
$$
\partial_{tx}f (t,x) = t^{-\frac{5}{2}} \Phi'\left(\frac{x+t}{\sqrt{t}}\right) \left[ x+t+\frac{1}{2} (t-x) \left( \frac{(x+t)^2}{t} -1 \right) \right],
$$
where $\Phi(x):= \int_{-\infty}^{x} \frac{1}{\sqrt{2\pi}} e^{-\frac{t^2}{2}} dt$ is the c.d.f. of the standard normal random variable. Now using that $m_t=\sqrt{t(t+1)}$, we consequently have that $\partial_{tx} f(t,m_t) =t^{-2} \sqrt{t+1} \Phi'( \sqrt{t} +\sqrt{t+1} ) > 0$.

It is left to verify Assumption \ref{Assump:3}(i). Let $x_1^\epsilon(t)$ be the unique phase transition point of the decreasing supermartingale coupling of $\mu_t$ and $\mu_{t+\epsilon}$; see Section \ref{sec:Decreasing_SMOT_one_period_primal}.
\begin{Lemma}\label{lem:x_1^epsilonNormal}
	Let $(\mu_t)_{t\in(0,1]}$ be specified by \eqref{eq:normal_density}. Then, $x_1^\epsilon(t)$ is uniquely determined by the equation
	$$
	1- \Phi\left(\frac{x_1^{\epsilon}(t)+t}{\sqrt{t}}\right) = \frac{1}{\epsilon} ( \sqrt{t+\epsilon} -  \sqrt{t}) \Phi'\left(\frac{x_1^{\epsilon}(t)+t}{\sqrt{t}}\right),\quad t\in(0,1),\ \epsilon\in(0,t-1].
	$$
	Furthermore, the limit $x_1(t):=\lim_{\epsilon\downarrow 0}x^\epsilon_1(t)$ exists, and is uniquely determined by the equation
	\begin{equation} \label{eq:Bownian_boundary}
	2 \sqrt{t} \left( 1 - \Phi \left(\frac{x_1(t)+t}{\sqrt{t}}\right)\right) = \Phi' \left(\frac{x_1(t)+t}{\sqrt{t}}\right).
	\end{equation}
\end{Lemma}
The proof of Lemma \ref{lem:x_1^epsilonNormal} requires the following auxiliary result.
\begin{Lemma} \label{lemm:phase_trans_eq_bachelier}
	For each $t \in (0,1]$, the equation $2 \sqrt{t} ( 1 - \Phi (x)) = \Phi' (x)$
	admits a unique (real-valued) solution  $x^*_t\in(-\infty,2\sqrt{t})$.	
\end{Lemma}
\begin{proof}
	The proof is an application of the intermediate value theorem. Let us consider the function $F(x):=2 \sqrt{t} ( 1 - \Phi (x)) - \Phi' (x)$. First, as $F'(x)=\Phi'(x) (x-2\sqrt{t})$, it is clear that $x \mapsto F(x)$ is continuous on $\R$, decreasing on $(-\infty, 2\sqrt{t}]$, increasing on $[2\sqrt{t},\infty)$. Furthermore, $\lim_{x \to -\infty} F(x)=2\sqrt{t} > 0$ and $\lim_{x \to \infty} F(x)= 0$. It follows that the equation $F(x)=0$ admits a unique solution in the interval $(-\infty,2\sqrt{t})$.	
\end{proof}
\begin{proof}[Proof of Lemma \ref{lem:x_1^epsilonNormal}]
	To determine $x_1^{\epsilon}(t)\in(\ell_t,r_t)$ we use the fact that under the decreasing supermartingale coupling $\hat\P^{\mu_t,\mu_{t+\epsilon}}$, $\mu\lvert_{[x_1^\epsilon(t),r_t)}$ is mapped to $\nu\lvert_{[y_1^\epsilon(t),r_{t+\epsilon})}$ in a martingale way, where $y_1^\epsilon(t)=y_1^{\mu_t,\mu_{t+\epsilon}}$ (see \eqref{eq:y_1Appendix}). In particular, the pair $(x_1^\epsilon(t),y_1^\epsilon(t))$ is uniquely determined by the mass and mean preservation condition
	\begin{equation}\label{eq:mean-massExamples}
	\int^{r_t=\infty}_{x_1^\epsilon(t)}z^i\mu_t(dz)=\int^{r_{t+\epsilon}=\infty}_{y_1^\epsilon(t)}z^i\mu_{t+\epsilon}(dz),\quad i=0,1.
	\end{equation}
	Note that due to the Dispersion Assumption (see Assumption \ref{Assump:1}), we must have that $y_1^\epsilon(t)\leq x_1^\epsilon(t)<m^\epsilon(t)$.
	
	Now apply the change of variables: $\hat{x}_1^{\epsilon}(t) = \frac{x_1^{\epsilon}(t)+t}{\sqrt{t}}$, $\hat{y}_1^{\epsilon}(t) = \frac{y_1^{\epsilon}(t)+t+\epsilon}{\sqrt{t+\epsilon}}$. Then \eqref{eq:mean-massExamples} with $i=0$ reads 
	\begin{equation} \label{eq:bechelier_transition_1}
	\Phi(\hat{x}_1^{\epsilon}(t)) = \Phi(\hat y_1^{\epsilon}(t)).
	\end{equation}
	For \eqref{eq:mean-massExamples} with $i=1$, we first re-write it as 
	$$
	\int_{x_1^{\epsilon}(t)}^{\infty} (x +t) f(t,x) dx + \epsilon \int_{x_1^{\epsilon}(t)}^{\infty} f(t,x) dx = \int_{y_1^{\epsilon}(t)}^{\infty} (x+t+\epsilon) f(t+\epsilon,x) dx,
	$$
	by adding $(t+\epsilon)$ times the equation \eqref{eq:mean-massExamples} with $i=0$ on both sides. Then direct calculation leads to
	\begin{equation} \label{eq:bechelier_transition_2}
	\sqrt{t} \Phi'(\hat{x}_1^{\epsilon}(t)) + \epsilon (1- \Phi(\hat{x}_1^{\epsilon}(t))) = \sqrt{t+\epsilon} \Phi'(\hat{y}_1^{\epsilon}(t)).
	\end{equation}
	
	From \eqref{eq:bechelier_transition_1}, we get $\hat{x}_1^{\epsilon}(t) = \hat{y}_1^{\epsilon}(t)$, or equivalently $\frac{x_1^{\epsilon}(t)+t}{\sqrt{t}} =  \frac{y_1^{\epsilon}(t)+t+\epsilon}{\sqrt{t+\epsilon}}$.
	By plugging this relation $\hat{x}_1^{\epsilon}(t) = \hat{x}_2^{\epsilon}(t)$ into \eqref{eq:bechelier_transition_2}, we obtain that 
	\begin{equation} \label{eq:Bownian_boundary_discrete}
	1- \Phi(\hat{x}_1^{\epsilon}(t)) = \frac{1}{\epsilon} ( \sqrt{t+\epsilon} -  \sqrt{t}) \Phi'(\hat{x}_1^{\epsilon}(t)),
	\end{equation}
	which is precisely the required algebraic equation that characterizes $x^\epsilon_1(t)$.
	
	We now deal with the limit $x_1(t):=\lim_{\epsilon\downarrow 0}x^\epsilon_1(t)$.
	
Recall that $\epsilon_0:=\epsilon_1\wedge\epsilon_2$, where $\epsilon_1,\epsilon_2$ are as in Assumptions \ref{Assump:1} and \ref{Assump:2}, respectively.
		
	 For each (small) $\delta>0$, define $F=F^\delta: \R^3 \to \R$ on  $\{ (t,\epsilon): 0 \leq \epsilon \leq \epsilon_0,\ \delta\leq t \leq 1-\epsilon \}\times\R$ by 
	\begin{equation}\label{eq:F_epsilon}
	F(t, \epsilon,x):= (1- \Phi(x)) - \frac{\sqrt{t+\epsilon} -\sqrt{t}}{\epsilon} \frac{1}{\sqrt{2\pi}} e^{-\frac{x^2}{2}},\quad (t,\epsilon,x)\in[\delta,1-\epsilon]\times(0, \epsilon_0]\times\R
	\end{equation}
	and 
	\begin{equation}\label{eq:F_0}
	F(t, 0, x):=  (1- \Phi(x)) - \frac{1}{2 \sqrt{t}} \frac{1}{\sqrt{2\pi}} e^{-\frac{x^2}{2}},\quad (t,x)\in[\delta,1]\times\R.
	\end{equation}
Note that, by \eqref{eq:Bownian_boundary_discrete}, $F(t,\epsilon,\hat x^\epsilon_1(t))=0$ for all $\epsilon \in(0,\epsilon_0]$ and $t\in[\delta,1-\epsilon]$.

Using Lemma \ref{lemm:phase_trans_eq_bachelier}, we define $\hat x_1^0(t)$, for all $t\in(0,1]$, to be the unique solution to $F(t,0,x)=0$. Then $x^0_1(t)$ is defined as $x_1^0(t):=\sqrt{t}\hat x_1^0(t)-t$, for each $t\in(0,1]$. Note that $x_1^0(\cdot)$ is continuous if and only if $\hat x_1^0(\cdot)$ is. Now since $\hat x^0_1(t)$ uniquely satisfies $F(t,0,x)=0$, we immediately have that $\lim_{\epsilon\downarrow0}\hat x_1^\epsilon(t)$ exists, for all $t\in(0,1]$, and uniquely satisfies \eqref{eq:Bownian_boundary}.

\end{proof}

\begin{Lemma} 
	For $\epsilon\in[0,\epsilon_0]$ and $t\in(0,1]$, let $x^\epsilon_1(t)$ be given by Lemma \ref{lem:x_1^epsilonNormal} with $x^0_1(t):=x_1(t)=\lim_{\epsilon\downarrow0}x_1^\epsilon(t)$.
	
	Then, for each (small) $\delta>0$, the map $(t,\epsilon)\mapsto x^\epsilon_1(t)$ is continuous on $\{ (t,\epsilon): 0 \leq \epsilon \leq \epsilon_0,\ \delta\leq t \leq 1-\epsilon \}\times\R$.	\end{Lemma}
\begin{proof}	
	Let $F$ be defined as in \eqref{eq:F_epsilon} and \eqref{eq:F_0}. To prove the continuity of $(t,\epsilon)\mapsto x^\epsilon_1(t):=\sqrt{t}\hat x^\epsilon_1(t)-t$, we will use the implicit function theorem and show that $(t,\epsilon)\mapsto \hat x^\epsilon_1(t)$ is continuous, where $F(t,\epsilon,\hat x^\epsilon_1(t))=0$.
	
	We claim that $(t,\epsilon, x) \mapsto F(t,\epsilon,x)$ is continuously differentiable. 
		
		First, for $0<\epsilon\leq\epsilon_0$, we have that 
		$$
		\frac{\partial F}{\partial \epsilon} (t, \epsilon,x) = - \frac{1}{\sqrt{2\pi}} e^{-\frac{x^2}{2}} \frac{\frac{1}{2}(t+\epsilon)^{-\frac{1}{2}} \epsilon - ( \sqrt{t+\epsilon} - \sqrt{t} ) }{\epsilon^2},
		$$
		and it follows that $\lim_{\epsilon \to 0} \frac{\partial F}{\partial \epsilon} (t, \epsilon,x)=-\Phi'(x)\frac14 t^{-\frac32}$. On the other hand, 
		$$
		\frac{\partial F}{\partial \epsilon} (t, 0,x) = \lim_{\epsilon \to 0} \frac{F(t,\epsilon,x)-F(t, 0,x)}{\epsilon} = \lim_{\epsilon \to 0} - \frac{1}{\sqrt{2\pi}} e^{-\frac{x^2}{2}} \frac{\frac{1}{2}t^{-\frac{1}{2}} \epsilon - ( \sqrt{t+\epsilon} - \sqrt{t} ) }{\epsilon^2}=-\Phi'(x)\frac14 t^{-\frac32},
		$$
		and hence $\frac{\partial F}{\partial \epsilon}$ is continuous. The continuity of $\frac{\partial F}{\partial t}$ and $\frac{\partial F}{\partial x}$ follow similarly by the direct computations.
		
		Now we check that $\frac{\partial F}{\partial x}(t,\epsilon,x)|_{F(t, \epsilon, x)=0}\neq0$. For $\epsilon \in[0,\epsilon_0]$, direct calculations show that
		$$
		\frac{\partial F}{\partial x}(t,\epsilon,x)|_{F(t, \epsilon, x)=0} =  \Phi'(x) \left[ x  \frac{1-\Phi(x)}{ \Phi'(x) } -1  \right] <0,
		$$
		where for the inequality we use the fact that $\frac{1-\Phi(x)}{\Phi'(x)} < \frac{1}{x}$ when $x > 0$ (for $x\leq 0$ the inequality is trivially satisfied).
		
		By the implicit function theorem, we get that $(t, \epsilon) \mapsto \hat x_1^{\epsilon}(t)$ is continuously differentiable on a compact set $\{ (t,\epsilon): 0 \leq \epsilon \leq \epsilon_0,\ \delta\leq t \leq 1-\epsilon \}$, and consequently uniformly bounded and continuous. The regularity of $(t, \epsilon) \mapsto x_1^{\epsilon}(t)$ follows from the regularity of $(t,x) \mapsto \hat x_1^{\epsilon}(t)$.	
	\end{proof}
Since all the assumption required for Theorems \ref{Thm:main} and \ref{Thm:optimality} hold we have the following.
\begin{Proposition}\label{Prop:sde_bachelier}
	Let $(\mu_t)_{t\in(0,1]}$ be specified by \eqref{eq:normal_density} and consider $t\mapsto x_1(t)$ as in Lemma \ref{lem:x_1^epsilonNormal}. Then the SDE as in Theorem \ref{Thm:main} is explicitly given by
	$$
	X_t = X_0 + \int_0^t j_u(s,X_{s-}) ( d N_s - \nu_s ds )   \mathbf{1}_{\{x_1(s)< X_{s-}<m_s\}} - \int_0^t \mathbf{1}_{ \{X_{s-}\leq x_1(s)\}}  j_d(s,X_{s-}) ds,
	$$
	where $m_s= \sqrt{s(s+1)}$, $\nu_s(s,X_{s-}):= \frac{j_d}{j_u} ( s, X_{s-} )$ with $j_u(t,x) = T_u(t,x) -x$, and $j_d(t,x)=  \frac{1}{2\sqrt{t}} ( \tilde{T}_u(t,x)-2 \sqrt{t} ) e^{\frac{\hat{x}^2}{2}- \frac{\tilde{T}_u(t,x)^2}{2}} - \frac{\hat{x}}{2\sqrt{t}} + 1$, for $x\in(x_1(s),m_s)$, and $j_d(s,X_{s-}):=\frac{s-x}{2s}$ for $x\in(\ell_s,x_1(s)]$; here $\hat{x}:=\frac{x+t}{\sqrt{t}}$ and $ \tilde{T}_u(t,x):=\sqrt{t} T_u(t,x) -t$ is uniquely determined by the equation
	\begin{align*}
	( \tilde{T}_u(t,x) - \hat{x}) ( \tilde{T}_u(t,x)-2\sqrt{t} ) \Phi'(\tilde{T}_u(t,x))  + 2\sqrt{t}  \left(\Phi(\tilde{T}_u(t,x))-\Phi(\hat{x})\right) =\Phi'(\hat x) -\Phi'(\tilde{T}_u(t,x)).
	\end{align*}
	
	In addition, if Assumption \ref{Assump:CostFunction}(ii) and the integrability condition \eqref{eq:optimality_integrability} are satisfied, then the law $\P^0$ of the above process is the optimal solution for the 
	primal problem $\eqref{eq:Primal_Full}$, and the strong duality
	$$
	\E^{\P^0} [ C(X_\cdot) ] = \mathbf{P}_{\infty}(\mu) =  \mathbf{D}_{\infty}(\mu)
	$$
	holds.
\end{Proposition}
\begin{proof}
	As it has been verified that Assumptions  \ref{Assump:1},  \ref{Assump:2},  \ref{Assump:3} are valid (up to $t=0$), using Theorems \ref{Thm:main} and \ref{Thm:optimality} we obtain the validity of the SDE and the optimality of $\P^0$. It is left to derive the explicit expressions of the terms $T_u, j_u, j_d$ that define the SDE.
	
	In the supermartingale region, i.e., for $x\in(\ell_t,x_1(t)]$, we have that 
	$$
	j_d(t,x) = \frac{1}{f(t,x)} \partial_t F(t,x) = \frac{\sqrt{t}}{\Phi'(\frac{x+t}{\sqrt{t}})} \partial_t \Phi(\frac{x+t}{\sqrt{t}}) = \frac{t-x}{2t}.
	$$
	
	In the martingale region, i.e., for $x\in(x_1(t),m_t)$, using \eqref{eq:T_S_charac} we have that $T_u(t,x)$ is uniquely determined by
	$$
	\int_x^{T_u(t,x)} (x - \xi) \frac{\partial}{\partial t} \left( \frac{1}{\sqrt{2 \pi t}} e^{-\frac{(\xi+t)^2}{2t}} \right) d \xi = 0.
	$$
	By a change of variables $\tilde{T}_u(t,x):=\sqrt{t} T_u(t,x) -t$, we obtain that $\tilde{T}_u(t,x)$ satisfies
	$$
	( \tilde{T}_u(t,x) - \hat{x}) ( \tilde{T}_u(t,x)-2\sqrt{t} ) \Phi'(\tilde{T}_u(t,x))  + 2\sqrt{t}  \left(\Phi(\tilde{T}_u(t,x))-\Phi(\hat{x})\right) =\Phi'(\hat x) -\Phi'(\tilde{T}_u(t,x)).
	$$
	Finally, for $x\in(x_1(t),m_t)$,
	$$
	j_d(t,x)= \frac{\partial_t F(t,x) - \partial_t F(t, T_u(t,x)) }{f(t,x)} = \frac{1}{2\sqrt{t}} ( \tilde{T}_u(t,x)-2 \sqrt{t} ) e^{\frac{\hat{x}^2}{2}- \frac{\tilde{T}_u(t,x)^2}{2}} - \frac{\hat{x}}{2\sqrt{t}} + 1,
	$$	
	as required.	
\end{proof}

\begin{Remark} \label{Rmk:Increasing_Bachelier} Let $(\mu_t)_{t\in(0,1]}$ be specified by \eqref{eq:normal_density} and for $t\in (0,1)$ and $\epsilon \in (0, 1-t]$, consider the increasing supermartingale coupling of $\mu_t$ and $\mu_{t+\epsilon}$. Then there exists the unique threshold $x_1^\epsilon(t)$, such that $(\ell_t,x_1^\epsilon(t)]$ is a martingale region, while $(x^\epsilon_1(t),r_t)$ corresponds to the supermartingale region.
	
By changing variables $\hat{x}_1^{\epsilon}(t) = \frac{x_1^{\epsilon}(t)+t}{\sqrt{t}}$ and using similar arguments as in Lemma \ref{lem:x_1^epsilonNormal}, we have that $x^\epsilon_1(t)$ uniquely solves
$$
\epsilon \Phi(\hat{x}_1^{\epsilon}(t)) = \sqrt{t} \Phi'(\hat{x}_1^{\epsilon}(t)) + \sqrt{t+\epsilon} \Phi'(-\hat{x}_1^{\epsilon}(t)).
$$
Then letting $\epsilon\downarrow 0$ we have that $\hat{x}_1^{\epsilon}(t) \to +\infty$ (provided the limit exists). It follows that ${x}_1^{\epsilon}(t) \to +\infty$, and therefore one could expect that the supermartingale transitions of the limiting SDE correspond to jumps from $\infty$ to $-\infty$.

This is in-line with our observations in Proposition \ref{Prop:sde_Sniform_decreasing} obtained in the uniform case. In particular, when the increasing supermartingale coupling is used in construction, and in the case of unbounded support of marginals, the convergence results may fail in general. We will consider these issues in our future research.
\end{Remark}

\subsection{Geometric Brownian motion with decreasing average}
In this section we consider the marginals $(\mu_t)_{t\in(0,1]}$ such that each $\mu_t$ has the density
\begin{equation}\label{eq:log_normal_density}
f(t,x)=\frac{1}{x\sqrt{2 \pi t}} e^{-\frac{(\ln x + t)^2}{2t}}, \quad t \in (0,1].
\end{equation}
Similarly to the previous Bachelier case (see Section \ref{sec:bachelier}), we can introduce $\delta>0$ and then work with $t\in[\delta,1]$ (or $t\in[\delta,1+\delta]$). For the convenience of notation we shall work on $(0,1]$. 

Assumption \ref{Assump:1} is clearly satisfied. Fix $t\in(0,1)$ and $\epsilon\in(0,1-t]$. For $m_\epsilon(t)$ and $m^\epsilon(t)$ (as in Assumption \ref{Assump:1}), direct computation leads to 
$$
( \ln m)^2 = \frac{t (t+\epsilon)}{\epsilon} \ln \frac{t+\epsilon}{t} + (t+\epsilon)t,\quad m\in\{m_\epsilon(t),m^\epsilon(t)\}.
$$
On the other hand, straightforward calculations show that $m_t$ (resp. $\tilde{m}_t$), the minimizer (resp. maximizer) of $\partial_t F(t,\cdot)$, is given by $\ln m_t=-\ln \tilde{m}_t=\sqrt{ t(t+1)}$.  It follows that Assumption \ref{Assump:2} is also valid. 

We now show that $(\mu_t)_{t\in(0,1]}$ also satisfies Assumption \ref{Assump:3}. First, Assumption \ref{Assump:3}(ii) is immediate, as $r_t=\infty$ for all $t\in(0,1]$. For Assumption \ref{Assump:3} (iii), simple calculations lead to
$$
\partial_{tx} f (t,m_t) = t^{-\frac{3}{2}} \sqrt{t+1} \Phi'( \sqrt{t} +\sqrt{t+1} ) e^{-2 \sqrt{t(t+1)}} >0.
$$

It is left to verify Assumption \ref{Assump:3}(i). Let $x_1^\epsilon(t)$ be the unique phase transition point of the decreasing supermartingale coupling of $\mu_t$ and $\mu_{t+\epsilon}$; see Section \ref{sec:Decreasing_SMOT_one_period_primal}.

\begin{Lemma}\label{lem:x_1^epsilonLogNormal}
	Let $(\mu_t)_{t\in(0,1]}$ be specified by \eqref{eq:log_normal_density}. Then, $x_1^\epsilon(t)$ is uniquely determined by the equation
	$$
		1-\Phi( \frac{\ln x_1^{\epsilon}(t)}{\sqrt{t}} ) = e^{-\frac{1}{2}\epsilon} [1-\Phi( \frac{\ln x_1^{\epsilon}(t)+t}{\sqrt{t}} - \sqrt{t+\epsilon})],\quad t\in(0,1),\ \epsilon\in(0,t-1].
	$$
	Furthermore, the limit $x_1(t):=\lim_{\epsilon\downarrow 0}x^\epsilon_1(t)$ exists, and is uniquely determined by the equation
\begin{equation} \label{eq:GBm_boundary}
\frac{1}{\sqrt{t}} \Phi'(\frac{\ln x_1(t) }{\sqrt{t}}) = 1- \Phi(\frac{\ln x_1(t) }{\sqrt{t}}).
\end{equation}
\end{Lemma}
The proof of Lemma \ref{lem:x_1^epsilonLogNormal} requires the following auxiliary result.
\begin{Lemma} \label{lemm:phase_trans_eq_GBm}
	For each $t \in (0,1]$, the equation $\sqrt{t} ( 1 - \Phi (x)) = \Phi' (x)$
	admits the unique (real-valued) solution  $x^*_t\in(-\infty,\sqrt{t})$.	
\end{Lemma}
\begin{proof}
	The proof is similar to Lemma \ref{lemm:phase_trans_eq_bachelier} and is an application of the intermediate value theorem. Hence we omit the details here.	
\end{proof}
\begin{proof}[Proof of Lemma \ref{lem:x_1^epsilonLogNormal}]
	To determine $x_1^{\epsilon}(t)\in(\ell_t,r_t)$ we use the fact that under the decreasing supermartingale coupling $\hat\P^{\mu_t,\mu_{t+\epsilon}}$, $\mu\lvert_{[x_1^\epsilon(t),r_t)}$ is mapped to $\nu\lvert_{[y_1^\epsilon(t),r_{t+\epsilon})}$ in a martingale way, where $y_1^\epsilon(t)=y_1^{\mu_t,\mu_{t+\epsilon}}$ (see \eqref{eq:y_1Appendix}). In particular, the pair $(x_1^\epsilon(t),y_1^\epsilon(t))$ is uniquely determined by the mass and mean preservation condition
	\begin{equation}\label{eq:mean-massExamples2}
	\int^{r_t=\infty}_{x_1^\epsilon(t)}z^i\mu_t(dz)=\int^{r_{t+\epsilon}=\infty}_{y_1^\epsilon(t)}z^i\mu_{t+\epsilon}(dz),\quad i=0,1.
	\end{equation}
	Note that due to the Dispersion Assumption (see Assumption \ref{Assump:1}), we must have that $y_1^\epsilon(t)\leq x_1^\epsilon(t)<m^\epsilon(t)$.
	
	Now apply the change of variables: $\hat{x}_1^{\epsilon}(t) =  \frac{\ln x_1^{\epsilon}(t)+t}{\sqrt{t}}$, $\hat{y}_1^{\epsilon}(t) =  \frac{\ln y_1^{\epsilon}(t)+t+\epsilon}{\sqrt{t+\epsilon}}$. Then \eqref{eq:mean-massExamples2} with $i=0$ reads 
	\begin{equation} \label{eq:GBm_transition_1}
	\Phi(\hat{x}_1^{\epsilon}(t)) = \Phi(\hat y_1^{\epsilon}(t)).
	\end{equation}
	For \eqref{eq:mean-massExamples2} with $i=1$, it is equivalent to
	$$
	\int_{\hat{x}_1^{\epsilon}(t)}^{+\infty} \frac{1}{\sqrt{2\pi}} e^{-\frac{1}{2}(y-\sqrt{t})^2} dy = e^{-\frac{1}{2}\epsilon} \int_{\hat{y}_1^{\epsilon}(t)}^{+\infty} \frac{1}{\sqrt{2\pi}} e^{-\frac{1}{2}(y-\sqrt{t+\epsilon})^2} dy,
	$$
	which leads to 
	\begin{equation} \label{eq:GBm_boundary_discrete}
	1-\Phi(\hat{x}_1^{\epsilon}(t) - \sqrt{t}) = e^{-\frac{1}{2}\epsilon} [1-\Phi(\hat{y}_1^{\epsilon}(t) - \sqrt{t+\epsilon})].
	\end{equation}

	From \eqref{eq:bechelier_transition_1}, we get $\hat{x}_1^{\epsilon}(t) = \hat{y}_1^{\epsilon}(t)$, or equivalently $ \frac{\ln x_1^{\epsilon}(t)+t}{\sqrt{t}}=  \frac{\ln y_1^{\epsilon}(t)+t+\epsilon}{\sqrt{t+\epsilon}}$.
	By plugging  $\hat{x}_1^{\epsilon}(t) = \hat{y}_1^{\epsilon}(t)$ into \eqref{eq:GBm_boundary_discrete}, we obtain that 
	\begin{equation} \label{eq:GBownian_boundary_discrete}
	1-\Phi( \frac{\ln x_1^{\epsilon}(t)}{\sqrt{t}} ) = e^{-\frac{1}{2}\epsilon} [1-\Phi( \frac{\ln x_1^{\epsilon}(t)+t}{\sqrt{t}} - \sqrt{t+\epsilon})],
	\end{equation}
	which characterizes $x^\epsilon_1(t)$.
	
	We now deal with the limit $x_1(t):=\lim_{\epsilon\downarrow 0}x^\epsilon_1(t)$.
	
Recall that $\epsilon_0:=\epsilon_1\wedge\epsilon_2$, where $\epsilon_1,\epsilon_2$ are as in Assumptions \ref{Assump:1} and \ref{Assump:2}, respectively.
		
	 For each (small) $\delta>0$, define $G=G^\delta: \R^3 \to \R$ on  $\{ (t,\epsilon): 0 \leq \epsilon \leq \epsilon_0,\ \delta\leq t \leq 1-\epsilon \}\times\R$ by 
	\begin{equation}\label{eq:G_epsilon}
	G(t, \epsilon,x):= \frac{1}{\epsilon} \left\{ (1- \Phi(x)) - e^{-\frac{1}{2} \epsilon} [1-\Phi( x + \sqrt{t} - \sqrt{t+\epsilon})] \right\},\quad (t,\epsilon,x)\in[\delta,1-\epsilon]\times(0, \epsilon_0]\times\R
	\end{equation}
	and 
	\begin{equation}\label{eq:G_0}
	G(t, 0, x):=  (1- \Phi(x)) - \frac{1}{ \sqrt{t}} \frac{1}{\sqrt{2\pi}} e^{-\frac{x^2}{2}}.
	\end{equation}

Note that, by \eqref{eq:GBownian_boundary_discrete}, $G(t,\epsilon,\hat x^\epsilon_1(t))=0$ for all $\epsilon \in(0,\epsilon_0]$ and $t\in[\delta,1-\epsilon]$.

Using Lemma \ref{lemm:phase_trans_eq_GBm}, we define $\hat x_1^0(t)$, for all $t\in(0,1]$, to be the unique solution to $F(t,0,x)=0$. Then $x^0_1(t)$ is defined as $x_1^0(t):=\exp \left(\sqrt{t}\hat x_1^0(t)-t\right)$, for each $t\in(0,1]$. Note that $x_1^0(\cdot)$ is continuous if and only if $\hat x_1^0(\cdot)$ is. Now since $\hat x^0_1(t)$ uniquely satisfies $F(t,0,x)=0$, we immediately have that $\lim_{\epsilon\downarrow0}\hat x_1^\epsilon(t)$ exists, for all $t\in(0,1]$, and uniquely satisfies \eqref{eq:Bownian_boundary}.

\end{proof}

\begin{Lemma}
	For $\epsilon\in[0,\epsilon_0]$ and $t\in(0,1]$, let $x^\epsilon_1(t)$ be given by Lemma \ref{lem:x_1^epsilonLogNormal} with $x^0_1(t):=x_1(t)=\lim_{\epsilon\downarrow0}x_1^\epsilon(t)$.
	
	Then, for each (small) $\delta>0$, the map $(t,\epsilon)\mapsto x^\epsilon_1(t)$ is continuous on $\{ (t,\epsilon): 0 \leq \epsilon \leq \epsilon_0,\ \delta\leq t \leq 1-\epsilon \}\times\R$.	\end{Lemma}

\begin{proof}

Let $G$ be defined as in \eqref{eq:G_epsilon} and \eqref{eq:G_0}. To prove the continuity of $(t,\epsilon)\mapsto x^\epsilon_1(t):= \exp \left( \sqrt{t}\hat x^\epsilon_1(t)-t \right)$, we will use the implicit function theorem and show that $(t,\epsilon)\mapsto \hat x^\epsilon_1(t)$ is continuous, where $G(t,\epsilon,\hat x^\epsilon_1(t))=0$.
	
	We claim that $(t,\epsilon, x) \mapsto G(t,\epsilon,x)$ is continuously differentiable.  Indeed,
	$$
	\begin{aligned}
	\lim_{\epsilon \to 0} \frac{\partial G}{\partial \epsilon}(t, \epsilon, x) &= -\frac18 (1-\Phi(x)) + \frac18 \Phi'(x) t^{-\frac32} (2t+1) + \frac18 \Phi''(x) t^{-1} \\
	&=\lim_{\epsilon \to 0} \frac{G(t,\epsilon,x)-G(t, 0,x)}{\epsilon},   
	\end{aligned}
	$$	
and hence $\frac{\partial G}{\partial \epsilon}$ is continuous. The continuity of $\frac{\partial G}{\partial t}$ and $\frac{\partial G}{\partial x}$ follow similarly by the direct computations.
		
		Now we check that $\frac{\partial G}{\partial x}(t,\epsilon,x)|_{G(t, \epsilon, x)=0}\neq0$. For $\epsilon \in[0,\epsilon_0]$, direct calculations show that
		$$
	\frac{\partial G}{\partial x}(t,\epsilon,x)|_{G(t, \epsilon, x)=0}  = -\Phi'(x) + e^{-\frac12 \epsilon} \Phi'(x+\sqrt{t} - \sqrt{t+\epsilon}) = -\frac{\Phi'(x)}{1- \Phi(x)}  + \frac{\Phi'(x+\sqrt{t} - \sqrt{t+\epsilon})}{1-\Phi(x+\sqrt{t} - \sqrt{t+\epsilon})},
	$$
	and the desired conclusion follows.
	
	We now study the monotonicity of the map $x \mapsto \frac{\Phi'(x)}{1- \Phi(x)}$:
	$$
	\left( \frac{\Phi'(x)}{1- \Phi(x)} \right)' = \frac{\Phi''(x) (1-\Phi(x)) +( \Phi'(x))^2 }{ (1-\Phi(x))^2 } = \Phi'(x)  \frac{ \Phi'(x) -x(1-\Phi(x)) }{ (1-\Phi(x))^2 } >0,
	$$ 
	where in the last step, in the case $x > 0$, we use the inequality $\frac{1-\Phi(x)}{\Phi'(x)} < \frac{1}{x}$ (the case $x\geq 0$ is trivially satisfied). Hence $x \mapsto \frac{\Phi'(x)}{1- \Phi(x)} $ is strictly increasing. It follows that $ \frac{\partial G}{\partial x}|_{G(t,\epsilon,x)=0} <0$ for $\epsilon\in(0,\epsilon_0]$. Similarly, for $\epsilon=0$, we have that
	$$
	\frac{\partial G}{\partial x}(t,0,x)|_{G(t, 0, x)=0}  = -\Phi'(x) + \frac{1}{\sqrt{t}} \Phi'(x) x = -\Phi'(x) + x (1-\Phi(x)) <0.
	$$ 
Finally, from the implicit function theorem, we have that $(t, \epsilon) \mapsto \ln x_1^{\epsilon}(t)$ is continuously differentiable on a compact set $\{ (t,\epsilon): 0 \leq \epsilon \leq \epsilon_0,\ \delta\leq t \leq 1-\epsilon \}$, and consequently uniformly bounded and continuous. 
\end{proof}

Finally we have the following characterization and the optimality of the limiting process.
\begin{Proposition}\label{Prop:sde_GBm}
	Let $(\mu_t)_{t\in(0,1]}$ be specified by \eqref{eq:log_normal_density} and consider $t\mapsto x_1(t)$ as in Lemma \ref{lem:x_1^epsilonLogNormal}. Then the SDE as in Theorem \ref{Thm:main} is explicitly given by
	$$
	X_t = X_0 + \int_0^t j_u(s,X_{s-}) ( d N_s - \nu_s ds )   \mathbf{1}_{\{x_1(s)< X_{s-}<m_s\}} - \int_0^t \mathbf{1}_{ \{X_{s-}\leq x_1(s)\}}  j_d(s,X_{s-}) ds,
	$$
	where $m_s=e^{\sqrt{s(s+1)}}$, $\nu_s(s,X_{s-}):= \frac{j_d}{j_u} ( s, X_{s-} )$ with $j_u(t,x) = T_u(t,x) -x$, and 
	$$
j_d(t,x)= \frac{x}{2} \left[ e^{\frac{(\ln x+t)^2}{2t}-\frac{(\ln (T_S(t,x))+t)^2}{2t}}(\frac{\ln (T_u(t,x))}{\sqrt{t}} -\sqrt{t} ) - (\frac{\ln x}{\sqrt{t}} -\sqrt{t} ) \right],
$$
	for $x\in(x_1(s),m_s)$, and $j_d(s,X_{s-}):=x \frac{s-\ln x}{2s}$ for $x\in(\ell_s,x_1(s)]$; here $T_u(t,x)$ is uniquely determined by the following equation:
$$
(1-\frac{x}{T_u(t,x)}) e^{-\frac{(\ln (T_u(t,x)))^2}{2t}} \ln (T_u(t,x)) = t (1-\frac{x}{T_u(t,x)}) e^{-\frac{(\ln x)^2}{2t}} - t^{\frac{3}{2}}  \left[ \Phi(\frac{\ln (T_u(t,x))}{\sqrt{t}}) - \Phi(\frac{\ln x}{\sqrt{t}}) \right].
$$
	In addition, if Assumption \ref{Assump:CostFunction}(ii) and the integrability condition \eqref{eq:optimality_integrability} are satisfied, then the law $\P^0$ of the above process is the optimal solution for the 
	primal problem $\eqref{eq:Primal_Full}$, and the strong duality
	$$
	\E^{\P^0} [ C(X_\cdot) ] = \mathbf{P}_{\infty}(\mu) =  \mathbf{D}_{\infty}(\mu)
	$$
	holds.
\end{Proposition}
\begin{proof} As it has been verified that Assumptions  \ref{Assump:1},  \ref{Assump:2},  \ref{Assump:3} are valid (up to $t=0$), using Theorems \ref{Thm:main} and \ref{Thm:optimality} we obtain the validity of the SDE and the optimality of $\P^0$. It is left to derive the explicit expressions of the terms $T_u, j_u, j_d$ that define the SDE. 

The calculations are similar to those of Proposition \ref{Prop:sde_bachelier}, and we omit the details.
\end{proof}

\appendix
\section{Brenier's Theorem} \label{sec:Brenier}
\subsection{Brenier's Theorem in Optimal Transport (OT)} \label{sec:Brenier_OT}
Let $(X,Y)$ be a random vector in $\R^2$ and let $\Pc_{2}$ be the set of all (Borel) probability measures on $\R^2$. We denote respectively $\mu$ and $\nu$ the (integrable) marginal distributions of $X$ and $Y$.

Let $c: \R^2 \to \R$ be a (measurable) cost, or payoff, function. The primal Optimal Transport problem corresponds to
$$
P_{OT}(\mu, \nu):=\sup_{\P \in \Pc_2(\mu, \nu)} \E^{\P} [ c(X,Y) ], ~~~~\mbox{ where } \Pc_2(\mu, \nu):= \{ \P \in \Pc_{2}: X \sim_{\P} \mu, Y \sim_{\P} \nu \}.
$$
The above optimisation problem has a dual problem, which is defined by
$$
D_{OT}(\mu,\nu):=\inf_{(\varphi, \psi) \in \Dc_{s}} \{ \mu(\varphi) + \nu(\psi) \}, \mbox{ where } \Dc_{s}:= \{ (\varphi, \psi): \mu(\varphi\vee0) + \nu(\psi\vee0)<\infty, \varphi \oplus \psi \geq c \}.
$$
In the above, $\varphi \oplus \psi (x,y) : = \varphi(x) + \psi (y)$ and $\chi(f): = \int f d \chi,$ for a (Borel) measure $\chi$ on $\R$ and a (Borel) measurable $f:\R\to\R$. Under mild conditions, it can be proved that the duality $P_{OT}(\mu, \nu)=D_{OT}(\mu, \nu)$ holds. (The notation $\Dc_{s}$ is motivated by the fact that, in financial terms, $(\varphi,\psi)\in\Dc_s$ corresponds to a \textit{static} hedging strategy.)

The Brenier's theorem gives the explicit expressions of the optimizers of both the primal and dual problems. In the case $\mu$ is atomless, a candidate primal optimizer is given by the Fr\'echet-Hoeffding (or quantile) coupling $\P_*$ defined by
\begin{equation}\label{eq:quantile}
\P_*(dx, dy):= \mu(dx) \delta_{T_*(x)} (dy),
\end{equation} push-forward map from $\mu$ to $\nu$ is where $T_*: = F_{\nu}^{-1} \circ F_{\mu}$. Here $F_{\nu}^{-1}$ is the right-continuous inverse of (the c.d.f. of $\nu$) $F_{\nu}$: $F_{\nu}^{-1}(t):= \inf \{ y: F_{\nu} > t \}$.

We further introduce candidate dual optimizers $\varphi_*$ and $\psi_*$ (up to a constant) by setting
\begin{equation}\label{eq:OTdualOptimizers}
\varphi_*(x):= c(x, T_*(x)) - \psi_* \circ T_*(x),\quad \psi'_*(y):=c_y(T_*^{-1}(y), y),\quad \forall x, y \in \R.
\end{equation}
\begin{Theorem}[Brenier \cite{brenier1991polar}, Rachev and R\"uschendorf \cite{rachev1998mass}]\label{thm:classicalB}
Suppose $\mu$ has no atoms. Let $c$ be an upper semi-continuous function with (at most) linear growth. Assume that the partial derivative $c_{xy}$ exists and satisfies the Spence-Mirrlees condition $c_{xy}>0$. 

Let $\P_*$ be as in \eqref{eq:quantile}, $(\varphi_*,\psi_*)$ as in \eqref{eq:OTdualOptimizers}, and assume that $(\mu( \varphi_*\vee0)+\nu(\psi_*\vee0))<\infty$.

Then $\P_*\in\Pc_2(\mu,\nu)$, $(\varphi_*, \psi_*) \in \Dc_{s}$ and 
$$
\int c(x, T_*(x)) \mu(dx)=P_{OT}(\mu, \nu) = D_{OT} (\mu, \nu)=\mu(\varphi_*)+\nu(\psi_*).
$$
\end{Theorem}
\begin{Remark}\label{rem:antitone_dual}
	In the case $c_{xy}<0$, then $P_{OT}=D_{OT}$ and the optimal coupling is the antitone coupling which is supported on $x\mapsto T^*(x):=F^{-1}_\nu\circ (1-F_\mu(x))$. The dual optimizers are still defined through \eqref{eq:OTdualOptimizers}, but with $T^*$ in place of $T_*$.
	\end{Remark}
\subsection{Brenier's Theorem in Martingale Optimal Transport (MOT)}\label{sec:Brenier_MOT}

Two (integrable) probability measures $\mu, \nu$ on $\R$ are increasing in convex order, denoted by $\mu \leq_c \nu$, if for all convex functions $f: \R \to \R$, we have that $\mu(f) \leq \nu(f)$. By the Strassen's classical result \cite{strassen1965existence}, $\mu \leq_c \nu$ is equivalent to non-emptiness of the set $\Mc_2(\mu, \nu):= \{ \P \in \Pc_2(\mu, \nu): \E^\P[ Y|X ]=X \}$. 

In this section we suppose that $\mu$ and $\nu$ have finite first moments, $\mu \leq_c \nu$, and the following Dispersion Assumption holds:
\begin{Assumption}[Dispersion] \label{ass:Dispersion_c_appendix}
	$\mu,\nu$ admit densities $f_\mu,f_\nu$ that satisfy
\begin{enumerate} 
	\item $f_\mu>0$ on $(\ell_\mu,r_\mu)$, $f_\nu>0$ on $(\ell_\nu,r_\nu)$, and $\ell_\nu\leq\ell_\mu<r_\mu\leq r_\nu$.
	\item There exists $\ell_\mu<\underline m:=\underline{m}^{\mu,\nu}<\overline m:=\overline{m}^{\mu,\nu}<r_\mu$ such that
	$$
	f_\mu>f_\nu,\textrm{ on }(\underline{m},\overline{m}),\quad f_\nu>f_\mu\textrm{ on }(\ell_\nu,\underline{m})\cup(\overline{m},r_\nu),\quad f_\mu=f_\nu,\textrm{ on }\{\underline{m},\overline{m}\}.
	$$
\end{enumerate}
\end{Assumption}
\begin{Remark}\label{rem:dispersion_c_appendix}
	Note that if $\overline\mu=\overline\nu$ and Assumption \ref{ass:Dispersion_c_appendix} holds, then $\mu\leq_c\nu$. Furthermore, the difference of distribution functions $y\mapsto F_\nu(y)-F_\mu(y)$ attains the (unique) maximum and minimum at $\underline{m}^{\mu,\nu}$ and $\overline{m}^{\mu,\nu}$, respectively.
	\end{Remark}
	
Given a (measurable) payoff function $c: \R^2 \to \R$, the primal problem of MOT is defined by
$$
P_{MOT}(\mu, \nu):=\sup_{\P \in \Mc_2(\mu, \nu)} \E^{\P} [ c(X,Y) ].
$$
The MOT problem also has a dual problem:
$$
D_{MOT}(\mu, \nu):= \inf_{(\varphi, \psi, h) \in \Dc_{ss}} \{ \mu(\varphi) + \nu(\psi) \},
$$
where the set of dual variables is defined as:
$$
\Dc_{ss}:= \{ (\varphi, \psi, h): \mu(\varphi\vee0)+\nu(\psi\vee0)<\infty, h:\R\to\R,\varphi \oplus \psi + h^{\otimes} \geq c \}
$$
with $\varphi \oplus \psi(x,y):=\varphi(x) + \psi(y)$, and $h^{\otimes}(x,y):=h(x)(y-x)$. (The notation $\Dc_{ss}$ is motivated by the fact that, in financial terms, $(\varphi,\psi,h)\in\Dc_{ss}$ corresponds to a \textit{semi-static} hedging strategy, where the static part is $(\varphi,\psi)$ and the dynamic part corresponds to $h$.)

We now introduce candidate optimal quantities. We start with a pair of functions $(T_d^{\mu,\nu},T_u^{\mu,\nu})$ that will support the optimal coupling in the case $\mu\leq_c\nu$ satisfy Assumption \ref{ass:Dispersion_c_appendix}.
 
Set
\begin{equation}\label{eq:g}
g^{\mu,\nu}(x,y):= F_{\nu}^{-1} (F_{\mu}(x) + F_\nu(y)-F_\mu(y)),\quad x\in(\ell_\mu,r_\mu),y\in(\ell_\nu,r_\nu). 
\end{equation}

Define
\begin{equation}\label{eq:T_dT_uDiagonal}
T^{\mu,\nu}_d(x)=T^{\mu,\nu}_u(x)=x,\quad x\in[\overline{m}^{\mu,\nu},r_\mu).
\end{equation}
For $x\in(\ell_\mu,\overline{m}^{\mu,\nu})$, let $T_d^{\mu,\nu}(x)$ be the (unique) scalar such that
\begin{equation}\label{eq:offT_d}
\int_{-\infty}^x [ F_{\nu}^{-1} ( F_{\mu}(\xi)) -\xi ] f_{\mu}(\xi) d \xi + \int_{-\infty}^{T_d(x)} \mathbf{1}_{(-\infty, \overline{m}^{\mu,\nu}]}(\xi) [g^{\mu,\nu}(x, \xi) - \xi] (f_\nu(\xi)-f_\mu(\xi)) d\xi =0
\end{equation}
and set 
\begin{equation}\label{eq:offT_u}
T_u^{\mu,\nu}(x)=g^{\mu,\nu}(x,T_d^{\mu,\nu}(x)),\quad x\in(\ell_\mu,\overline{m}^{\mu,\nu}).
\end{equation}
\begin{Remark}\label{rem:T_dT_uProperties}
	The existence and uniqueness of \eqref{eq:offT_d} is given by Henry-Labord\`ere and Touzi \cite{HLTouzi16}. Furthermore, $T_d(x)\leq x\leq T_u(x)$ for all $x\in(\ell_\mu,r_\mu)$, and, under Assumption \ref{ass:Dispersion_c_appendix}, $T_d,T_u$ are both continuous, $T_d$ (resp. $T_u$) is strictly increasing (resp. decreasing) on $(\ell_\mu,\overline{m}^{\mu,\nu}]$. See Henry-Labord\`ere and Touzi \cite[Section 4]{HLTouzi16} (note, however, that \cite{HLTouzi16} constructed the left-curtain martingale coupling, while in the present paper we are working with the right-curtain coupling.)
\end{Remark}
	
Using the pair $(T^{\mu,\nu}_d,T^{\mu,\nu}_u)$ we now define a candidate martingale coupling through its disintegration w.r.t. the first marginal $\mu$. Define a measure $\P^{\mu,\nu}$ on $\R^2$ by
\begin{equation}\label{eq:right_curtain}
\begin{aligned}
\P^{\mu,\nu}(dx,dy)&=\mu(dx)\mathbf{1}_{[\overline{m}^{\mu,\nu},r_\mu)}(x) \delta_{x} (dy) \\   &+\mu(dx)\mathbf{1}_{(\ell_\mu,\overline{m}^{\mu,\nu})}(x) \left[ q^{\mu,\nu}(x) \delta_{T^{\mu,\nu}_u(x)} (dy) +  (1-q^{\mu,\nu}(x)) \delta_{T^{\mu,\nu}_d(x)} (dy)  \right],
\end{aligned}
\end{equation}
where
\begin{equation}\label{eq:probUP}
q^{\mu,\nu}(x):=\frac{x-T^{\mu,\nu}_d(x)}{T^{\mu,\nu}_u(x)-T^{\mu,\nu}_d(x)},\quad x\in(\ell_\mu,\overline{m}^{\mu,\nu}),
\end{equation}
corresponds to the (conditional) martingale probability of jumping to $T^{\mu,\nu}_u(x)$ from $x$.

We now introduce a candidate triple $(\varphi_*,\psi_*,h_*)=(\varphi^{\mu,\nu}_*,\psi^{\mu,\nu}_*,h^{\mu,\nu}_*)$ for the dual problem. To ease the notation we will write $(T_d,T_u)=(T_d^{\mu,\nu},T_u^{\mu,\nu})$.

Define $h_*$, up to a constant, by
\begin{equation}\label{eq:h_*}
\begin{aligned}
h_*&= h_* \circ T_d^{-1} + c_y(\cdot, \cdot) - c_y(T_d^{-1}, \cdot)\quad \mbox{on } [\overline{m}^{\mu,\nu},r_\mu),\\
h'_*&=\frac{c_x(\cdot, T_u) -c_x(\cdot, T_d) }{T_u-T_d}\quad \mbox{on } (\ell_\mu,\overline{m}^{\mu,\nu}).
\end{aligned}
\end{equation}
(Note that the constant can be chosen such that $h_*$ is continuous.) The (continuous) function $\psi_*$ is then given (also up to a constant) by
\begin{equation}\label{eq:psi_*}
\begin{aligned}
\psi'_* &= c_y(T_d^{-1}, \cdot) - h_* \circ T_d^{-1}\quad \mbox{on } [\overline{m}^{\mu,\nu},r_\mu),\\
\psi'_*&=c_y(T_u^{-1}, \cdot) - h_* \circ T_u^{-1}\quad \mbox{on } (\ell_\mu,\overline{m}^{\mu,\nu}).
\end{aligned}
\end{equation}
Finally the function $\varphi_*$ is defined by
\begin{equation}\label{eq:phi_*}
\begin{aligned}
\varphi_*(x)=\E^{\P^{\mu,\nu}}[c(X,Y)-\psi_*(Y)\lvert X=x],\quad x\in(\ell_\mu,r_\mu).
\end{aligned}
\end{equation}

\begin{Theorem}[Henry-Labord\`ere and Touzi {\cite[Theorems 4.5, 5.1]{HLTouzi16}}]\label{thm:martingaleB}
Suppose $\mu\leq_c\nu$ satisfies Assumption \ref{ass:Dispersion_c_appendix}. Assume further that $\int \varphi_*\vee0 d\mu +\int \psi_*\vee0 d\nu < \infty$, and that the partial derivative $c_{xyy}$ exists and $c_{xyy}<0$ on $\R ^2$. Let $\P_*=\P^{\mu,\nu}$ be as in \eqref{eq:right_curtain} and $(\varphi_*,\psi_*,h_*)$ as in \eqref{eq:phi_*}, \eqref{eq:psi_*}, \eqref{eq:h_*}.

Then $\P_*\in\Mc_2(\mu,\nu)$, $(\varphi_*,\psi_*,h_*)\in\Dc_{ss}$ and the strong duality holds:
$$
\E^{\P_*} [c(X,Y)] = P_{MOT}(\mu, \nu) = D_{MOT}(\mu, \nu)= \mu(\varphi_*) + \nu(\psi_*).
$$
\end{Theorem}

\section{One-period Supermartingale Optimal Transport} \label{sec:Decreasing_SMOT_one_period}

\subsection{Solution to the primal SMOT problem and related properties}\label{sec:Decreasing_SMOT_one_period_primal}

Let $\mu,\nu\in\Pc$ be in convex-decreasing order, i.e., $\mu\leq_{cd}\nu$ (for two measures $\eta,\chi$ on $\R$ we also write $\eta\leq\chi$ if $\eta(A)\leq\chi(A)$ for all Borel $A\subseteq \R$). Let $\Sc_2(\mu, \nu):= \{ \P \in \Pc_2(\mu, \nu): \E^\P[ Y|X ]\leq X \}$ be the set of supermartingale couplings with given marginals $\mu$ and $\nu$. The following theorem defines the decreasing supermartingale coupling $\hat\P=\hat\P^{\mu,\nu}\in \Sc_2(\mu,\nu)$.
\begin{Theorem}[Nutz and Stebegg \cite{NutzStebegg.18}]\label{thm:decreasing_characterization}
	Suppose $\mu\leq_{cd}\nu$ and let $\hat\P\in \Sc_2(\mu,\nu)$ be the decreasing supermartingale coupling. It satisfies any, and then all, of the following properties
\begin{enumerate}
	\item $\hat \P$ solves \eqref{eq:SMOT_onePrimal} whenever $c$ satisfies $c_{xyy}<0$ and $c_{xy}>0$;
	\item For each $x\in\R$, $\hat\P\lvert_{[x,\infty)\times\R}$ is the smallest element (w.r.t. $\leq_{cd}$) of $\{\theta:\mu\lvert_{[x,\infty)}\leq_{cd}\theta\leq\nu\}$;
		\item There exists $\Gamma\subseteq\R^2$ and $M\subset \R$ such that $\hat \P_*$ is second-order right-monotone and first-order left-monotone w.r.t. $(\Gamma,M)$ in the sense of Definition \ref{defn:support_mon}.
\end{enumerate}
\end{Theorem}
\begin{Definition}\label{defn:support_mon}
	Let $(\Gamma, M)\in B(\R^2)\times B(\R)$. We say
	\begin{enumerate}
		\item $\Gamma$ is second-order right-monotone if for all $(x,y_1),(x,y_2),(x',y')\in\Gamma$ with $x'<x$ we have that $y'\notin(y_1,y_2)$;
		\item $(\Gamma,M)$ is first-order left-monotone if for all $(x_1,y_1),(x_2,y_2)\in\Gamma$ with $x_1<x_2$ and $x_2\notin M$ we have that $y_1\leq y_2$.
	\end{enumerate}
\end{Definition}
The second property of $\hat\P$ in Theorem \ref{thm:decreasing_characterization} implies that, for each $x\in\R$, $\hat \P\lvert_{[x,\infty)\times\R}$ is the shadow of $\mu\lvert_{[x,\infty)}$ in $\nu$: given $\mu_0\leq\mu\leq_{cd}\nu$, the \textit{shadow measure}, denoted by $S^\nu(\mu_0)$, is the smallest element (w.r.t. $\leq_{cd}$) of $\{\theta:\mu_0\leq_{cd}\theta\leq\nu\}$.

Furthermore, the set $M$, that appears in the third property of $\hat \P$ in Theorem \ref{thm:decreasing_characterization}, corresponds to the martingale points of $\hat \P$ in the sense that $\hat\P\lvert_{M\times\R}$ is a martingale. The explicit constructions of $\hat\P$ and $M$ were provided by Bayraktar et. al. \cite[Theorem 6.1, Proposition 5.1]{BayDengNorgilas2}. In particular, there exists $(x_n)_{n\geq1}$ in $\R$ such that $M=\bigcup_{n\geq1}[x_{n+1},x_n)$, so that $\hat\P$ alternates between being a martingale and supermartingale at most countably many times. For each $n\geq1$, the transitions of $\hat\P$ on $[x_{n+1},x_n)$ are of the right-curtain type (see \eqref{eq:right_curtain}), while on $\R\setminus M$, $\hat\P$ mimics the transitions of the quantile coupling (see \eqref{eq:quantile}).

For a measure $\eta$ on $\R$ let $C_\eta:\R\to\R$ be given by $C_\eta(k):=\int_\R(x-k)^+\eta(dx)$, $k\in\R$. Define
$c^{\mu,\nu}:\R\to\R$ by
$$
c^{\mu,\nu}(x)=\sup_{k\in\R}\{C_{\mu\lvert_{[x,\infty)}}(k)-C_\nu(k)\},\quad x\in\R.
$$
Then, provided $\mu$ is atom-less, $c$ is non-negative non-decreasing, continuous and 
$$
c(x)=\overline{\mu\lvert_{[x,\infty)}}-\overline{S^\nu(\mu\lvert_{[x,\infty)})}.
$$
See Bayraktar et. al. \cite[Lemmas 3.1, 5.1]{BayDengNorgilas2}. Then the first point (starting from the right, i.e., $r_\mu$, and moving to the left towards $\ell_\mu$) where $\hat\P$ transitions from being a martingale to supermartingale is given by
\begin{equation}\label{eq:x_1Appendix}
x_1^{\mu,\nu}=\inf\{x\in\R:c^{\mu,\nu}(x)=0\}\quad\textrm{with}\quad\inf\emptyset=r_\mu.
\end{equation}
Furthermore, define
\begin{equation}\label{eq:y_1Appendix}
y_1^{\mu,\nu}=\ell_{S^\nu(\mu\lvert_{[x_1^{\mu,\nu},\infty)})}.
\end{equation}
Note that, since $\mu\lvert_{[x_1^{\mu,\nu},\infty)}\leq_{cd}S^\nu(\mu\lvert_{[x_1^{\mu,\nu},\infty)})$, we have that $y_1^{\mu,\nu}\leq x_1^{\mu,\nu}$.

\begin{Lemma}\label{lem:transition_dispersion}
Suppose that $\mu\leq_{cd}\nu$ satisfy Assumption \ref{ass:Dispersion_c_appendix}. Let $\hat\P$ be the decreasing supermartingale coupling of Theorem \ref{thm:decreasing_characterization}.
 
 Then the ($\mu$-a.s. unique) set of martingale points of $\hat\P$ is given by $M=[x_1^{\mu,\nu},r_\mu)$. Furthermore, the second marginal of $\hat\P\lvert_{M\times\R}$ is $\nu\lvert_{[y_1^{\mu,\nu},r_\nu)}$.
\end{Lemma}
\begin{proof}[Proof of Lemma \ref{lem:transition_dispersion}]
	We assume that $\overline\mu>\overline\nu$, since otherwise $\mu\leq_c\nu$, and there is nothing to prove. Indeed, if $\overline\mu=\overline\nu$, then $x_1^{\mu,\nu}=\ell_\mu$ and $\hat\P=\P^{\mu,\nu}$ is the right-curtain martingale coupling; see \eqref{eq:right_curtain}.
	
	Suppose that $x_1^{\mu,\nu}\in(\ell_\mu,r_\mu]$. Let $\overline m:=\overline{m}^{\mu,\nu}$. Note that $\mu\lvert_{[\overline m, r_\mu)}\leq \nu$ and thus $\mu\lvert_{[\overline m, r_\mu)}=S^\nu(\mu\lvert_{[\overline m, r_\mu)})$. It follows that $x_1^{\mu,\nu}\leq \overline m$. In fact, since $\mu$ and $\nu-\mu\lvert_{[\overline m, r_\mu)}$ both have strictly positive densities in the neighbourhood of $\overline m$ and $f_\mu>f_\nu$ to the left of $\overline m$, we must have that $x_1^{\mu,\nu}<\overline m$.
	
	We now argue that $S^\nu(\mu\lvert_{[x_1^{\mu,\nu}, r_\mu)})=\nu\lvert_{[y_1(\mu,\nu),r_\nu)}$. First, using the associativity of the shadow measure (see Bayraktar et al. \cite{BayDengNorgilas}), we have that
	$$
	S^\nu(\mu\lvert_{[x_1^{\mu,\nu}, r_\mu)})=\mu\lvert_{[\overline m, r_\mu)}+S^{\nu-\mu\lvert_{[\overline m, r_\mu)}}(\mu\lvert_{[x_1^{\mu,\nu}, \overline m)}).
	$$
Thanks to the properties of the shadow measure ($S^{\nu-\mu\lvert_{[\overline m, r_\mu)}}(\mu\lvert_{[x_1^{\mu,\nu}, \overline m)})$ has the smallest variance among all possible target laws of $\mu\lvert_{[x_1^{\mu,\nu}, \overline m)}$ within $\nu-\mu\lvert_{[\overline m, r_\mu)}$), we have that $$
S^{\nu-\mu\lvert_{[\overline m, r_\mu)}}(\mu\lvert_{[x_1^{\mu,\nu}, \overline m)})=\nu\lvert_{[y_1(\mu,\nu),\overline m)}+(\nu-\mu\lvert_{[\overline m, r_\mu)})\lvert_{(\overline m,r_0)}
$$ for some $r_0\in(\overline m,r_\nu]$. But if $r_0<r_\nu$, then $\nu-S^\nu(\mu\lvert_{[x_1^{\mu,\nu}, r_\mu)})$ charges $[r_0,r_\nu)$. Then, since $x_1^{\mu,\nu}<\overline m<r_0<r_\nu$, we can find a small enough $x<x_1^{\mu,\nu}$, such that $\mu\lvert_{[x,x_1^{\mu,\nu})}\leq_c S^{\nu-	S^\nu(\mu\lvert_{[x_1^{\mu,\nu}, r_\mu)})}(\mu\lvert_{[x,x_1^{\mu,\nu})})$, contradicting the minimality of $x_1^{\mu,\nu}$. It follows that $r_0=r_\nu$.

It is left to show that $x_1^{\mu,\nu}$ is the ($\mu$-a.s.) unique regime-switching point. For this we use the following observations. The transitions of the decreasing supermartingale coupling $\hat\P$ is either those of the right-curtain (see \eqref{eq:right_curtain}) or the quantile coupling (see \eqref{eq:quantile}); see Nutz and Stebegg \cite{NutzStebegg.18} or Bayraktar et al. \cite{BayDengNorgilas2}. In particular, by the definition of $x_1^{\mu,\nu}$, we must have that, (locally) to the left of $x_1^{\mu,\nu}$, $\hat\P$ corresponds to the quantile coupling of $\mu\lvert_{(-\infty,x_1^{\mu,\nu})}$ and $\nu\lvert_{(-\infty,y_1^{\mu,\nu})}$, denoted by $\pi^q$ (which is just a restriction of the quantile coupling of $\mu$ and $\nu$ to $(-\infty,x_1^{\mu,\nu})\times(-\infty,y_1^{\mu,\nu})$).

Under $\pi^q$, each $x\in(\ell_\mu,x_1^{\mu,\nu})$ is mapped to $F^{-1}_\nu(F_\mu(x))$. Note that $x\mapsto F^{-1}_\nu(F_\mu(x))$ is non-decreasing, continuous, and by construction (and the supermartingale requirement) we have that $F^{-1}_\nu(F_\mu(x_1^{\mu,\nu}))=y_1^{\mu,\nu}\leq x_1^{\mu,\nu}$. Let $\bar x$ be defined by $F_\mu(\bar x)=F_\nu(\bar x)$. Due to Assumption \ref{ass:Dispersion_c_appendix}, $\bar x$ is unique and satisfies $\bar x\in(\underline m, \overline m)$. In particular, $F_\nu>F_\mu$ on $(-\infty,\bar x)$, while $F_\nu<F_\mu$ on $(\bar x,\infty)$.

If $x_1^{\mu,\nu}\leq\bar x$, then $F^{-1}_\nu(F_\mu(x))<x$ for all $x<x_1^{\mu,\nu}$, and the quantile coupling provides strict supermartingale transitions on $(-\infty,x_1^{\mu,\nu})$. Then using characterization of $\hat P_*$ in terms of the shadow measure, it follows that $\hat\P=\pi^q$ on $(-\infty,x_1^{\mu,\nu})$, which proves our claim. On the other hand, $x_1^{\mu,\nu}>\bar x$ cannot happen, since then $y_1^{\mu,\nu}=F^{-1}_\nu(F_\mu(x_1^{\mu,\nu}))> x_1^{\mu,\nu}$, a contradiction. 
\end{proof}

Finally, provided that $\mu$ is atomless and using the construction of Bayraktar et al. \cite{BayDengNorgilas2}, we have that the decreasing supermartingale coupling is supported on the graph of two functions $\hat T_d=\hat T_d^{\mu,\nu},\hat T_u=\hat T_u^{\mu,\nu}:(\ell_\mu,r_\mu)\to(\ell_\nu,r_\nu)$. If, in addition, the Dispersion Assumption holds, then using Lemma \ref{lem:transition_dispersion} and utilising the construction and properties of the right-curtain martingale coupling $P^{\mu,\nu}$ (as in \eqref{eq:right_curtain}) we have that $(\hat T_d,\hat T_u)$ satisfies some useful properties.

First,
\begin{equation}\label{eq:decreasing_Diagonal}
\hat T_d(x)\leq x\leq \hat T_u(x)\textrm{ for all }x\in(\ell_\mu,r_\mu), \quad\textrm{and}\quad \hat T_d=\hat T_u\textrm{ on }[\bar m,r_\mu).
\end{equation}
Furthermore,
\begin{equation}\label{eq:hatT_d}
\begin{aligned}
&\hat T_d\textrm{ is continuous and strictly increasing on }(\ell_\mu,r_\mu),\\ &\hat T_d(x) = F^{-1}_\nu\circ F_\mu \textrm{ on } (\ell_\mu,x_1^{\mu,\nu})
\end{aligned}
\end{equation}
and
\begin{equation}\label{eq:hatT_u}
\begin{aligned}
&\hat T_u\textrm{ is continuous on }(x_1^{\mu,\nu},r_\mu), \textrm{ strictly decreasing on } (x_1^{\mu,\nu},\overline m),\\
&\hat T_u\equiv \infty \textrm{ on } (\ell_\mu,x_1^{\mu,\nu}].
\end{aligned}
\end{equation}

In particular, in terms of $(\hat T_d,\hat T_u)$, the decreasing supermartingale coupling $\P^{\mu,\nu}$ is given by
\begin{equation}\label{eq:decreasing_representation}
\begin{aligned}
\hat \P(dx,dy)&=\mu(dx)\mathbf{1}_{(\ell_\mu,x_1^{\mu,\nu}]\cup[\overline m, r_\mu)}(x)\delta_{\hat T_d(x)}(dy)\\
&+\mu(dx)\mathbf{1}_{(x_1^{\mu,\nu},\overline m)}\left\{\frac{\hat T_u(x)-x}{\hat T_u(x)-\hat T_d(x)}\delta_{\hat T_d(x)}(dy)+\frac{x-\hat T_d(x)}{\hat T_u(x)-\hat T_d(x)}\delta_{\hat T_u(x)}(dy)\right\}.
\end{aligned}\end{equation}
Equivalently, we can express $\hat \P$ in terms of the right-curtain coupling $\P^{\mu\lvert_{[x_1^{\mu,\nu},r_\mu)},\nu\lvert_{[y_1^{\mu,\nu}r_\nu)}}$ of $\mu\lvert_{[x_1^{\mu,\nu},r_\mu)}$ and $\nu\lvert_{[y_1^{\mu,\nu}r_\nu)}$ (see \eqref{eq:right_curtain}):
\begin{equation}\label{eq:decreasing_representation2}
\begin{aligned}
\hat \P(dx,dy)&=\mu(dx)\mathbf{1}_{(\ell_\mu,x_1^{\mu,\nu}]}(x)\delta_{F^{-1}_\nu(F_\mu(x))}(dy)+\mathbf{1}_{(x^{\mu,\nu}_1,r_\mu)}\P^{\mu\lvert_{[x_1^{\mu,\nu},r_\mu)},\nu\lvert_{[y_1^{\mu,\nu}r_\nu)}}(dx,dy).
\end{aligned}\end{equation}

While the pair $(\hat T_d, \hat T_u)$ can be explicitly constructed using potential functions of the underlying measures (see Bayraktar et al. \cite{BayDengNorgilas2}), the following representation is more useful in the present context.

Recall the definition of $g=g_{\mu,\nu}$ (see \eqref{eq:g}).

\begin{Proposition} \label{Prop:one_period_transport_map} For each $x\in[x_1^{\mu,\nu},\overline m^{\mu,\nu})$, $\hat T_u(x)$ is uniquely characterised by the integral equation:
$$
\int_{x}^{\infty} [F^{-1}_\nu (F_\mu( \xi)) -\xi ] f_\mu(\xi) d \xi + \int_{\hat T_u(x)}^{\infty} \mathbf{1}_{(x_1^{\mu,\nu}, \infty)}(\xi)  [g (x, \xi) - \xi] (f_\nu(\xi)-f_\mu(\xi)) d \xi=0,
$$
and $\hat T_d(x)=g(x, \hat T_u(x))$.
\end{Proposition}
\begin{proof}
	Since $\mu\lvert_{[x_1^{\mu,\nu},r_\mu)}\leq_c\nu\lvert_{[y_1(\mu,\nu),r_\nu)}$ and the decreasing supermartingale coupling $\hat \P^{\mu,\nu}$ of $\mu\leq_{cd}\nu$ coincides with the right-curtain coupling of $\mu\lvert_{[x_1^{\mu,\nu},r_\mu)}\leq_c\nu\lvert_{[y_1(\mu,\nu),r_\nu)}$ on $[x_1^{\mu,\nu},r_\mu)$, the result follows immediately from the construction of the latter, which is given in Section \ref{sec:Brenier_MOT} (see \eqref{eq:offT_d} and \eqref{eq:offT_u}).
\end{proof}
\subsection{Optimal dual strategy for decreasing SMOT}\label{sec:SMOTdual}

We introduce now the optimal dual strategies to the dual problem \eqref{eq:Dual_Full} in the case the decreasing supermartingale coupling $\hat \P=\hat\P^{\mu,\nu}$ (see Theorem \ref{thm:decreasing_characterization}) is optimal. Recall that the pair $(\hat T_d,\hat T_u)$ represents the supporting functions of $\hat \P$, while $x_1=x_1^{\mu,\nu}$ is the unique threshold that separates martingale and supermartingale regions.

In the following we suppose that $\mu\leq_{cd}\nu$ satisfy Assumption \ref{ass:Dispersion_c_appendix}.

We introduce the following triple $(\hat{\varphi}, \hat{\psi}, \hat{h})=(\hat{\varphi}(\mu,\nu), \hat{\psi}(\mu,\nu), \hat{h}(\mu,\nu))$, as the candidate optimal dual strategies for $\mathbf{D}_2(\mu, \nu)$ (with $\mathbf{D}_2$ defined in \eqref{eq:SMOT_oneDual}):
 \begin{equation}\begin{aligned}
&\hat h'(x):= \frac{c_x(x, \hat T_u (x))- c_x(x, \hat T_d(x))}{\hat T_u(x)-\hat T_d(x) }, \quad x_1< x < \overline m;  \\
&\lim_{x\downarrow x_1}\hat h(x)=0;\\
& \hat h(x):= \hat h((\hat T_u)^{-1}(x)) - c_y( (\hat T_u^{-1}(x), x ),~~x \geq \overline m; \\
& \hat h(x): = 0,~~x \leq x_1.
\end{aligned}
 \end{equation}
 Note that, when $x=x_1$, $\hat T_u(x)=\infty$ and therefore $\hat{h}'(x)=0$. Furthermore, $\hat h$ is continuous at $x\in\{x_1,\overline m\}$ (note that $c_y(x,x)=0$ and $\lim_{x \to \overline m} (\hat T_u)^{-1}(x) = \overline m$).
 
  Denoting {$L^{-1}(x) := (\hat T_u)^{-1}(x) \mathbf{1}_{\{x \geq \overline m\}}+(\hat T_d)^{-1}(x) \mathbf{1}_{\{x<\overline m\}}$,} we introduce $\hat \psi$, up to a constant, by
$$
\hat \psi' (x): = c_y(L^{-1}(x), x) - \hat h( L^{-1}(x)).
$$
Note that $\hat \psi$ can be chosen to be continuous, which then also makes 
$$
c(\cdot, \hat T_u(\cdot) ) - \hat \psi (\hat T_u(\cdot)) - c(\cdot, \hat T_d(\cdot) ) + \hat \psi (\hat T_d(\cdot))- (\hat T_u(\cdot) - \hat T_d(\cdot) ) \hat h(\cdot)
$$
continuous.

Finally, we define
$$
\hat \varphi(x):=\E^{\hat\P}[c(X,Y)-\psi(Y)\lvert X=x],\quad x\in\R.
$$
Note that, for $x_1 \leq x \leq \overline m$
$$
\hat \varphi(x) =\frac{x-\hat T_d(x)}{\hat T_u (x)-\hat T_d  (x)}\left( c(x,\hat T_u(x)) - \hat \psi ( \hat T_u(x)) \right) + \frac{\hat T_u (x) - x}{\hat T_u (x)-\hat T_d(x)}\left( c(x,\hat T_d(x)) - \hat \psi (\hat T_d(x)) \right),
$$ 
$$
\hat \varphi(x) := c(x,x) - \hat \psi( x) =  -\hat  \psi (x),\quad x > \overline m
$$
and
$$
\hat \varphi(x) := c(x,\hat T_d(x)) - \hat \psi ( \hat T_d(x)),\quad x < x_1.
$$

From our definitions, we immediately have that $\E^{\hat\P}[c(X,Y)]=\mu(\hat \varphi(X))+\nu(\hat\psi(Y))$. It is left to show that the triple $(\hat \varphi,\hat\psi,\hat h)$ defines a superhedge.   
\begin{Lemma} \label{Lemma:OptimalDual}
Let $\mu\leq_{cd}\nu$ be such that Assumption \ref{ass:Dispersion_c_appendix} holds. Suppose that the reward function $c:\R^2 \to \R$ satisfy Assumption \ref{Assump:CostFunction}. Assume further that $\int \hat\varphi\vee 0 d\mu +\int \psi\vee0 d\nu < \infty$. 

We have that $(\hat \varphi, \hat \psi, \hat h) \in \Dc_2(\mu, \nu)$.
\end{Lemma}

\begin{proof}
To obtain the proof we need to show that $\hat h\geq0$ and that $c(x,y)\leq \hat \varphi(x)+\hat\psi(y)+\hat h (x)(y-x)$ for all $x,y\in\R$. The second point, however, follows immediately from the arguments of the classical and martingale versions of Brenier's theorem (see Theorems \ref{thm:classicalB} and \ref{thm:martingaleB}). We now verify that $h\geq 0$; we will use that $c(\cdot, \cdot)= c_{y}(\cdot, \cdot)=0$ and $c_{xy}>0$. 

First, as $c_{xy} > 0$, we have that when $x_1 \leq x < \overline m$,
$$
\hat h' (x):= \frac{c_x(x, \hat T_u(x))- c_x(x, \hat T_d(x))}{\hat T_u(x)-\hat T_d(x) } > 0.
$$
Hence, since $\hat h(x_1)=0$, we have that $\hat h(x) \geq 0$ for all $x_1\leq x < \overline m$.

For $x \geq \overline m$, as $x_1 \leq (\hat T_u)^{-1}(x) \leq \overline m$, we have that $\hat h((\hat T_u)^{-1}(x)) \geq 0$. In addition, as $c_{y}(x,x)=0$ and $c_{xy}>0$, we have that $c_y( (\hat T_u)^{-1}(x), x ) \leq 0$. It follows that 
$$
\hat h(x) = \hat h((\hat T_u)^{-1}(x)) - c_y( (\hat T_u)^{-1}(x), x ) \geq 0.
$$
Since $\hat h\equiv0$ on $(-\infty,x_1)$, the no short-selling constraint is satisfied.
\end{proof}

\nocite{*}

\bibliography{references}


\end{document}